\documentclass{article}

\usepackage{amsmath,amssymb}
\usepackage{graphicx}
\usepackage{bm}
\usepackage{theorem}


\theorembodyfont{\upshape} 
\theorembodyfont{\rmfamily}

\newtheorem{thm}{Theorem}[section]
\newtheorem{lem}[thm]{\bfseries Lemma}        %
\newtheorem{remark}[thm]{\bfseries Remark}    %
\newtheorem{prop}[thm]{\bfseries Proposition} %
\newtheorem{cor}[thm]{\bfseries Corollary}     
\newtheorem{defn}[thm]{\bfseries Definition}

\newtheorem{conj}[thm]{\bfseries Conjecture}

\newenvironment{proof}{\medskip                    %
\noindent{\scshape Proof:}}{\quad $\Box$\medskip}  %

\def\rank{\mathrm{rank}}
\def\id{\mathrm{id}}
\def\Fix{\mathrm{Fix}}

\begin{document}

\title{On symmetry groups of oriented matroids 
}
\author{Hiroyuki Miyata\\
 Department of Computer Science, Gunma University \\
 Kiryu, Gunma, 376-8515, Japan \\
 {\tt hmiyata@cs.gunma-u.ac.jp}}

\maketitle

\begin{abstract}
Symmetries of geometric structures such as hyperplane arrangements, point configurations and polytopes
have been studied extensively for a long time.
However, symmetries of oriented matroids, a common combinatorial abstraction of them, are not understood well.

In this paper, we aim to obtain a better understanding of symmetries of oriented matroids.
First, we put focus on symmetries of matroids, and give a general construction that generates a $3$-dimensional point configuration with 
a matroidal symmetry that cannot be realized as a geometric symmetry.
The construction is based on the observation that every non-trivial rotation in the $2$-dimensional Euclidean space has a unique
fixed point but that there is no corresponding property for matroids.
The construction suggests that the lack of the fixed point theorem generates a big gap between matroidal symmetries and geometric symmetries 
of point configurations.
Motivated by this insight, we study fixed-point properties for symmetry groups of oriented matroids.

For rotational symmetries of oriented matroids, we prove a useful property, which corresponds to the uniqueness of fixed points (in the rank $3$ case).
Using it, we classify rotational and full symmetry groups of simple oriented matroids of rank $3$.
In addition, we define fixed-point-admitting (FPA) property for subgroups of symmetry groups of oriented matroids, and
make classification of rotational symmetry groups with FPA property of simple acyclic oriented matroids of rank $4$.
We conjecture that the symmetry group of every acyclic simple oriented matroid has FPA property.
\end{abstract}

\section{Introduction}
Symmetries of geometric structures such as polytopes, hyperplane arrangements and point configurations 
have been paid big interests for a long time. They have been studied extensively and rich theories concerning their symmetries have been developed (see \cite{H90,S94}, for example).
Those geometric structures play important roles in computer science (e.g. combinatorial optimization, computational geometry, etc.), and 
combinatorial aspects (e.g. number of faces, underlying graphs, etc.)
are especially important in this context.
Some combinatorial aspects of those geometric structures have been abstracted into simple axiom systems, which have led to rich theories such as matroid theory~\cite{O11} and oriented matroid theory~\cite{BLSWZ99}.
These combinatorial structures capture some combinatorial behaviors of those geometric structures very precisely.
For example, in the framework of oriented matroids, the upper bound theorem of polytopes can be proved~\cite{M82}, and  many aspects of  linear programming theory
can also be discussed (see \cite[Chapter 10]{BLSWZ99}).
 
One of natural questions concerning matroids and oriented matroids would be whether their symmetries
admit similar properties to symmetries of concrete geometric structures they abstract.
However, there is not so much work on their symmetries, and are not understood well.
The aim of this paper is to contribute to a better understanding of symmetries of matroids and oriented matroids.

There are some related work, which studies gaps between {\it geometric symmetries} and {\it combinatorial symmetries}.
Geometric symmetries of point configurations are symmetries induced by affine automorphisms.
By combinatorial symmetries, we mean symmetries of underlying combinatorial structures such as the face lattices of polytopes, and
the associated matroids or oriented matroids of the configurations.
Every geometric symmetry induces a combinatorial symmetry, but the converse is not always true.
There are considerable studies on this topic since it is a fundamental question when a combinatorial symmetry cannot be realized geometrically, 
in order to understand symmetries of underlying combinatorial structures.

Here, we review studies on gaps between combinatorial symmetries and geometric symmetries of related structures.
In 1971, Mani~\cite{M71} proved that every combinatorial symmetry (face-lattice symmetry) can be realized geometrically for $3$-polytopes.
Then Perles showed that the same holds for $d$-polytopes with $d+3$ vertices~\cite[p.120]{G03}.
After that, it had been a big question whether every combinatorial symmetry (face-lattice symmetry) can be realized geometrically for every polytope.
Bokowski, Ewald and Kleinschmidt~\cite{BEK84} resolved the question by presenting a $4$-polytope with $10$ vertices
with a non-realizable combinatorial symmetry (face-lattice symmetry).
For point configurations, Shor~\cite{S91} constructed an example of a $2$-dimensional configuration of $64$ points with
a non-realizable combinatorial symmetry (oriented-matroid symmetry) and
Richter-Gebert~\cite{R96_2} presented an example of a $2$-dimensional configuration of $14$ points with a non-realizable combinatorial symmetry
(oriented-matroid symmetry).
In 2006, Paffenholz~\cite{P06} constructed a two-parameter infinite family of $4$-polytopes with non-realizable combinatorial symmetries (face-lattice symmetries).
\subsection*{Contribution of the paper}
First, we study a gap between combinatorial symmetries (matroidal symmetries) and geometric symmetries of point configurations.
We consider a general method to construct $3$-dimensioinal point configurations with non-realizable matroidal symmetries (Theorem \ref{thm:general_construction}).
Similarly to the constructions of polytopes with non-realizable combinatorial symmetries
by Bokowski, Ewald and Kleinschmidt~\cite{BEK84} and Paffenholz~\cite{P06},
a key tool for our construction is the fixed point theorem, which asserts that
every non-trivial rotation in the $2$-dimensional Euclidean space has a unique fixed point.
There is no corresponding property for matroids and it makes the gaps between geometric symmetries and matroidal symmetries.
Motivated by this insight, we propose to study fixed-point properties of symmetry groups of oriented matroids.
We prove a useful property for rotational symmetries of oriented matroids, which corresponds to the uniqueness of fixed points in the rank $3$ case (Theorems \ref{thm:uniqueness_general}). 
Based on it, we prove that rotational symmetry groups of simple oriented matroids of rank $3$ are classified into
the cyclic groups $\mathbb{Z}_n$ ($n \geq 1$) of order $n$, the dihedral groups $D_{2n}$ ($n \geq 1$) of order $2n$\footnote{In the literature, the dihedral group of order $2n$ is also written as $D_n$, but we write $D_{2n}$ in this paper.}, 
the alternating group $A_4$ (Theorem \ref{thm:classification_rotation_rank3}), 
and that full symmetry groups are classified into $\mathbb{Z}_n$ ($n \geq 1$),  $D_{2n}$ ($n \geq 1$) and the symmetric group $S_4$ (Theorem \ref{thm:classification_full_rank3}). 
Furthermore, we define {\it fixed-point-admitting (FPA) property} for subgroups of symmetry groups of oriented matroids and
study FPA rotational symmetry groups of simple acyclic oriented matroids of rank $4$.
In particular, it is proved that FPA rotational symmetry groups of acyclic simple oriented matroids of rank $4$
are classified into $\mathbb{Z}_n$ ($n \geq 1$), $D_{2n}$ ($n \geq 1$), $A_4$, $A_5$ and $S_4$ (Theorem \ref{main_thm}). 
This result completely coincides with the classification of (geometric) rotational symmetry groups of $3$-dimensional point configurations.
\\
\\
{\bf Organization of the paper}

In Section 2, we explain some terminologies on matroids and oriented matroids.
Section 3 is devoted to studying symmetries of matroids.
We give a general construction of $3$-dimensional point configuratoins with non-realizable matroidal symmetries.
Motivated by insights in Section 3, we, in Section 4, study fixed-point properties of symmetry groups of oriented matroids.
In Sections 5 and 6, symmetry groups of simple oriented matroids of rank $2$ and those of rank $3$ are investigated respectively.
Based on results in the previous sections, we classify FPA rotational symmetry groups of acyclic simple oriented matroids of rank $4$ in Section 7.
Finally, we make a conclusion of the paper in Section 8.

\section{Preliminaries}
In this paper, we assume that the reader is familiar with matroids and oriented matroids.
Here, we recall basics, which will be used in this paper.
For details on matroids and oriented matroids, see~\cite{BLSWZ99,O11}.
Throughout the paper, we use the notation $[n]$ to denote the set $\{ 1,2,\dots,n\}$ for $n \in \mathbb{N}$.
\subsection{Definitions on matroids}
We start with some definitions on matroids.
There are many structures by which we can specify a matroid, such as independent sets, bases, a rank function and flats. 
\begin{defn}(Independent sets)\\
For a finite set $E$ and a collection ${\cal I} \subseteq 2^E$ satisfying the following axioms,
the pair $(E,{\cal I})$ is called a {\it matroid} on the {\it ground set} $E$ with the {\it independent sets} ${\cal I}$.
\begin{itemize}
\item[(I1)] $\emptyset \in {\cal I}$.
\item[(I2)] If $A \in {\cal I}$, then $B \in {\cal I}$ for any $B \subseteq A$.
\item[(I3)] If $A,B \in {\cal I}$ and $|A| > |B|$, then there exists $a \in A$ such that $B \cup \{ a \} \in {\cal I}$.
\end{itemize}
\end{defn}
For a matroid $M=(E,{\cal I})$, let $\rank_{M}(\cdot ):2^E \rightarrow \mathbb{Z}$ be the map such that
\[ \rank_{M}(F):= \max\{|F'| \mid F' \subseteq F, F' \in {\cal I} \}. \]
The map $\rank_{M}$ is called the {\it rank function} of $M$. The {\it rank} of $M$ is defiend as $\rank_{M}(E)$ and is denoted by $\rank(M)$.
Let ${\cal B}:= \max\{ B \mid B \in {\cal I} \}$ (with respect to inclusion).
An element of ${\cal B}$ is called a {\it basis} of $M$. It is known that $|B|=\rank(M)$ for any $B \in {\cal B}$.
The set of bases is actually enough to specify a matroid (see \cite{O11}).
The function $\delta_M:E^r \rightarrow \{ 1, 0 \}$, where $r:= \rank(M)$, defined by
\[ \delta_M (b) = 
\begin{cases}
1 & \text{if $b$ is a basis of $M$,}\\
0 & \text{otherwise}
\end{cases}
\]
is called the {\it characteristic function} of $M$. In the following, we specify matroids by the pairs of their ground sets and characteristic functions.

Matroids arise naturally from vector configurations and point configurations.
Let $E$ be a finite set and $V=({\bm v_e})_{e \in E} \in \mathbb{R}^{d \times |E|}$ a $d$-dimensional vector configuration.
The {\it associated matroid} of $V$ is defined as $M_V=(E,\delta_V)$, where 
\[ \delta_V(f_1,\dots,f_d):= 
\begin{cases}
1 & \text{if $\det({\bm v_{f_1}},\dots,{\bm v_{f_d}}) \neq 0$,}\\
0 & \text{otherwise} 
\end{cases}
\]
for $f_1,\dots,f_d \in E$.
If a matroid can be represented as the associated matroid of some vector configuration, it is said to be {\it realizable}.

Let $P=({\bm p_e})_{e \in E} \in \mathbb{R}^{d \times |E|}$ be a $d$-dimensional point configuration.
Then, the {\it associated vector configuration} $V_P:=({\bm v_e})_{e \in E} \in \mathbb{R}^{(d+1) \times |E|}$ is such that
\[ {\bm v_e}:= \begin{pmatrix} {\bm p_e} \\ 1\end{pmatrix}\]
for each $e \in E$.
The {\it associated matroid} $M_P$ of $P$ is defined as $M_{V_P}$.

To understand combinatorial structures of point configurations, affine subspaces spanned by some of the points play fundamental roles.
They are abstracted by the notion of flats.
\begin{defn}(Flats of matroids) \\
Let $M$ be a matroid on a ground set $E$.
A subset $F \subseteq E$ is called a {\it flat} of $M$ 
if $\rank_{M}(F) < \rank_{M}(F \cup \{ e \})$ for all $e \in E \setminus F$.
For $S \subseteq E$, we denote by ${\rm span}_{M}(S)$ the minimal flat of $M$ that contains $S$.
\end{defn}
For the associated matroid $M_P$ and $S \subseteq E$, ${\rm span}_{M_P}(S)$ is the set of points of $P$ lying on the affine hull of $S$.

\subsection{Definitions on oriented matroids}
\subsubsection{Axiom systems}
Oriented matroids also have various equivalent axiom systems. Let us first see the {\it chirotope axioms}.
\begin{defn}(Chirotope axioms)\\
Let $E$ be a finite set and $r\geq 1$  an integer. A {\it chirotope of rank $r$ on $E$} is a map 
$\chi : E^{r} \rightarrow \{ +,-,0\}$ that satisfies the following properties
for any $i_{1},\dots,i_{r},j_{1},\dots,j_{r} \in E$. 
\begin{center}
\begin{itemize}
\item[($B1$)] 
$\chi$ is not identically zero.
\item[($B2$)] 
$\chi (i_{\sigma (1)},\dots,i_{\sigma (r)}) = {\rm sgn} (\sigma) \chi (i_{1},\dots,i_{r})$
for all $i_1,\dots,i_r \in E$ and any permutation $\sigma$ on $[r]$.
\item[($B3$)] 
For all $i_1,\dots,i_r,j_1,\dots,j_r \in E$, we have
\begin{align*}
\{ \chi (i_1,\dots,i_r) \cdot \chi (j_1,\dots,j_r)\} \cup \{ \chi (j_s,i_2,\dots,i_r) \cdot \chi (j_1,\dots,j_{s-1},i_1,j_{s+1},\dots,j_r) \mid s = 1,\dots,r \} \\
\supseteq \{ +,-\} \text{ or } =\{ 0 \}.
\end{align*}
\end{itemize}
\end{center}
\label{chirotope_axioms}
\end{defn}
A pair $(E,\{ \chi, -\chi \})$ is called an {\it oriented matroid} of rank $r$ on a ground set $E$.
From $\chi$, we define the map $\delta_{\chi}:E^r \rightarrow \{ 1,0\}$ such that 
\begin{align*}
\delta_{\chi} (\lambda) = 
\begin{cases}
1 & \text{if $\chi(\lambda) \neq 0$,} \\
0 & \text{otherwise.}
\end{cases} 
\end{align*}
The pair $(E,\delta_{\chi})$ is called the {\it underlying matroid} of ${\cal M}$ and is denoted by $\underline{\cal M}$.
The rank function $\rank_{\cal M}(\cdot ):2^E \rightarrow \mathbb{Z}$ of ${\cal M}$ is defined by  $\rank_{\underline{\cal M}}(\cdot )$. 
A subset $F \subseteq E$ is a {\it flat} of ${\cal M}$ if it is a flat of $\underline{\cal M}$.
For $A \subseteq E$, we denote by ${\rm span}_{\cal M}(A)$ the minimal flat of ${\cal M}$ that contains $A$. 

Oriented matroids also naturally arise from vector configurations and point configurations.
For a finite set $E$ and a $d$-dimensional vector configuration $V=({\bm v_e})_{e \in E}$,
define a map $\chi_V: E^d \rightarrow \{ 0,+,-\}$ by
\[ \chi_V(i_1,\dots,i_d):={\rm sign}(\det({\bm v_{i_1}},\dots,{\bm v_{i_d}})) \text{ for $i_1,\dots,i_d \in E$.}\]
The oriented matroid $(E,\{ \chi_V,-\chi_V\})$ is called the {\it associated oriented matroid} of $V$ and is denoted by ${\cal M}_V$.
For a $d$-dimensional point configuration $P=( {\bm p_e})_{e \in E}$,
the associated oriented matroid ${\cal M}_P$ is given by the rank $d+1$ oriented matroid ${\cal M}_{V_P}$, where $V_P$ is the associated vector configuration of $P$.
It is sometimes useful to consider oriented matroids arising from {\it signed point configurations}.
A signed point configuration is a triple $S=(P=( {\bm p_e})_{e \in E},W,B)$, where $P$ is a point configuration and $(W, B)$ is a partition of $E$.
A point indexed by an element of $W$ (resp.\ $B$) is called a {\it positive point} (resp.\ {\it negative point}).
The associated vector configuration $V_S=({\bm v}^{S}_e)_{e \in E}$ of $S$ is such that
\begin{align*}
{\bm v}^{S}_e :=
\begin{cases}
{\bm v}_e & \text{for $e \in W$,} \\
-{\bm v}_e & \text{for $e \in B$,}
\end{cases}
\end{align*}
where ${\bm v}_e$, for each $e \in E$, is the associated vector of ${\bm p_e}$.
The associated oriented matroid ${\cal M}_S$ of $S$ is defined by the associated oriented matroid ${\cal M}_{V_S}$. 

An oriented matroid can also be specified by a collection of {\it covectors}.
In the next axiom system, we will use the following notation.
For sign vectors $X,Y \in \{ 0,+,-\}^E$, a sign vector $X \circ Y \in \{ 0,+,-\}^E$ is defined as follows.
\[ 
(X \circ Y)(e) :=
\begin{cases}
X(e) & \text{if $X(e) \neq 0$,}\\
Y(e) & \text{otherwise.}
\end{cases}
\]
We also use the following notations.
\[ S(X,Y):= \{ e \in E \mid X(e) = -Y(e) (\neq 0) \}.\]
\[ X \succeq Y \Leftrightarrow \text{$X(e)=Y(e)$ for all $e \in E$ such that $Y(e) \neq 0$.}\]
Another partial ordering $\geq$ of sign vectors will also be used later.
Let us first consider the ordering $- < 0 < +$.
This induces the partial ordering $\geq$ on the covectors ${\cal V}^*$ of an oriented matroid ${\cal M}$ as follows.
For $X,Y \in {\cal V}^*$,
\[ X \geq Y \Leftrightarrow X(e) = Y(e) \text{ or } X(e) > Y(e), \text{ for all $e \in E$.}\]

\begin{defn}(Covector axioms)\\
An element of a set ${\cal V}^* \subseteq \{ 0,+,-\}^E$ of sign vectors satisfying the following axioms
is called a {\it covector}. 
\begin{itemize}
\item[($V0$)] $0 \in {\cal V}^*$.
\item[($V1$)] $X \in {\cal V}^*$ implies $-X \in {\cal V}^*$.
\item[($V2$)] For $X,Y \in {\cal V}^*$, $X \circ Y \in {\cal V}^*$.
\item[($V3$)] For $X,Y \in {\cal V}^*$ and $e_0 \in S(X,Y)$, there exists a covector $Z \in {\cal V}^*$ such that
$Z(e_0)=0$ and $Z(e)=(X \circ Y)(e)$ for all $e \in E \setminus S(X,Y)$. \begin{flushright}(vector elimination)\end{flushright}
\end{itemize}
\label{defn:covector}
\end{defn}
Axiom ($V3$) can be replaced by the following axiom.
\begin{quote}
\begin{itemize}
\item[($V3^f$)] For $X,Y \in {\cal V}^*$ and $U \subseteq S(X,Y)$, there exists a covector $Z \in {\cal V}^*$ and $u \in U$
such that
$Z(u)=0$,
$Y|_U \succeq Z|_U$
and  $Z(e)=(X \circ Y)(e)$ for all $e \in E \setminus S(X,Y)$.
\end{itemize}
\end{quote}
This operation is called {\it conformal elimination}.

It suffices to consider minimal covectors to specify an oriented matroid.
An element of the set ${\cal C}^*$ defined as follows is called a {\it cocircuit}.
\[ {\cal C}^*:= \{ X \in {\cal V}^* \mid  X \not\succ Y \text{ for all $Y \in {\cal V}^* \setminus \{ 0\}$} \}.\]
The set ${\cal C}^*$ is characterized by a simple axiom system, called {\it cocircuit axioms}.
From a chirotope $\chi$ of rank $r$ on a ground set $E$, the cocircuits ${\cal C}^*$ are reconstructed as follows:
\begin{align*}
{\cal C}^* = \{ (\chi(\lambda, e))_{e \in E}  \mid \lambda \in E^{r-1}\}.
\end{align*}

Let us now see the dual notions of covectors and cocircuits.
They are defined through {\it orthogonality}.
Sign vectors $X$ and $Y$ are {\it orthogonal}
if $\{ X(e) \cdot Y(e) \mid e \in E \} = \{ 0\}$ or $\{ X(e) \cdot Y(e) \mid e \in E \} \supseteq \{ +,-\}$
(multiplication of signs is defined analogously to that of numbers).
We write $X \perp Y$ if $X$ and $Y$ are orthogonal.
The sets ${\cal V}$ and ${\cal C}$ are defined as follows.
\begin{align*}
{\cal V} &:= \{ X \in \{ +,-,0\}^E \mid X \perp Y \text{ for all $Y \in {\cal V}^*$}\},\\
{\cal C} &:= \{ X \in {\cal V} \mid  X \not\succ Y \text{ for all $Y \in {\cal V} \setminus \{ 0\}$} \}.
\end{align*}
An element of ${\cal V}$ (resp.\ ${\cal C}$) is called a {\it vector} (resp.\ {\it circuit}) of ${\cal M}$.
Actually, the sets ${\cal V}$ and ${\cal C}$ satisfy the covector axioms and the cocircuit axioms respectively.
The dual oriented matroid ${\cal M}^*$ of ${\cal M}$ is the oriented matroid with the set ${\cal V}$ of covectors.
The following are useful relations between chirotopes and circuits and those between chirotopes and cocircuits.
For $i_1,\dots,i_{r-1},e,f \in E$ such that $\chi (i_1,\dots,i_{r-1},e) \neq 0$ and $\chi (i_1,\dots,i_{r-1},f) \neq 0$
 and $C \in {\cal C}$ with $E \setminus C^0 = \{ i_1,\dots,i_{r-1},e,f\}$,
we have
\[ \chi (e, i_1,\dots, i_{r-1}) = -C(e)C(f)\chi (f, i_1,\dots, i_{r-1}).\]
For $i_1,\dots,i_{r-1},e,f \in E$ such that $\chi (i_1,\dots,i_{r-1},e) \neq 0$ and $\chi (i_1,\dots,i_{r-1},f) \neq 0$
 and $D \in {\cal C}^*$ with $D^0 = \{ i_1,\dots,i_{r-1}\}$,
we have
\[ \chi (e, i_1,\dots, i_{r-1}) = D(e)D(f)\chi (f, i_1,\dots, i_{r-1}).\]

\subsubsection{Basic operations}
Let ${\cal M}=(E,{\cal V}^*)$ be an oriented matroid with the set of covectors ${\cal V}^*$ and $A \subseteq E$.
The {\it deletion} ${\cal M}\setminus A$ of ${\cal M}$ by $A$ is 
the oriented matroid $(E \setminus A,{\cal V}^*|_{E \setminus A})$, where ${\cal V}^*|_{E \setminus A}:=\{ V|_{E \setminus A} \mid V \in {\cal V}^*\}$.
It is also called the {\it restriction} of ${\cal M}$ by $E \setminus A$ and is denoted by ${\cal M}|_{E \setminus A}$.
We use the notation $\rank_{\cal M}(E\setminus A)$ to denote $\rank({\cal M}|_{E\setminus A})$.
When there is no confusion, we simply write it as $\rank(E \setminus A)$.
The {\it contraction} ${\cal M}/A$ of ${\cal M}$ by $A$ is the oriented matroid $(E \setminus A, {\cal V}^*/A)$,
where ${\cal V}^*/A:= \{ V|_{E\setminus A} \mid V \in {\cal V}^*, V|_A = 0 \}$.
We have
\begin{align*}
\rank({\cal M}) = \rank({\cal M}|_A) + \rank({\cal M}/A).
\end{align*}
Let $\chi$ be a chirotope of ${\cal M}$ and $s:=\rank(E \setminus A)$, $t:=\rank(A)$.
Take $a_1,\dots,a_{r-s} \in A$ with $\rank((E \setminus A) \cup \{ a_1,\dots,a_{r-s}\})=r$
and $b_1,\dots,b_t \in E$ with $\rank(\{ b_1,\dots,b_t\})=t$.
Then, we define $\chi_{\setminus A}: (E\setminus A)^s \rightarrow \{ +,-,0\}$ and $\chi_{/A}: (E\setminus A)^{r-t} \rightarrow \{ +,-,0\}$
as follows.
\begin{align*}
\chi_{\setminus A} (i_1,\dots,i_s) := \chi(i_1,\dots,i_s,a_1,\dots,a_{r-s}),\\
\chi_{/A}(j_1,\dots,j_{r-t}) := \chi(j_1,\dots,j_{r-t},b_1,\dots,b_t)
\end{align*}
for all $i_1,\dots,i_s,j_1,\dots,j_{r-t} \in E$.
Then, $\chi_{\setminus A}$ and $\chi_{/A}$ are chirotopes of ${\cal M} \setminus A$ and ${\cal M}/A$ respectively.
It is important to note that $\chi_{\setminus A}$ and $\chi_{/A}$ are determined (up to taking negative) independently of the choice of 
$a_1,\dots,a_{r-s}$ and $b_1,\dots,b_t$. See \cite[p.125]{BLSWZ99}, for details.

If an oriented matroid ${\cal N}$ can be written as ${\cal N} = {\cal M}|_F$ for some $F \subseteq E$ with $|E \setminus F|=1$,
${\cal M}$ is said to be a {\it single element extension} of ${\cal N}$.
We may have two cases: (i) $\rank_{\cal M}(E) = \rank_{\cal M}(F)$ or (ii) $\rank_{\cal M}(E) = \rank_{\cal M}(F)$.
If $\rank_{\cal M}(E) = \rank_{\cal M}(F)+1$, the element $p \in E \setminus F$ is called a {\it coloop} of ${\cal M}$.

\subsubsection{Some classes of oriented matroids}
We are sometimes interested in some special classes of oriented matroids, which have better correspondence with some geometric structures.
An element $e \in E$ is called a {\it loop} if $X(e)=0$ for all $X \in {\cal V}^*$.
An oriented matroid ${\cal M}$ is said to be {\it loopless} if it has no loops.
If $X(e)=X(f)$ (resp.\ $X(e)=-X(f)$) for all $X \in {\cal V}^*$, $e$ and $f$ are said to be {\it parallel} (resp.\ {\it antiparallel}).
${\cal M}$ is said to be {\it simple} if it has neither loops, distinct parallel elements nor distinct antiparallel elements.
If ${\cal M}$ is realized by a vector configuration, a loop corresponds to the zero vector, parallel elements to vectors with the same direction,
and antiparallel elements to vectors with the opposite directions.

\begin{defn}(uniform oriented matroids)\\
An oriented matroid ${\cal M} = (E, \{ \chi, -\chi \})$ of rank $r$ is {\it uniform} if
$\chi (i_1,\dots,i_r) \neq 0$ for all distinct $i_1,\dots,i_r \in E$.
\end{defn}
Equivalently, ${\cal M}$ is uniform if $|C^0| = n -r-1$ for all circuits $C \in {\cal C}$.
It is also possible to say that ${\cal M}$ is uniform if $|D_0| = r-1$ for all cocircuits $D \in {\cal C}^*$.
If ${\cal M}$ is realized by a point configuration $P$, ${\cal M}$ is uniform if and only if $P$ is in general position.

\begin{defn}(acyclic oriented matroids)\\
If an oriented matroid ${\cal M}$ satisfies one of the following equivalent conditions, it is said to be {\it acyclic}.
\begin{itemize}
\item ${\cal M}$ has the positive covector.
\item ${\cal M}$ does not have a non-negative vector.
\end{itemize}
\end{defn}
Oriented matroids arising from point configurations are always acyclic.
If an oriented matroid is not acyclic, it is said to be {\it cyclic}.
\begin{defn}(cyclic oriented matroids)\\
If an oriented matroid ${\cal M}=(E,{\cal V})$ has a non-negative vector, ${\cal M}$ is said to be {\it cyclic}.
If there exists a non-negative vector $X_e \in {\cal V}$ with $X_e(e)=+$ for every $e \in E$, we say that 
${\cal M}$ is {\it totally cyclic}.
\end{defn}
An oriented matroid is acyclic if and only if the dual oriented matroid is totally cyclic.

For an acyclic oriented matroid ${\cal M}$ on a ground set $E$,
an element $e \in E$ is called an {\it extreme point} of ${\cal M}$ if there is the covector $X_e$
such that $X_e(e)=0$ and $X_e(f)=+$ for all $f \in E \setminus \{ e \}$.
If all elements of ${\cal M}$ are extreme points, ${\cal M}$ is said to be a {\it matroid polytope}. 
\begin{defn}(matroid polytopes)\\
If an acyclic oriented matroid ${\cal M}$ satisfies one of the following equivalent conditions, ${\cal M}$ is called a {\it matroid polytope}.
\begin{itemize}
\item 
For every $e \in E$, where $E$ is the ground set of ${\cal M}$, 
there exists the covector $X_e$ of ${\cal M}$ such that $(X_e)^+=\{ e\}$ and $(X_e)^0 = E \setminus \{ e \}$.
\item ${\cal M}$ does not have a vector $V$ with $|V^+| \leq 1$.
\end{itemize}
\end{defn}
\label{def:faces}
The notion of faces of polytopes naturally translates into the oriented matroid setting.
\begin{defn}(faces of matroid polytopes)\\
For a matroid polytope ${\cal M}=(E,{\cal V}^*)$, a subset $F \subseteq E$ is a {\it face} of ${\cal M}$
if there exists $X \in {\cal V}^*$ such that $X^0=F$ and $X^+=E \setminus F$, or equivalently
if ${\cal M}/F$ is acyclic.
\end{defn}
In the rank $3$ case, every matroid polytope is {\it relabeling equivalent} to an {\it alternating matroid} 
(see Proposition \ref{prop:rank_relabeling} in Appendix 1).
\begin{defn}\mbox{}
\begin{itemize}
\item Two oriented matroids ${\cal M}=(E,{\cal V}^{*})$ and ${\cal N}=(F,{\cal W}^{*})$
are said to be {\it relabeling equivalent} (or {\it isomorphic}) if there exists a bijection $\phi:E \rightarrow F$
such that $X \in {\cal V}^{*} \Leftrightarrow \phi (X) \in {\cal W}^{*}$. 
For relabeling equivalent oriented matroids ${\cal M}$ and ${\cal N}$, we write ${\cal M} \simeq {\cal N}$.
\item Two oriented matroids ${\cal M}$ and ${\cal N}$ are said to be {\it reorientation equivalent} if
$_{-A}{\cal M}$ and ${\cal N}$ are relabeling equivalent for some $A \subseteq E$.
Here, $_{-A}{\cal M}$ is the oriented matroid specified by the collection of covectors $\{ _{-A}X \mid X \in {\cal V}^*\}$,
where $_{-A}X \in \{ +,-,0\}^E$ is the vector defined as follows:
\[ _{-A}X(e):=
\begin{cases}
-X(e) & \text{ for $e \in A$,}\\
X(e) & \text{ for $e \notin A$.}
\end{cases}
\]
The oriented matroid $_{-A}{\cal M}$ is called the reorientation of ${\cal M}$ by $A$.
\end{itemize}
\end{defn}
If an oriented matroid ${\cal M}$ is specified in chirotope representation $(E,\{ \chi, -\chi \})$,
the reorientation $_{-A}{\cal M}$ is given by $(E,\{ {_{-A}\chi}, -{_{-A}\chi} \})$, where
\[ {_{-A}}\chi(i_1,\dots,i_r) := (-1)^{|A \cap \{ i_1,\dots,i_r \}|}\chi (i_1,\dots,i_r) \text{ for $i_1,\dots,i_r \in E$.}\]
Two oriented matroids ${\cal M}=(E, \{ \chi, -\chi \})$ and ${\cal N}=(F, \{ \chi', -\chi' \})$ are relabeling equivalent if and only if
there exists a bijection $\phi:E \rightarrow F$ such that
\begin{align*}
 \chi' (\phi (i_1),\dots,\phi (i_r) ) &= \chi(i_1,\dots,i_r) \text{ for all $i_1,\dots,i_r \in E$, or} \\
 \chi'(\phi (i_1),\dots,\phi (i_r) ) &= -\chi(i_1,\dots,i_r) \text{ for all $i_1,\dots,i_r \in E$,}
\end{align*}
where $r$ is the rank of ${\cal M}$ (and ${\cal N}$).

\begin{defn}(Alternating matroids)\\
Let $r,n \in \mathbb{N}$ be such that $n \geq r$.
The alternating matroid $A_{r,n}$ is the oriented matroid $([n],\{ \chi, -\chi \})$ of rank $r$ such that
$\chi (i_1,\dots,i_r)=+$ for all $i_1,\dots,i_r \in [n]$ with $1 \leq i_1 < i_2 < \dots < i_r \leq n$.
\end{defn}
Every alternating matroid is known to be a matroid polytope and to be realizable. 
Circuit and cocircuit structure of alternating matroids is understood well.
\begin{prop}(Circuits and cocircuits of an alternating matroid)
\label{prop:alter_circuits}
\begin{itemize}
\item
A sign vector $X \subseteq \{ +,-,0\}^n$ is a circuit of $A_{r,n}$ if and only if
$|X^0| = n-r-1$ and $X(i) = -X(j)$ for all consequent $i,j \in X^+ \cup X^-$.

\item
A sign vector $Y \subseteq \{ +,-,0\}^n$ is a cocircuit of $A_{r,n}$ if and only if
$|Y^0| = r-1$ and $Y(i) = Y(j)$ for all $i,j \in [n]$ such that $i-j$ is odd, 
and $Y(i) = -Y(j)$ for all $i,j \in [n]$ such that $i-j$ is even.
\end{itemize}
\end{prop} 
For more details on alternating matroids, see \cite[Section 9.4]{BLSWZ99}.

\subsection{Definitions on symmetries}
\subsubsection{Geometric symmetries of point configurations}
Let $P:=( {\bm p_1},\dots,{\bm p_n} ) \in \mathbb{R}^{d \times n}$ be a $d$-dimensional point configuration.
A permutation $\sigma$ on $[n]$ is a {\it geometric symmetry} of $P$ if there exists an affine transformation $f$ such that
\[ f({\bm p_i}) = {\bm p_{\sigma(i)}} \ \text{for all $i \in [n]$.} \]
Here, we present some other equivalent formulations.
Let ${\bm v_1},\dots,{\bm v_n} \in \mathbb{R}^{d+1}$ be the associated vector configuration of $P$.
A permutation $\sigma$ on $[n]$ is a geometric symmetry of $P$ if and only if 
there exists a linear transformation $A$ such that
\[ A{\bm v_i} = {\bm v_{\sigma(i)}}  \text{ for all $i \in [n]$.} \]
Actually, this condition is equivalent to the following condition:
\[
\det({\bm v_{\sigma(i_1)}},\dots,{\bm v_{\sigma(i_{d+1})}}) = \det({\bm v_{i_1}},\dots,{\bm v_{i_{d+1}}})  \text{ for all $i_1, \dots, i_{d+1} \in [n]$}  \text{ $\cdots$ (G1), or}
\]
\[
\det({\bm v_{\sigma(i_1)}},\dots,{\bm v_{\sigma(i_{d+1})}}) = - \det({\bm v_{i_1}},\dots,{\bm v_{i_{d+1}}})  \text{ for all $i_1, \dots, i_{d+1} \in [n]$.} \text{ $\cdots$ (G2).}
\]
The permutation $\sigma$ is called a {\it (geometric) rotational symmetry} of $P$ if (G1) holds.
On the other hand, $\sigma$ is called a {\it (geometric) reflection symmetry} of $P$ if (G2) holds.

\subsubsection{Symmetries of oriented matroids}
For an oriented matroid ${\cal M}=(E,\{ \chi,-\chi\})$ of rank $r$,
a permutation $\sigma$ on $E$ is a {\it symmetry} of ${\cal M}$ if ${\cal M}$ is invariant under $\sigma$, i.e.,
the following holds:
\[ (\sigma \cdot \chi) (i_1,\dots,i_r) := \chi (\sigma(i_1),\dots,\sigma(i_r)) = \chi (i_1,\dots,i_r) \text{ for all $i_1, \dots, i_r \in E$} \text{ $\cdots$ (O1),  or}\]  
\[ (\sigma \cdot \chi) (i_1,\dots,i_r) := \chi (\sigma(i_1),\dots,\sigma(i_r)) = - \chi (i_1,\dots,i_r) \text{ for all $i_1, \dots, i_r \in E$.} \text{ $\cdots$ (O2).} \]
The permutation $\sigma$ is called a {\it (combinatorial) rotational symmetry} of ${\cal M}$ if Condition (O1) holds. 
We call $\sigma$ a {\it (combinatorial) reflection symmetry} of ${\cal M}$ if it satisfies Condition (O2).
The group formed by all rotational symmetries of ${\cal M}$ is called the {\it rotational symmetry group} of ${\cal M}$
and is denoted by $R({\cal M})$.
The group formed by all symmetries of ${\cal M}$ is called the (full) symmetry group of ${\cal M}$ and is denoted by $G({\cal M})$.

A symmetry $\sigma$ of ${\cal M}$ acts on the cocircuits ${\cal C}^*$ and the covectors ${\cal V}^*$ as follows.
For a permutation $\sigma$ on $E$ and a sign vector $X \in \{ +,-,0\}^E$, 
let $\sigma \cdot X \in \{ +,-,0\}^E$ be such that $(\sigma \cdot X)(e) = X (\sigma (e))$ for each $e \in E$.
Then, we have $\sigma \in G({\cal M})$ if and only if $\sigma (X) \in {\cal V}^*$ (resp.\ ${\cal C}^*$) 
for all $X \in {\cal V}^*$ (resp.\ ${\cal C}^*$).

\subsubsection{Combinatorial symmetries of point configurations}
Let $P:=( {\bm p_1},\dots,{\bm p_n} ) \in \mathbb{R}^{d \times n}$ be a $d$-dimensional point configuration.
A permutation $\sigma$ on $[n]$ is an {\it oriented-matroid symmetry} of $P$ if 
the associated oriented matroid of $P$ is invariant under $\sigma$, i.e.,
\[
{\rm sign}(\det({\bm v_{\sigma(i_1)}},\dots,{\bm v_{\sigma(i_{d+1})}})) = {\rm sign}(\det({\bm v_{i_1}},\dots,{\bm v_{i_{d+1}}}))  \text{ for all $i_1, \dots, i_{d+1} \in [n]$, or}
\]
\[
{\rm sign}(\det({\bm v_{\sigma(i_1)}},\dots,{\bm v_{\sigma(i_{d+1})}})) = -{\rm sign}(\det({\bm v_{i_1}},\dots,{\bm v_{i_{d+1}}}))  \text{ for all $i_1, \dots, i_{d+1} \in [n]$.}
\]
On the other hand, $\sigma$ is a {\it matroidal symmetry} if the associated matroid of $P$ is invariant under $\sigma$, i.e.,
\[ \delta_P (\sigma(i_1),\dots,\sigma(i_{d+1})) = \delta_P (i_1,\dots,i_{d+1}) \text{ for all $i_1, \dots, i_{d+1} \in [n]$.} \]
oriented-matroid symmetries and  matroidal symmetries are often called {\it combinatorial symmetries}.

It can easily be checked that every geometric symmetry induces a combinatorial symmetry.
Given a combinatorial symmetry $\sigma$ of $P$, $\sigma$ is said to be {\it geometrically realizable} if
there is a point configuration $P'$ with ${\cal M}_P = {\cal M}_{P'}$ that has $\sigma$  as a geometric symmetry. 
In general, every combinatorial symmetry is not geometrically realizable.
We will study this issue for matroids further in Section 3.

\subsubsection{Symmetries of oriented matroids and inseparability graphs}
Inseparability graphs~\cite{LV78} is a useful tool in studying symmetries of oriented matroids.
Here, we give a brief explanation on inseparability graphs. For more details, see \cite[Section 7.8]{BLSWZ99}.

\begin{defn}(inseparability graphs)\\
For an oriented matroid ${\cal M}=(E,{\cal V}^*)$, the {\it inseparability graph} $IG({\cal M})=({\cal V(M)}, {\cal E(M)})$ of ${\cal M}$ is the graph with
${\cal V(M)}=E$ such that $\{ e, f \} \in {\cal E(M)}$ if and only if  $e \neq f$, and $X(e) = X(f)$ for all $X \in {\cal V}^*$ or $X(e) = -X(f)$ for all $X \in {\cal V}^*$.
\end{defn}
Note that $IG({\cal M})=IG({_{-A}{\cal M}})$ for any $A \subseteq E$.
An important observation is that if $\sigma$ is a symmetry of ${\cal M}$, then $\sigma$ is also a symmetry of $IG({\cal M})$.
Structure of inseparability graphs of uniform oriented matroids is well understood as shown in the following theorem.
\begin{thm}
\label{thm:inseparability}
(\cite{CD90})\\
Let ${\cal M}$ be a uniform oriented matroid of rank $r$ on $E$.
\begin{itemize}
\item If $r=1$ or $r=|E|-1$, then $IG({\cal M})$ is the complete graph on $E$.
\item If $r=2$, then $IG({\cal M})$ is an $|E|$-cycle.
\item If $2 \leq r \leq |E|-2$, then $IG({\cal M})$ is either an $|E|$-cycle or disjoint union of $k \geq 2$ paths.
\end{itemize}
\end{thm}

\subsubsection{Some useful observations}
\label{sec:useful}
Let ${\cal M}$ be a simple oriented matroid of rank $r$ on a ground set with a chirotope $\chi$, the cocircuits ${\cal C}^*$.
Suppose that there exists $A \subseteq E$ with $\rank_{\cal M}(A) = r-1$ that is invariant under $R({\cal M})$.
Let $D \in {\cal C}^*$ be one of the opposite cocircuits with $D^0 \supseteq A$.
Then, it holds that $\sigma \cdot D = D$ for all $\sigma \in R({\cal M})$ if and only if $\sigma$ is a rotational symmetry of ${\cal M}/A$
and that  $\sigma \cdot D = -D$ for all $\sigma \in R({\cal M})$ if and only if $\sigma|_{E \setminus A}$ is a reflection symmetry of ${\cal M}/A$
 (recall that $D=(\chi (\lambda, e))_{e \in E}$ for some $\lambda \in E^{r-1}$).
We have
\begin{align*}
 \chi (e,i_1,\dots,i_{r-1}) &= \chi (\sigma (e), \sigma (i_1), \dots, \sigma (i_{r-1})) \\
&= D(e)D(\sigma(e))\chi (e, \sigma (i_1), \dots, \sigma (i_{r-1})) 
\end{align*}
for all $i_1,\dots,i_{r-1} \in A$ and $e \in E \setminus A$ with $\rank_{\cal M}(A \cup \{ e \} ) = r$.
Therefore, if $\sigma|_{E \setminus A}$ is a rotational symmetry of ${\cal M}/A$, then $\sigma|_A$ is a rotational symmetry of ${\cal M}|_A$.
Similarly,  if $\sigma|_{E \setminus A}$ is a reflection symmetry of ${\cal M}/A$, then $\sigma|_A$ is a reflection symmetry of ${\cal M}|_A$.

\subsubsection{Symmetries of alternating matroids of rank $3$}
In this paper, alternating matroids of rank $3$ will appear repeatedly.
Here, we give a summary on symmetries of alternating matroids of rank $3$.
For now, we relabel the elements $1,2,\dots,n$ of $A_{3,n}$ to $0,1,\dots,n-1$.
Then, for $k=0,1,\dots,n-1$, the permutation $\sigma_k$ on $\{ 0 \} \cup [n-1]$ defined by
\[ \sigma_k (i) = i + k \bmod n\]
is a rotational symmetry of $A_{3,n}$.
We call $\sigma_k$ the {\it $k$-th rotational symmetry} of $A_{3,n}$.
Note that $R(A_{3,n})$ is generated by the 1st rotational symmetry.
On the other hand, the permutation $\tau$ on $\{ 0 \} \cup [n-1]$ defined by
\[ \tau (i) = -i \bmod n\]
is a reflection symmetry of $A_{3,n}$.
The reflection symmetries of $A_{3,n}$ are described by $\tau, \sigma_1 \tau, \dots, \sigma_{n-1} \tau$.
We remark that $\tau^2 =\id$ and $\tau \sigma_k \tau^{-1} = \sigma_k^{-1}$ for every $k$.
Therefore, we have $R(A_{3,n}) \simeq \mathbb{Z}_n$ and $G(A_{3,n}) \simeq D_{2n}$.

The alternating matroid $A_{3,n}$ is geometrically realized by a regular $n$-gon.
In this setting, $k$-th rotational symmetry $\sigma_k$ corresponds to the rotation by $\frac{2k}{n}$.
The reflection symmetry $\tau \sigma_k$ corresponds to one of the reflection symmetries of a regular $n$-gon.
When $n$ is odd, $\tau \sigma_k$ is the reflection across the axis passing through the point $k$ and the midpoint of points $k+\frac{n-1}{2}$
and $k-\frac{n-1}{2}$ (the numbers are interpreted modulo $n$).
When $n$ and $k$ are even,  $\tau \sigma_k$ is the reflection across the axis formed by points $k$ and $k+\frac{n}{2}$.
If $n$ is even and $k$ is odd, $\tau \sigma_k$ is the reflection across the axis determined by the midpoint of $k$ and $k+1$,
and that of $k+\frac{n}{2}$ and $k+\frac{n}{2}+1$.

\section*{Notations}
Here, we summarize the notations we have introduced.
In the following, we suppose that ${\cal M}$ and ${\cal N}$ are oriented matroids on a ground set $E$
and that $X$ and $Y$ are sign vectors on $E$, and that $A$ is a subset of $E$.
\begin{itemize}
\item $[n] := \{ 1,2,\dots,n\}$.
\item $X^0: = \{ e \in E \mid X(e) = 0\}$.
\item $X^+: = \{ e \in E \mid X(e) = +\}$.
\item $X^-: = \{ e \in E \mid X(e) = -\}$.
\item $X \succeq Y$: $X(e) = Y(e)$ or $Y(e)=0$ for all $e \in E$.
\item $X \geq Y$: $X(e) \geq Y(e)$ for all $e \in E$, where $- < 0 < +$.
\item ${\rm span}_{\cal M}(A)$: the flat of ${\cal M}$ spanned by $A$.
\item $\rank_{\cal M}(A)$: the rank of $X$ of the flat spanned by $A$.
\item $_{-A}{\cal M}$: the reorientation of ${\cal M}$ by $A$.
\item ${\cal M}|_A$: the restriction of ${\cal M}$ to $A$.
\item ${\cal M}/A$: the contraction of ${\cal M}$ by $A$.
\item ${\cal M} \simeq {\cal N}$: the oriented matroids ${\cal M}$ and ${\cal N}$ are relabeling equivalent (isomorphic).
\item $IG({\cal M})$: the inseparability graph of ${\cal M}$.
\item $A_{r,n}$: the alternating matroid of rank $r$ on the ground set $[n]$.
\item $R({\cal M})$: the rotational symmetry group of ${\cal M}$.
\item $G({\cal M})$: the full symmetry group of ${\cal M}$.
\item $D_{2n}$: the dihedral group of order $2n$.
\end{itemize}

\section{A gap between matroidal symmetries and geometric symmetries of point configurations}
As a starting point, we study a gap between matroidal symmetries and geometric symmetries of point configurations.

\subsection{A rank 4 matroid with 8 elements having a matroidal symmetry that cannot be realized geometrically}
\label{subsec:construction}
Let $P=( {\bm p_1},{\bm p_2},\dots,{\bm p_8} ) \in \mathbb{R}^{3 \times 8}$ be the point configuration defined by
\begin{align*}
({\bm p_1},{\bm p_2},\dots,{\bm p_8})
=
\begin{pmatrix}
0 & 0 & 1 & 1 & 0 & 0 & \frac{1}{4} & \frac{1}{5} \\
0 & 0 & 0 & 0 & 1 & 1 & \frac{1}{4} & \frac{1}{5} \\
0 & 1 & 0 & 1 & 0 & 1 & -1          & 2
\end{pmatrix}.
\end{align*}
The associated matroid $M_P$ is specified by the following $4$-element non-bases.
\begin{align*}
\{ 1,2,3,4 \}, \{ 1,2,5,6\}, \{ 3,4,5,6\}. 
\end{align*}
The matroid $M_P$ has a symmetry $\sigma$:
\begin{align*}
 \sigma = 
\begin{pmatrix}
1 & 2 & 3 & 4 & 5 & 6 & 7 & 8 \\
3 & 4 & 5 & 6 & 1 & 2 & 7 & 8
\end{pmatrix}.
\end{align*}
In the following, we prove that $M_P$ cannot be realized as a point configuration that has $\sigma$ as a geometric symmetry.
For $S \subseteq \mathbb{R}^3$, we will denote by ${\rm aff}(S)$ (resp. ${\rm conv}(S), {\rm relint}(S)$) the affine hull (resp. convex hull, relative interior) of $S$.
We do not write parentheses when there is no confusion.

We assume that there is an affine automorphism $f$ of $P$ inducing $\sigma$.
Then $C := {\rm conv}\{ {\bm p_1},{\bm p_3},{\bm p_5} \}$ is setwise invariant under $f$.
Note that $f$ has a fixed point ${\bm w_0}$ in ${\rm relint}({\rm conv}\{ {\bm p_1},{\bm p_3},{\bm p_5} \})$ (${\bm w_0}=({\bm p_1}+{\bm p_3}+{\bm p_5})/3$).
Since both ${\bm w_0}$ and ${\bm p_8}$ are invariant under $f$, ${\rm aff}\{{\bm w_0},{\bm p_8}\}$ is pointwise invariant under $f$.
The same applies to ${\rm aff}\{{\bm w_0},{\bm p_7}\}$.
Assume that ${\bm p_7} \notin {\rm aff}\{{\bm w_0},{\bm p_8}\}$. 
Then $D := {\rm aff}\{{\bm w_0},{\bm p_7},{\bm p_8}\}$ is $2$-dimensional and is pointwise 
invariant under $f$. Note that ${\rm dim}({\rm aff}(C) \cap D) = 1$ since ${\rm aff}(C) \cap D \neq \emptyset$.
Thus the restriction $f|_{{\rm aff}(C)}$ fixes a $1$-dimensional space pointwisely.
This contradicts to the fact that every non-trivial rotational symmetry in the 2-dimensional Euclidean space has a unique fixed point.
Therefore, we have ${\bm p_7} \in {\rm aff}\{{\bm w_0},{\bm p_8}\}$.

Now we assume that the lines ${\rm aff}\{{\bm p_1},{\bm p_2}\}$ and ${\rm aff}\{{\bm p_3},{\bm p_4}\}$  are parallel.
Then the lines ${\rm aff}\{{\bm p_1},{\bm p_2}\}$,
 ${\rm aff}\{{\bm p_3},{\bm p_4}\}$ and ${\rm aff}\{{\bm p_5},{\bm p_6}\}$ are all parallel.
Therefore, the lines ${\rm aff}\{{\bm p_1},{\bm p_2}\}$ and ${\rm aff}\{{\bm w_0},{\bm w_1}\}(={\rm aff}\{{\bm p_7},{\bm p_8}\})$ are also parallel
and thus $\chi_P (1,2,7,8) = 0$, which is a contradiction.

Therefore, the lines ${\rm aff}\{{\bm p_1},{\bm p_2}\}$ and ${\rm aff}\{{\bm p_3},{\bm p_4}\}$ are not parallel.
Then the lines ${\rm aff}\{{\bm p_1},{\bm p_2}\}$, ${\rm aff}\{{\bm p_3},{\bm p_4}\}$ and ${\rm aff}\{{\bm p_5},{\bm p_6}\}$ are all non-parallel.
In this case, the lines ${\rm aff}\{{\bm p_1},{\bm p_2}\}$ and ${\rm aff}\{{\bm p_3},{\bm p_4}\}$ intersect because points 
${\bm p_1}$, ${\bm p_2}$, ${\bm p_3}$ and ${\bm p_4}$ are on the same plane.
The same applies to ${\rm aff}\{{\bm p_3},{\bm p_4}\}$ and ${\rm aff}\{{\bm p_5},{\bm p_6}\}$, and to ${\rm aff}\{{\bm p_1},{\bm p_2}\}$ and ${\rm aff}\{{\bm p_5},{\bm p_6}\}$.
Those three intersection points are in fact the same. This is proved by the following simple computation.
\begin{align*}
\begin{split}
&{\rm aff}\{{\bm p_1},{\bm p_2}\} \cap {\rm aff}\{{\bm p_3},{\bm p_4}\} \\
&=  ({\rm aff}\{{\bm p_1},{\bm p_2},{\bm p_3},{\bm p_4}\} \cap {\rm aff}\{{\bm p_1},{\bm p_2},{\bm p_5},{\bm p_6}\}) \cap ({\rm aff}\{{\bm p_1},{\bm p_2},{\bm p_3},{\bm p_4}\} \cap {\rm aff}\{{\bm p_3},{\bm p_4},{\bm p_5},{\bm p_6}\})\\
&= {\rm aff}\{{\bm p_1},{\bm p_2},{\bm p_3},{\bm p_4}\} \cap {\rm aff}\{{\bm p_3},{\bm p_4},{\bm p_5},{\bm p_6}\}  \cap {\rm aff}\{{\bm p_1},{\bm p_2},{\bm p_5},{\bm p_6}\}.
\end{split}
\end{align*}
This intersection is invariant under $A$ and thus is on ${\rm aff}\{{\bm p_7},{\bm p_8}\}$.
From the above discussion, the lines ${\rm aff}\{{\bm p_5},{\bm p_6}\}$ and ${\rm aff}\{{\bm p_7},{\bm p_8}\}$ have the intersection ${\bm p}$, and thus
points ${\bm p_5}$, ${\bm p_6}$, ${\bm p_7}$ and ${\bm p_8}$ are on the same plane. This leads to $\chi_P(5,6,7,8) = 0$, which is a contradiction. 
\begin{figure}[h]
\begin{center}
\includegraphics[scale=0.2]{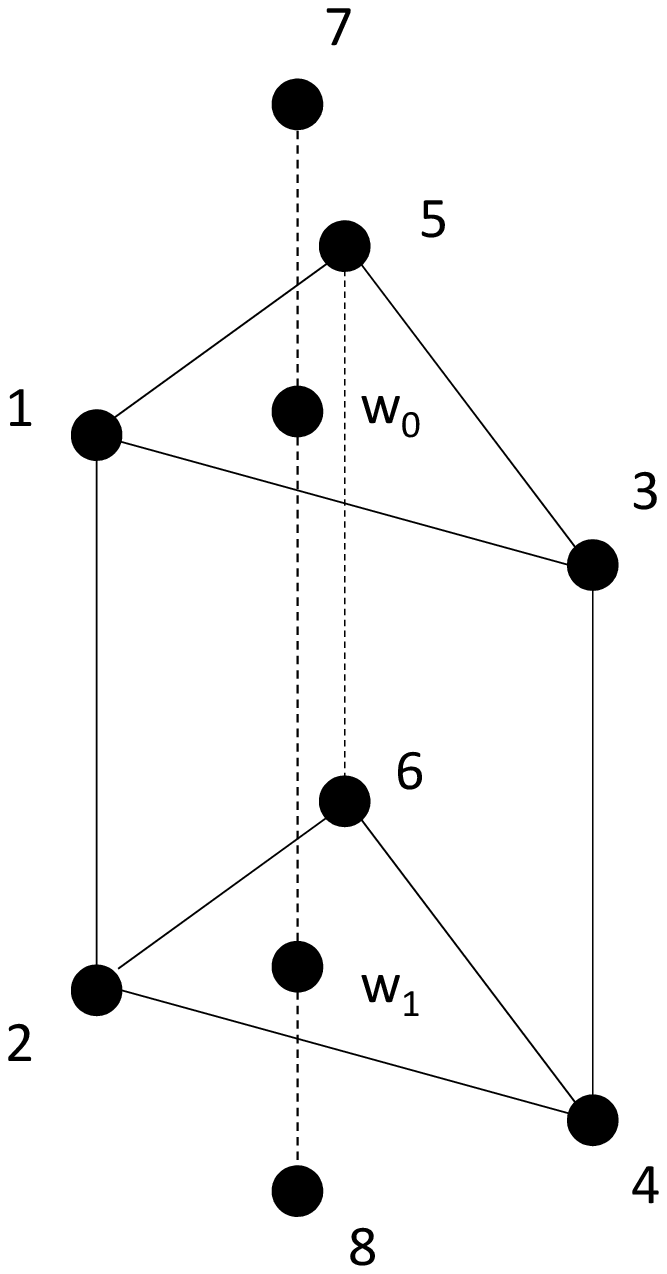}
\includegraphics[scale=0.2]{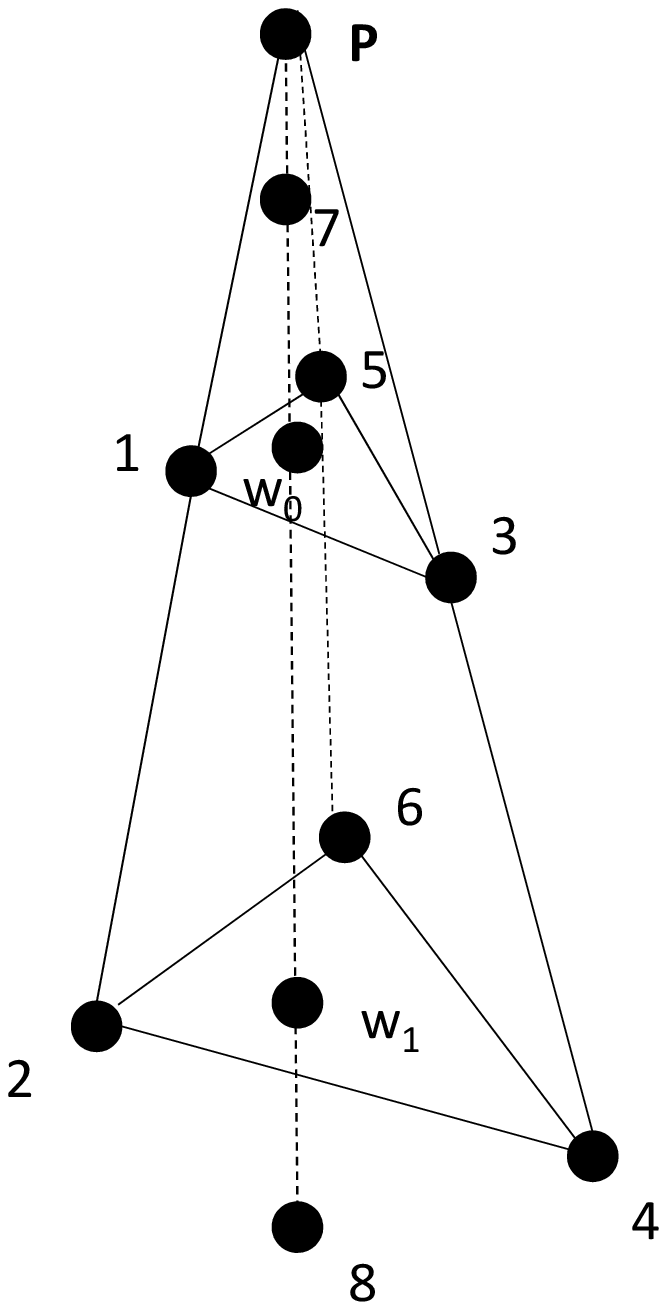}
\caption{Case1 [left], Case2 [right]}
\end{center}
\end{figure}

\subsection{General construction}
In a similar way to Section \ref{subsec:construction}, 
we prove that from an arbitrary $2$-dimensional configuration of $n$ points with a non-trivial geometric rotational symmetry of order $m \geq 3$,
we can construct a $3$-dimensional configuration of $2n+2$ points with a matroidal symmetry that cannot
be realized geometrically.
\begin{thm}
\label{thm:general_construction}
Let $P:=( {\bm p_1},\dots,{\bm p_n} ) \in \mathbb{R}^{2 \times n}$ be a point configuration with a non-trivial geometric rotational symmetry $\sigma$
of order $m \geq 3$. 
Then let us consider a point configuration $Q := ( {\bm q_1},\dots,{\bm q_{2n+2}} ) \in \mathbb{R}^{3 \times (2n+2)}$ 
with ${\bm q_i}:= \begin{pmatrix} {\bm p_i} \\ 0\end{pmatrix},{\bm q_{n+i}}:= \begin{pmatrix} {\bm p_i} \\ 1\end{pmatrix}$ for $i=1,\dots,n$
where ${\bm q_{2n+1}},{\bm q_{2n+2}}$ are {\it generic} points and 
where 
the line ${\rm aff}\{ {\bm q_{2n+1}},{\bm q_{2n+2}}\}$ is not parallel to 
${\rm aff}\{ {\bm p_1},{\bm p_{n+1}}\}$.
(Here, we say that a point ${\bm r}$ in a point configuration $R$ is generic if it holds that ${\bm r} \notin {\rm aff}\{ {{\bm r_1},{\bm r_2}}\}$ for all ${\bm r_1},{\bm r_2} \in R \setminus \{ {\bm r} \}$.)
Then, the point configuration $Q$ has a matroidal symmetry that cannot be realized geometrically.
\end{thm}
\begin{proof}
Let $M_Q$ is the associated matroid of $Q$.
Note that $M_Q$ has a symmetry
\begin{align*}
\tau := 
\begin{pmatrix}
1          & \cdots & n          & n+1          & \cdots & 2n           & 2n+1 & 2n+2 \\
\sigma (1) & \cdots & \sigma (n) & n+\sigma(1)  & \cdots & n+\sigma(n)  & 2n+1 & 2n+2
\end{pmatrix}.
\end{align*}
Note that the order of $\tau$ is $m$.
Suppose that there is a realization $R=({\bm r_1},\dots,{\bm r_{2n+2}}) \in \mathbb{R}^{3 \times (2n+2)}$of $M_Q$
and an affine transformation $f$ of $\mathbb{R}^{3}$ with $f({\bm r_i}) = {\bm r_{\tau(i)}}$ for $i=1,\dots,2n+2$.

Let ${\bm r_{i_0}}$ be an extreme point of $C_1:={\rm conv}\{ {\bm r_1},\dots,{\bm r_n} \}$.
Then ${\bm r_{i_0+n}}$ is also an extreme point of $C_2:={\rm conv}\{ {\bm r_{n+1}},\dots,{\bm r_{2n}}\}$.
Consider the orbits $O_1=\{ {\bm r_{i_0}},\dots,{\bm r_{i_m}} \}$ and $O_2=\{ {\bm r_{i_0+n}},\dots,{\bm r_{i_m+n}} \}$ under the action of the cyclic group generated by $\tau$.
We remark that $|O_1|=|O_2| \geq 3$. 

Now let us consider ${\bm w_0}:=\frac{1}{m+1}({\bm r_{i_0}}+\dots+{\bm r_{i_m}})$ and ${\bm w_1}:=\frac{1}{m+1}({\bm r_{i_0+n}}+\dots+{\bm r_{i_m+n}})$.
Since ${\bm w_0}$ and ${\bm w_1}$ are invariant under $f$, the line ${\rm aff}\{ {\bm w_0}, {\bm r_{2n+1}}\}$ is pointwise invariant under $f$.
Note that ${\bm r_{2n}} \in {\rm aff}\{ {\bm w_0}, {\bm r_{2n+1}}\}$. Otherwise, we have a $2$-dimensional space
$D:={\rm aff}\{ {\bm w_0}, {\bm r_{2n}}, {\bm r_{2n+1}}\}$ pointwise invariant under $f$.
Since $D \cap {\rm aff}(C_1) \neq \emptyset$, it holds that ${\rm dim}({\rm aff}(C_1) \cap D) = 1$.
Thus the restriction $f|_{{\rm aff}(C_1)}$ fixes a $1$-dimensional space pointwisely,
which is a contradiction.
Therefore, it holds that ${\bm r_{2n}} \in {\rm aff}\{ {\bm w_0}, {\bm r_{2n+1}}\}$. 
Similarly, we have ${\bm w_1} \in {\rm aff}\{ {\bm w_0}, {\bm r_{2n+1}}\}$.

Now assume that the lines ${\rm aff}\{ {\bm r_{i_0}},{\bm r_{i_0+n}} \}$ and ${\rm aff}\{ {\bm r_{i_1}},{\bm r_{i_1+n}} \}$
are parallel.
Then the lines ${\rm aff}\{ {\bm r_{i_0}},{\bm r_{i_0+n}} \}$,  ${\rm aff}\{ {\bm r_{i_0}},{\bm r_{i_0+n}} \}$, $\dots,$ ${\rm aff}\{ {\bm r_{i_m}},{\bm r_{i_m+n}} \}$ are all parallel.
This implies that the lines ${\rm aff}\{ {\bm r_{i_0}},{\bm r_{i_0+n}} \}$ and ${\rm aff}\{ {\bm w_0},{\bm w_1} \}$ are parallel and thus that
$\chi_P(i_0,i_0+n,2n+1,2n+2)=0$. This is a contradiction.

Let us consider the case when the lines ${\rm aff}\{ {\bm r_{i_0}},{\bm r_{i_0+n}} \}$ and ${\rm aff}\{ {\bm r_{i_1}},{\bm r_{i_1+n}} \}$
are not parallel.
Then the lines ${\rm aff}\{ {\bm r_{i_0}},{\bm r_{i_0+n}} \}$ and ${\rm aff}\{ {\bm r_{i_1}},{\bm r_{i_1+n}} \}$ have an intersection point
since $\chi_P(i_0,i_0+n,i_1,i_1+n)=0$.
The same applies to
 ${\rm aff}\{ {\bm r_{i_1}},{\bm r_{i_1+n}} \}$ and ${\rm aff}\{ {\bm r_{i_2}},{\bm r_{i_2+n}} \}$, and
 ${\rm aff}\{ {\bm r_{i_0}},{\bm r_{i_0+n}} \}$ and 
 ${\rm aff}\{ {\bm r_{i_2}},{\bm r_{i_2+n}} \}$.
 These three intersection points are in fact the same.
 This follows from the following relation:
 \begin{align*}
 \begin{split}
 &{\rm aff}\{ {\bm r_{i_0}},{\bm r_{i_0+n}} \} \cap {\rm aff}\{ {\bm r_{i_1}},{\bm r_{i_1+n}} \} \\
 &= {\rm aff}\{ {\bm r_{i_0}},{\bm r_{i_0+n}},{\bm r_{i_1}},{\bm r_{i_1+n}} \}
 \cap
 {\rm aff}\{ {\bm r_{i_1}},{\bm r_{i_1+n}},{\bm r_{i_2}},{\bm r_{i_2+n}} \}
 \cap
 {\rm aff}\{ {\bm r_{i_0}},{\bm r_{i_0+n}},{\bm r_{i_2}},{\bm r_{i_2+n}} \}\\
 &={\rm aff}\{ {\bm r_{i_1}},{\bm r_{i_1+n}} \} \cap {\rm aff}\{ {\bm r_{i_2}},{\bm r_{i_2+n}} \} \\
 &= {\rm aff}\{ {\bm r_{i_0}},{\bm r_{i_0+n}} \} \cap {\rm aff}\{ {\bm r_{i_2}},{\bm r_{i_2+n}} \} 
\end{split}
 \end{align*}
 Repeating the same argument, we conclude that
 \begin{align*}
  {\rm aff}\{ {\bm r_{i_0}},{\bm r_{i_0+n}} \} \cap \dots \cap {\rm aff}\{ {\bm r_{i_m}},{\bm r_{i_m+n}} \}
  \end{align*}
 is non-empty.
 Since it is a fixed point of $f$, it is on the line ${\rm aff}\{ {\bm w_{0}},{\bm w_{1}} \}={\rm aff}\{ {\bm r_{2n+1}},{\bm r_{2n+2}} \}$.
 This implies $\chi_P(i_0,i_0+n,2n+1,2n+2)=0$, which is a contradiction.
\end{proof}
\begin{cor}
For $n, p \geq 3$,
from any $2$-dimensional configuration of $n$ points with a geometric rotational symmetry of order $p$, we can construct
a $3$-dimensional configuration of $2n+2$ points with a matroidal symmetry of order $p$ that cannot be realized geometrically.
\end{cor}

\section{Fixed point properties for rotational symmetries of oriented matroids}
The construction of a gap between matroidal symmetries and geometric symmetries in the previous section is based on the fact that every non-trivial rotation in the $2$-dimensional Euclidean space
has a unique fixed point, but that matroids do not have the corresponding property.
A natural question is whether oriented matroids have the corresponding property or not.
In this section, we study the corresponding property for oriented matroids.

\subsection{Uniqueness property}
The following is an oriented-matroid analogue of the uniqueness property of a fixed point of a non-trivial rotational symmetry 
in the $2$-dimensional Euclidean space. 
\begin{prop}
\label{prop:uniqueness}
Let $E$ be a finite set and ${\cal M}=(E,\{ \chi, -\chi \} )$ be a loopless oriented matroid of rank $3$ on $E$ 
with a rotational symmetry $\sigma$ of order $a$.
Assume that $\sigma$ is a {\it non-trivial} rotational symmetry, i.e.,
there exists $x \in E$ such that 
$x$ and $\sigma(x)$ are not parallel.
Then,
if $\sigma (p) = p, \sigma (q) = q$, it holds that $\rank(\{ p,q\})=1$, i.e.,
$\chi(p,q,x)=0$ for all $x \in E$ (i.e., $p$ and $q$ are parallel or antiparallel). 
\end{prop}
\begin{proof}
\\
Let us assume that ${\rm rank}(\{ p,q\})=2$.
We will see that $x$ and $\sigma (x)$ are parallel for each $x \in E$ and obtain a contradiction.
\begin{quote}
{\bf Lemma.}
Let $x \in E$ be such that $\chi (p,q,x) \neq 0$ (i.e., $\rank (\{ p,q,x\} ) = 3$).
Then, the elements $x$ and $\sigma (x)$ are parallel. 
\end{quote}
{\scshape Proof of lemma:}
\\
For any $s,t \in \mathbb{N}$, the following holds:
\begin{align*}
\begin{split}
\{ \chi(p,q,x) \chi(p,\sigma^s(x),\sigma^{s+t}(x)), -\chi(p,q,\sigma^s(x)) \chi(p,x,\sigma^{s+t}(x)), \chi(p,q,\sigma^{s+t}(x)) \chi(p,x,\sigma^s(x)) \} \\
\supseteq \{ +,-\} \text{ or } =\{ 0 \}
\end{split}
\end{align*}
by (B3) of the chirotope axioms.
Therefore, we have
\begin{align*}
\begin{split}
 \{ \chi(p,q,x) \chi(p,x,\sigma^{t}(x)), -\chi(p,q,x) \chi(p,x,\sigma^{s+t}(x)), \chi(p,q,x) \chi(p,x,\sigma^s(x)) \} \\
\supseteq \{ +,-\} \text{ or } =\{0\}.
\end{split}
\end{align*}
Since $\chi (p,q,x) \neq 0$, it holds that 
\begin{align*}
 \{ \chi(p,x,\sigma^{t}(x)), -\chi(p,x,\sigma^{s+t}(x)),\chi(p,x,\sigma^s(x)) \} 
\supseteq \{ +,-\} \text{ or } =\{0\}.
\end{align*}
Therefore, for any $s,t \in \mathbb{N}$ such that $\chi(p,x,\sigma^{t}(x)) = \chi(p,x,\sigma^s(x))$, we have 
\begin{align*}
 \chi(p,x,\sigma^s(x)) = \chi(p,x,\sigma^{s+t}(x)).
\end{align*}
By induction, we have
\begin{align*} 
\chi(p,x,\sigma(x)) = \chi(p,x,\sigma^l(x)) \text{ for all $l \geq 1$.}
\end{align*}
In particular, we have $\chi (p,x,\sigma^{a}(x)) = \chi (p,x,x) = 0$ and thus
\begin{align*} 
\chi (p,x,\sigma^l (x)) = 0 \text{ for all $l \geq 0$.}
\end{align*}
Similarly, it holds that
\begin{align*} 
\chi (q,x,\sigma^l (x)) = 0 \text{ for all $l \geq 0$.}
\end{align*}
By (B3) of the chirotope axioms,
\begin{align*}
\begin{split}
& \{
\chi (x,y,\sigma (x)) \chi (x,p,q),
- \chi (x,y,p) \chi (x,\sigma (x),q),
\chi (x,y,q) \chi (x,\sigma (x),p)
\} \\
&=
\{ 0,\chi (x,y,\sigma (x)) \chi (x,p,q) \}
\supseteq
\{ +,- \}
\text{ or }
= \{0\}
\end{split}
\end{align*}
for any $y \in E$.
Since $\chi(p,q,x) \neq 0$,
it holds that $\chi (x,y,\sigma (x)) = 0$ for any $y \in E$.
Therefore, the elements $x$ and $\sigma (x)$ are parallel or antiparallel.
In fact, $x$ and $\sigma (x)$ are parallel since $\chi (p,q,x) = \chi (p,q, \sigma (x)) (\neq 0)$.
This proves the lemma.
\\
\\
Next, let us consider $x \in E$ such that $\chi(p,q,x)=0$.
Assume that $x$ and $\sigma(x)$ are neither parallel nor antiparallel (i.e., $\rank(\{ x, \sigma (x) \})=2$).
Then we can take $y \in E$ such that $\chi(y,x,\sigma(x)) \neq 0$ using (I3) of the independent set axioms.
Note that $x,\sigma(x) \in {\rm span}_{\cal M}(\{p,q\})$.
Since $y \notin {\rm span}_{\cal M}(\{ p,q\})$, the elements $y$ and $\sigma (y)$ are parallel by the above lemma.
For any $s,t \in \mathbb{N}$, the following holds:
\begin{align*}
\begin{split}
\{ \chi(y,p,x) \chi(y,\sigma^s(x),\sigma^{s+t}(x)), -\chi(y,p,\sigma^s(x)) \chi(y,x,\sigma^{s+t}(x)), \chi(y,p,\sigma^{s+t}(x)) \chi(y,x,\sigma^s(x)) \} \\
\supseteq \{ +,-\} \text{ or } =\{ 0 \}
\end{split}
\end{align*}
by (B3) of the chirotope axioms.
Note that $\chi (y,p,x)=\chi (y,p,\sigma^s(x))= \chi (y,p,\sigma^{s+t}(x))$.
If $\chi(p,x,y) \neq 0$, then $\chi (y,x,\sigma (x))=0$ similarly to the above discussion, which is a contradiction.
Therefore, we have $\chi (p,x,y)=0$.
Since $p$ and $x$ are not parallel nor antiparallel, there exists $z \in E$ such that $\chi (p,x,z) \neq 0$.
However, it holds that
\begin{align*}
\begin{split}
\{ \chi(p,q,y) \chi(p,x,z), -\chi(p,q,x) \chi(p,y,z), \chi(p,q,z) \chi(p,y,x) \} \\
= \{ \chi(p,q,y) \chi(p,x,z) (\neq 0), 0 \} \supseteq \{ +,-\} \text{ or } =\{ 0 \}.
\end{split}
\end{align*}
This is a contradiction.
Therefore, $x$ and $\sigma (x)$ are parallel or antiparallel.
In fact, $x$ and $\sigma (x)$ are parallel since $\chi (y,p,x) = \chi (\sigma (y),p,\sigma (x)) (\neq 0)$ (recall that $y$ and $\sigma (y)$ are parallel).

As a conclusion, $\sigma$ is a trivial rotational symmetry, which is a contradiction.
Therefore, we have $\rank(\{ p,q \}) = 1$.
\end{proof}
\\
The above proof can easily be extended to the following theorem.
\begin{thm}
\label{thm:uniqueness_general}
Let $E$ be a finite set and ${\cal M}$ a loopless oriented matroid of rank $r$ on $E$ with a non-trivial rotational symmetry $\sigma$.
Let us set $\Fix(\sigma):=\{ e \in E \mid \sigma (e) =e \}$.
Then, we have $\rank({\cal M}|_{\Fix(\sigma)}) \leq r-2$.
\end{thm}
\begin{proof}
Assume that $\rank({\cal M}|_{\Fix(\sigma)}) > r-2$. Then, we can take
$p_1,\dots,p_{r-1} \in \Fix(\sigma)$ and $x \in E$ such that
$\chi(p_1,\dots,p_{r-1},x) \neq 0$.
The same argument as the proof of Theorem \ref{prop:uniqueness} yields a contradiction.
\end{proof}
\\
One of useful applications of the above theorem is as follows.
\begin{cor}
\label{cor:uniqueness}
Let $E$ be a finite set and ${\cal M}$ be an oriented matroid of rank $r$ on $E$ without loops and parallel elements. 
For $\sigma, \tau \in R({\cal M})$, 
if $\sigma |_X = \tau |_X$ for $X \subseteq E$ with $\rank({\cal M}|_X) \geq r-1$,
then $\sigma  = \tau$.
The same holds when $\sigma, \tau \in G({\cal M}) \setminus R({\cal M})$.
\end{cor}
\begin{proof}
Assume that $\sigma |_X = \tau |_X$ but that $\sigma \neq \tau$.
Then $\sigma \tau^{-1}$ is a non-trivial rotational symmetry of ${\cal M}$.
However, $\sigma \tau^{-1}$ fixes $X$ pointwisely.
This is a contradiction.
\end{proof}

\begin{remark}
It might also be natural to ask if the following uniqueness property holds.
\begin{quote}
Let ${\cal M}$ be an oriented matroid of rank $3$ on a ground set $E$ with a non-trivial rotational symmetry $\sigma$.
If single element extensions ${\widehat {\cal M}}'$ and ${\widehat {\cal M}}''$ of ${\cal M}$ by a new element $p$
have rotational symmetries $\sigma'$ and $\sigma''$ respectively such that
$\sigma'|_E =  \sigma''|_E = \sigma$ and $\sigma' (p) = p$ and $\sigma'' (p) = p$,
then we have ${\widehat {\cal M}}' = {\widehat {\cal M}}''$.
\end{quote}
It does not always hold.
See Figure \ref{fig:non_uniqueness}.
The oriented matroid on $6$ elements can be extended by one element to two different oriented matroids
where the new elements are fixed points of rotational symmetries.
\begin{figure}[ht]
\begin{center}
\includegraphics[scale=0.40,clip]{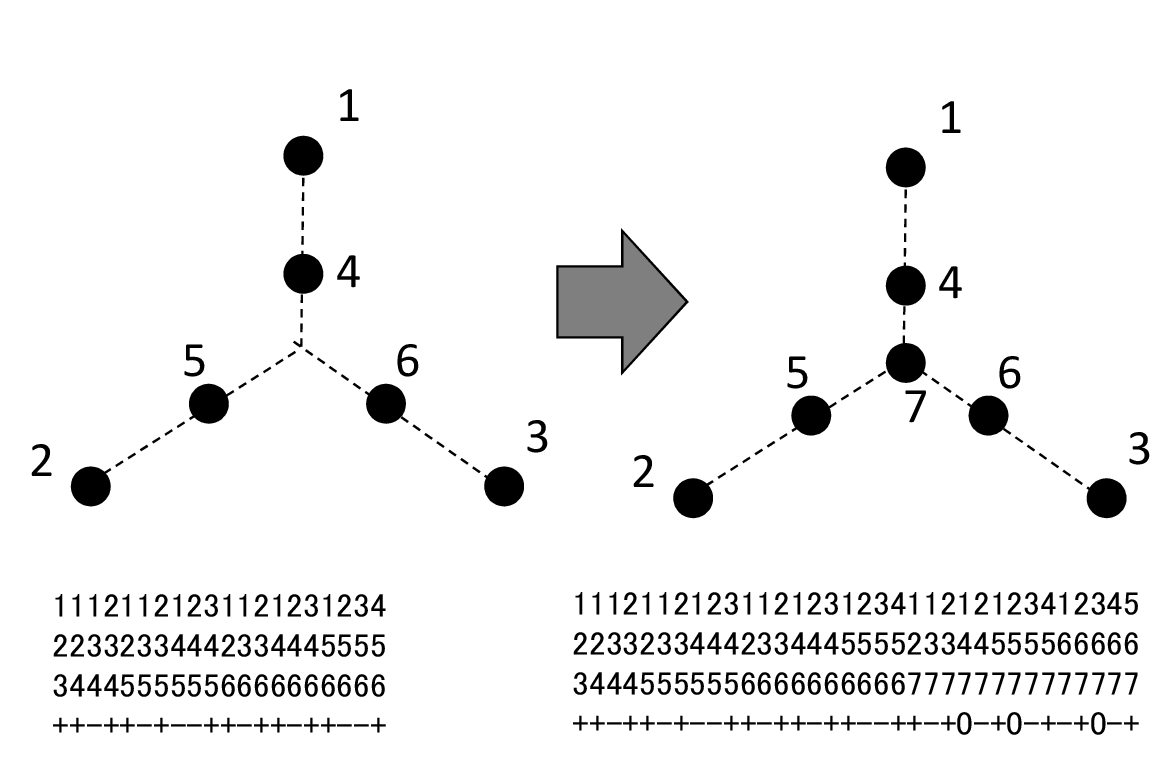}
\includegraphics[scale=0.40,clip]{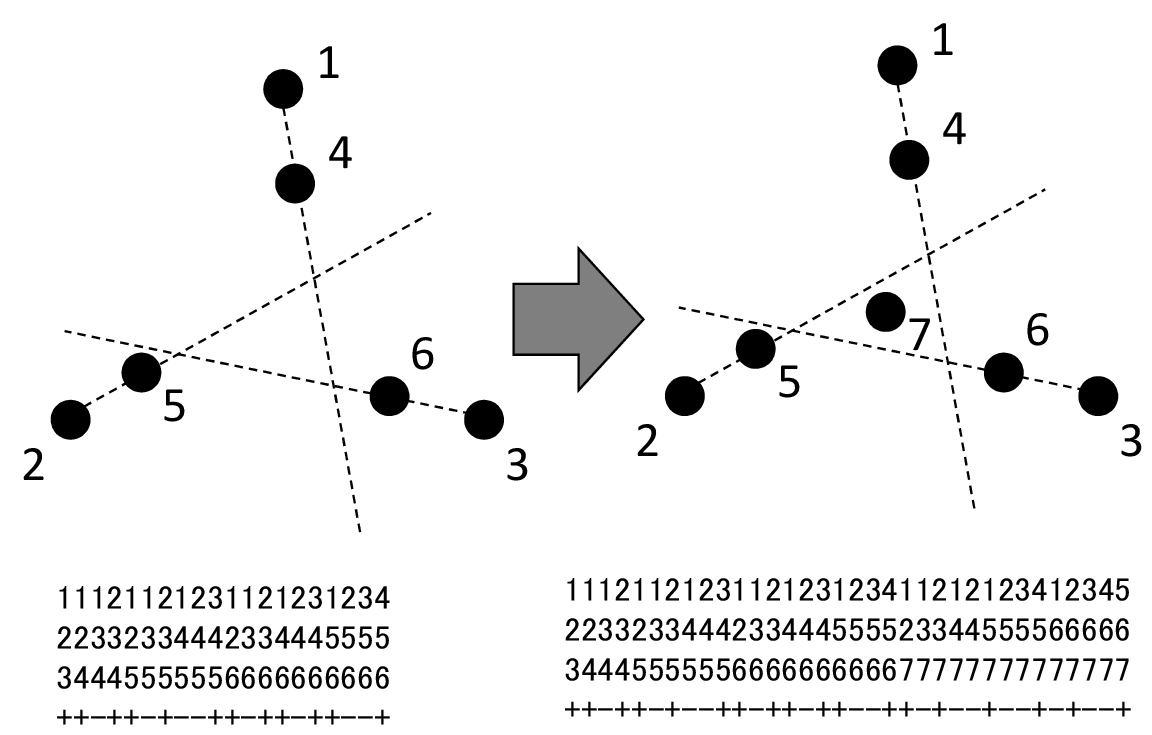}
\end{center}
\caption{Two oriented matroids  whose deletions by the fixed points $7$ are the same}
\label{fig:non_uniqueness}
\end{figure}
\\
\end{remark}

The ``uniqueness property'' for reflection symmetries is formulated as follows.
\begin{prop}
\label{cor:uniqueness2}
For a reflection symmetry $\tau$ of a rank $r$ oriented matroid ${\cal M}=(E,\{ \chi, -\chi \})$, we have
\[ {\rm rank}(\Fix(\tau)) \leq r-1. \]
Therefore, if $\sigma|_X = \tau|_X$ for $\sigma, \tau \in G({\cal M})$ and $X \subseteq E$ such that ${\rm rank}(X)=r$, then $\sigma = \tau$.
\end{prop}
\begin{proof}
For any $q_1,\dots,q_r \in \Fix(\tau)$,
we have $-\chi (q_1,\dots,q_r) = \chi (\tau (q_1),\dots,\tau(q_r))=\chi (q_1,\dots,q_r)$ and thus $\chi (q_1,\dots,q_r)=0$.  
\end{proof}
\\
We remark that $\sigma|_X = \tau|_X$ can happen for $\sigma \in R({\cal M}), \tau \in G({\cal M}) \setminus R({\cal M})$ (and thus $\sigma \neq \tau$) and $X \subseteq E$ with $\rank (X)=r-1$.
See Figure \ref{fig:order2}.
The rotational symmetry $\sigma$ and the reflection symmetry $\tau$ has the same group action on $\{ y, \sigma (y) \}$.
\begin{figure}[h]
\begin{center}
\includegraphics[bb = 37 392 558 770, scale=0.30,clip]{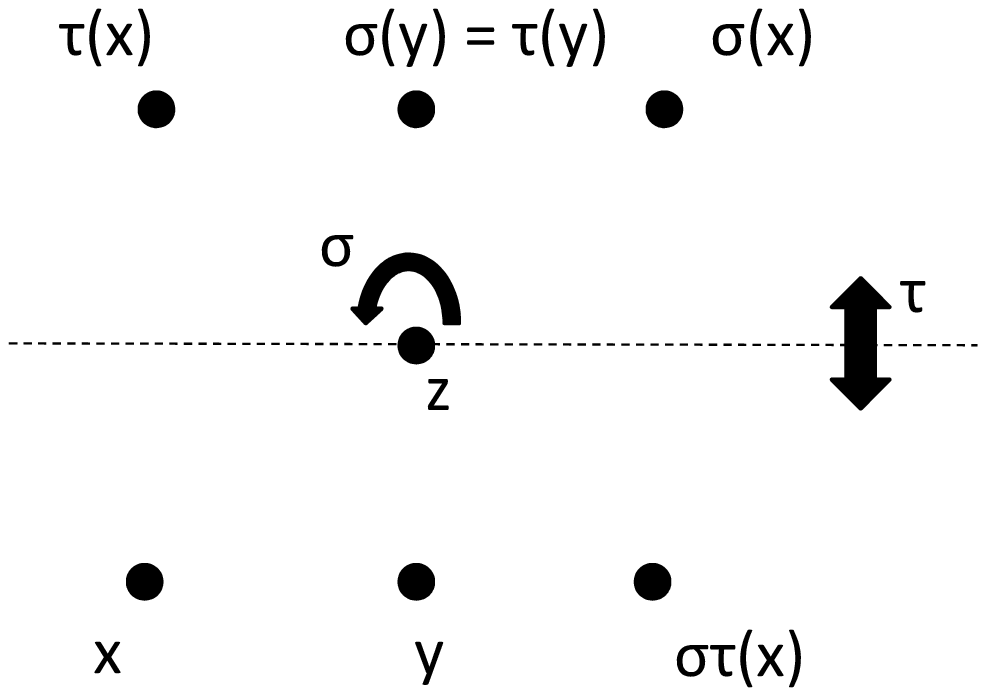}
\end{center}
\caption{Two different symmetries having the same group action on a rank $2$ subset}
\label{fig:order2}
\end{figure}

\subsection{Existence property}
Now we define existence property of fixed points of oriented-matroid symmetries.

\begin{defn}
\label{def:existence}
For a simple acyclic oriented matroid ${\cal M}$ on a ground set $E$, we say that
a subgroup $G_f({\cal M}) \subseteq G({\cal M})$ has {\it fixed-point-admitting (FPA) property}
if
there exists a single element extension ${\cal M}\cup p$ of ${\cal M}$ 
with a new point $p$
that is invariant under the permutation
${\widehat \sigma}$ on $E \cup \{ p \}$ with ${\widehat \sigma}|_E= \sigma$ and ${\widehat \sigma} (p) = p$ for every $\sigma \in G_f({\cal M})$.
\end{defn}

\begin{remark}
The oriented matroid ${\cal M} \cup p$ can be chosen to be acyclic.
If ${\cal M} \cup p$ is not acyclic, the reorientation $_{-\{ p \}}({\cal M} \cup p)$ is acyclic.
\end{remark}

\begin{remark}
For a simple acyclic oriented matroid ${\cal M}$ of  rank $r$ with a non-trivial FPA rotational symmetry  group $R_f({\cal M})$,
we can take a fixed point $q$ and can construct a rank $r-1$ oriented matroid
$({\cal M} \cup q)/q$, which is also invariant under $R_f({\cal M})$.
Let $\sigma \in R_f({\cal M})$ be a FPA rotational symmetry of ${\cal M}$ and ${\widehat \chi}$ a chirotope of ${\cal M} \cup p$.
Then, the map $\chi:E^{r-1} \rightarrow \{ +,-,0\}$ defined by
\[ \chi (i_1,\dots,i_{r-1}):= {\widehat \chi}(q,i_1,\dots,i_{r-1}) \text{ for $i_1,\dots,i_{r-1} \in E$} \]
is a chirotope of $({\cal M}\cup q)/q$.
Then, it holds that
\[ \sigma \cdot \chi (i_1,\dots,i_{r-1}) = \chi(\sigma(i_1),\dots,\sigma(i_{r-1}))=
{\widehat \chi}(q,\sigma(i_1),\dots,\sigma(i_{r-1}))={\widehat \chi}(q,i_1,\dots,i_{r-1}) = \chi(i_1,\dots,i_{r-1})\]
for all $i_1,\dots,i_{r-1} \in E$. Therefore, $\sigma$ is a rotational symmetry of $({\cal M}\cup q)/q$.
In this way, we can decrease the rank of an oriented matroid by one preserving rotational symmetries
(acyclicity and simpleness are not always preserved).
A similar discussion applies to the case of FPA subgroups of full symmetry groups.
\end{remark}

\begin{prop}
\label{prop:FP_inside}
Let ${\cal M}$ be an acyclic oriented matroid on a ground set $E$ with a FPA symmetry group $G_f({\cal M})$,
$p$ a fixed point of $G_f({\cal M})$,
and $O \subseteq E$ a  $G_f({\cal M})$-orbit.
If ${\cal M}|_{O \cup \{ p \}}$ is acyclic and $\rank(O \cup \{ p \} ) = \rank(O) \geq 2$, 
then $p$ must be {\it inside} $O$, 
i.e., $X(p)=+$ for any cocircuit $X$ of ${\cal M}|_{O \cup \{ p \}}$ such that $X|_O \geq 0$.
Therefore, the contraction ${\cal M}|_{O \cup \{ p \}}/p$ is not acyclic (and thus totally cyclic by transitivity).
\end{prop}
\begin{proof}
Let ${\cal N} := ({\cal M} \cup p)|_{O \cup \{ p \}}$.
First, suppose that there exists a cocircuit $X$ of ${\cal N}$
with $X(p)=-$ and $X|_O \geq 0$.
Then, the fixed point $p$ is an extreme point of ${\cal N}$ 
(since ``in an acyclic oriented matroid any non-empty half-space contains an extreme point.'' See \cite[Proposition 1.2]{LV80}).
Let us take an element of $O$ arbitrarily. Then it is also an extreme point of ${\cal N}$.
Otherwise, none of the elements of $O$ is an extreme point of ${\cal N}$ by transitivity.
Since ${\cal N}$ is acyclic, there must be at least $\rank_{\cal N}(O \cup \{ p \})$ extreme points (\cite[Theorem 1.3]{LV80}), which is a contradiction.
Therefore, all elements of $O \cup \{ p \}$ are extreme points of ${\cal N}$.
This means that ${\cal N}$ is a matroid polytope and thus that we have $Y(p) \geq 0$ for every cocircuit $Y$ with $Y|_O \geq 0$, which is a contradiction. 
Therefore, it holds that $X(p) \geq 0$ for every cocircuit $X$ of ${\cal N}$ such that $X|_O \geq 0$.

Next, suppose that there exists a cocircuit $X^*(p)=0$ and $X^*|_O \geq 0$.
If $X^*|_O = 0$, then $\rank_{\cal N}(O \cup \{ p \}) = \rank_{\cal N}(O)-1$, a contradiction.
Thus we can take $z \in O$ such that $X^*(z) = +$.
By transitivity, it holds that $\bigcap_{\tau \in G_f({\cal M})}\tau((X^*)^0) = \{ p \}$.
Therefore, the fixed point $p$ is an extreme point of ${\cal N}$ and thus ${\cal N}$ is a matroid polytope.
Let $F := (X^*)^0$ and consider ${\cal N}' = {\cal N}|_F$.
Note that $\rank_{{\cal N}'} (F) = \rank_{{\cal N}'}  (F \setminus \{ p \})$.
Since every non-negative covector is a composition of some non-negative cocircuits (every face is an intersection of some facets) for a matroid polytope (for a proof, see \cite[p.3]{LV80}),
 it holds that $Y(p) \geq 0$ for every cocircuit $Y$ of ${\cal N}$' such that $Y|_{F \setminus \{ p \}} \geq 0$.
We can continue this discussion until we obtain a matroid polytope ${\widetilde {\cal N}}$ on a ground set ${\widetilde E}$ such that
${\widetilde Y}(p)=+$ for every cocircuit ${\widetilde Y}|_{{\widetilde E} \setminus \{ p \}} \geq 0$.
If $\rank({\widetilde {\cal N}}) > 1$, it contradicts to the observation that $p$ is an extreme point.
If $\rank({\widetilde {\cal N}}) = 1$, it contradicts to the assumption that ${\cal M}$ is simple.

As a conclusion, we have $X(p)=+$ for every cocircuit $X$ of ${\cal N}$ such that $X|_O \geq 0$.
Therefore, there is a vector $C$ of ${\cal N}$ such that $C^+ = O$ and $C^- = \{ p \}$.
(The vector $C$ is orthogonal to any cocircuit $V$ of ${\cal N}$ such that $(V|_O)^+ \neq \emptyset$ and $(V|_O)^- \neq \emptyset$.
$C$ is also orthogonal to any cocircuit $W$ of ${\cal N}$ such that $W|_O \geq 0$ since we have  $W(p) = +$.) 
This implies the existence of the positive vector of ${\cal N}/p$.
\end{proof}

In this paper, we investigate structure of FPA rotational symmetry groups.
Note that the geometric rotational symmetry group of any point configuration has FPA property.
We conjecture that the symmetry group of every simple matroid polytope also has FPA property.

\begin{conj}
For every simple matroid polytope ${\cal M}$,
the symmetry group $G({\cal M})$ has FPA property. 
\end{conj}

\begin{remark} 
Let us assume that the conjecture is true and
consider a simple non-acyclic oriented matroid ${\cal M}$ of rank $r$ on a ground set $E$
where $R({\cal M})$ acts  transitively on $E$.
Then, the oriented matroid ${\cal M}$ is totally cyclic and thus the dual oriented matroid ${\cal M}^*$ is acyclic.
Since $R({\cal M^*}) (=R({\cal M}))$ acts transitively on $E$, the dual ${\cal M}^*$ is a matroid polytope of rank $|E|-r$.
Then, let us take a fixed point $q$ and consider a single element extension ${\cal M}^* \cup q$.
Note that we have $R({\cal M}^*)=R(({\cal M}^* \cup q)/q)$
and that the contraction $({\cal M}^*\cup q)/q$ is cyclic and thus totally cyclic by transitivity.
This leads to that the dual oriented matroid ${\cal N}:=(({\cal M}^* \cup q)/q)^*$ is acyclic.
${\cal N}$ is a rank $r+1$ matroid polytope by transitivity.
The same applies to $G({\cal M})$.
\end{remark}
The above remark shows that classification of (rotational) symmetry groups of non-acyclic oriented matroids of rank $r$
is closely related to that of acyclic oriented matroids of rank $r+1$ under the conjecture.

\begin{remark}
The assumption that ${\cal M}$ is acyclic is necessary.
Consider the case that ${\cal M} = {_{-\{ 2,4\}}A_{3,4}}$. 
Then ${\cal M}$ has rotational symmetries $\sigma = \begin{pmatrix} 1 & 2 & 3 & 4 \\ 2 & 3 & 1 & 4 \end{pmatrix}$ and
$\tau = \begin{pmatrix} 1 & 2 & 3 & 4 \\ 1 & 3 & 4 & 2 \end{pmatrix}$.
Note that $\Fix(\sigma) = \{ 4\}$ and that $\Fix(\tau)= \{ 1 \}$, and that
any (simple) single element extension of ${\cal M}$ does not have other fixed points of $\sigma$ and $\tau$ by Theorem \ref{prop:uniqueness}.
\end{remark}

\section{Symmetry groups of simple oriented matroids of rank $2$}
In this section, we classify rotational and full symmetry groups of simple oriented matroids of rank $2$.
The strategy, which will also be used later, is as follows.
For a simple oriented matroid ${\cal M}$ of rank $2$, let $X$ be a rank $2$ $G({\cal M})$-orbit.
Since the group structure of $G({\cal M})$ is determined by relations among each element of $G({\cal M})$,
the group structure can be understood by the action of $G({\cal M})$ on $X$ by Corollaries \ref{cor:uniqueness} and \ref{cor:uniqueness2}
 (for example, $\sigma|_X \tau|_X = \tau|_X \sigma|_X$ leads to $\sigma \tau = \tau \sigma$ for $\sigma, \tau \in G({\cal M})$).
Therefore, our first goal is to understand structure of  $G({\cal M})$-orbits.
Once we know orbit structure of their symmetry groups, then symmetry groups of oriented matroids are ones of their subgroups.
The same approach also works for $R({\cal M})$.

\begin{prop}
\label{prop:classification_rank2_rotation}
Let ${\cal M}=(E, \{ \chi, -\chi \})$ be a simple oriented matroid of rank $2$.
Then, $R({\cal M})$ is isomorphic to the cyclic group $\mathbb{Z}_{2p-1}$ for some $p \in \mathbb{N}$.
An $R({\cal M})$-orbit is isomorphic to $_{-\{ 2,4,\dots,2p-2\}}A_{2,2p-1}$ if $p \geq 2$ and to $A_{1,1}$ if $p=1$.
\end{prop}
\begin{proof}
If $R({\cal M}) \simeq \{ \id \}$, the proposition is clearly true and thus we assume $|R({\cal M})| \geq 2$.
Without loss of generality, we consider the case where $R({\cal M})$ acts transitively on $E$.
First, note that $_{-A}{\cal M} \simeq A_{2,|E|}$ for some $A \subseteq E$ and that 
the inseparability graph $IG({\cal M})$ is the $|E|$-cycle (see Theorem \ref{thm:inseparability}).
Therefore, $R({\cal M})$ is a subgroup of the dihedral group $D_{2|E|}$.
By what is known as Cavior's theorem, any subgroup of a dihedral group is a cyclic group or a dihedral group and this leads to that 
$R({\cal M})$ must contain a subgroup that is isomorphic to the cyclic group $\mathbb{Z}_{|E|}$ or $D_{|E|}$ (only when $|E|$ is even).

First, we assume that $R({\cal M})$ contains a subgroup that is isomorphic to $\mathbb{Z}_{|E|}$, i.e.,  $R({\cal M}) \simeq \mathbb{Z}_{|E|}$ or  $R({\cal M}) \simeq D_{2|E|}$.
Take $x \in E$ and let $\sigma \in R({\cal M})$ be the symmetry that corresponds to the smallest-angle rotation of $IG({\cal M})$. 
Then, the elements $x,\sigma (x),\dots,\sigma^{|E|-1}(x)$ form an alternating matroid in this order, i.e.,
$_{-A}\chi(\sigma^i(x),\sigma^j(x)) = +$ for all $i,j \in \mathbb{Z}$ with $0 \leq i < j < |E|$.
Here, note that $|E|$ must be odd. Otherwise, it holds that $\chi (x,\sigma^{\frac{|E|}{2}} (x)) = \sigma^{\frac{|E|}{2}} \cdot \chi (x,\sigma^{\frac{|E|}{2}} (x)) =  \chi (\sigma^{\frac{|E|}{2}} (x),x)$, which is a contradiction.

We have
\begin{align*}
(-1)^{|A \cap \{ \sigma^i(x),\sigma^j (x)\}|} = \chi (\sigma^i(x),\sigma^j (x))= \chi (x,\sigma^{j-i} (x)) =  (-1)^{|A \cap \{ x,\sigma^{j-i} (x)\}|}
\end{align*}
for any $i,j \in \mathbb{Z}$ with $0 \leq i < j < |E|$.
It follows that $|A \cap \{ \sigma^i(x),\sigma^j(x) \}|=|A \cap \{ x,\sigma^{j-i}(x) \}|$ for any $i,j \in \mathbb{Z}$ with $0 \leq i < j < |E|$.
Therefore, if $x,\sigma (x) \in A$, then we have $2=|A \cap \{ x, \sigma (x) \}| = |A \cap \{ \sigma (x), \sigma^2 (x) \}|$ and thus
$\sigma^2(x) \in A$. Continuing this discussion, we obtain $A=E$ and thus $+=\chi (x,\sigma (x)) = \chi (\sigma^{|E|-1}(x),x)=-$, which is a contradiction.
A contradiction is similarly obtained if $x,\sigma (x) \notin A$.
If $x \notin A$ and $\sigma (x) \in A$, we have $1 =|A \cap \{ x, \sigma (x) \}| = |A \cap \{ \sigma (x), \sigma^2 (x) \}|$
and thus $\sigma^2(x) \in A$.  Continuing this discussion, we have $A=\{ \sigma^{2k+1}(x) \mid k \in \mathbb{Z}, 0 \leq 2k+1 < |E| \}$.
If $x \in A$ and $\sigma (x) \notin A$, then the same argument leads to $A=\{ \sigma^{2k}(x) \mid k \in \mathbb{Z}, 0 \leq 2k < |E| \}$.
Therefore, $R({\cal M})$ is isomorphic to $\mathbb{Z}_{|E|}$.

If $R({\cal M}) \simeq D_{|E|}$, there exist $\sigma, \tau \in R({\cal M})$  such that 
$\sigma$ and $\tau$ correspond to a rotation and a reflection of $IG({\cal M})$ respectively and
$x$,$\tau \sigma^{-1}(x)$,$\sigma(x)$,$\tau \sigma^{-2}(x)$,$\sigma^2(x)$,$\dots$,$\tau (x)$ form a cycle  ($=IG({\cal M})$) in this order.
Take $A \subseteq E$ such that  $_{-A}{\cal M} \simeq A_{2,|E|}$.
From the above discussion, we have $|E|=4q-2$ for some $q \in \mathbb{N}$, and
$A \cap \{ \sigma^k (x) \mid k \in \mathbb{Z}, \ 0 \leq k < 2q-1 \} = \{ \sigma^{2k} (x) \mid k \in \mathbb{Z}, \ 0 \leq 2k < 2q-1 \}$ or 
$A \cap \{ \sigma^k (x) \mid k \in \mathbb{Z}, \ 0 \leq k < 2q-1 \} = \{ \sigma^{2k+1} (x) \mid k \in \mathbb{Z}, \ 0 < 2k+1 \leq 2q-1 \}$.
Let us  assume without loss of generality that  
$A \cap \{ \sigma^k (x) \mid k \in \mathbb{Z}, \ 0 \leq k < 2q-1 \} = \{ \sigma^{2k} (x) \mid k \in \mathbb{Z}, \ 0 \leq 2k < 2q-1 \}$.
If $\tau \sigma^{-1}(x) \in A$, we have  
$A \cap \{ \tau\sigma^{-k} (x) \mid k \in \mathbb{Z}, \ 0 \leq k < 2q-1 \} = \{ \tau\sigma^{-2k-1} (x) \mid k \in \mathbb{Z}, \ 0 < 2k+1 \leq 2q-1 \}$.
Then, $+ = \chi (x,\tau \sigma^{-1} (x)) = \tau \cdot \chi (x,\tau \sigma^{-1} (x)) = \chi (\tau (x),\sigma^{2q-2} (x)) = -$, which is a contradiction.
If $\tau \sigma^{-1}(x) \notin A$, then we have 
$A \cap \{ \tau\sigma^{-k} (x) \mid k \in \mathbb{Z}, \ 0 \leq k < 2q-1 \} = \{ \tau\sigma^{-2k} (x) \mid k \in \mathbb{Z}, \ 0 \leq 2k < 2q-1 \}$.
Then, it follows that $- = \chi (x,\tau \sigma^{-1} (x)) = \tau \cdot \chi (x,\tau \sigma^{-1} (x)) = \chi (\tau (x),\sigma^{2q-2} (x)) = +$, which is a contradiction.
Therefore, we have $R({\cal M}) \not\simeq D_{|E|}$.
\end{proof}
\\
Since $R({_{-\{ 2,4,\dots,2p-2\}}A_{2,2p-1}}) \simeq \mathbb{Z}_{2p-1}$ ($p \geq 2$) and $R(A_{2,3}) \simeq \{ \id \}$,
the cyclic group $\mathbb{Z}_{2p-1}$,  for each $p \in \mathbb{N}$, is indeed the rotational symmetry group of a simple oriented matroid of rank $2$.

\begin{prop}
\label{prop:classification_rank2_full}
Let ${\cal M}=(E, \{ \chi, -\chi \})$ be a simple oriented matroid of rank $2$.
\begin{itemize}
\item The order of $\tau \in G({\cal M}) \setminus R({\cal M})$ must be $2$.
\item $G({\cal M})$ is a dihedral group.
A $G({\cal M})$-orbit is isomorphic to $A_{1,1}$ or $A_{2,2}$, or  $_{-\{ 2,4,\dots,2p-2\}}A_{2,2p-1}$ for some $p \geq 2$.
\end{itemize}
\end{prop}
\begin{proof}
If $|G({\cal M})| \leq 2$, i.e., $G({\cal M}) \simeq \{ \id \}$ or $G({\cal M}) \simeq \mathbb{Z}_2 (\simeq D_2)$, the proposition is clearly true and thus we assume that $|G({\cal M})| \geq 3$.
Without loss of generality, we consider the case where $G({\cal M})$ acts transitively on $E$.
The same discussion as the proof of Proposition \ref{prop:classification_rank2_rotation} leads to that $G({\cal M})$ is isomorphic to one of  $\mathbb{Z}_{|E|}$, $D_{2|E|}$ and $D_{|E|}$ (only when $|E|$ is even).
Let us first consider the case $G({\cal M}) \simeq \mathbb{Z}_{|E|}$ or $G({\cal M}) \simeq D_{2|E|}$.
Let $\sigma \in G({\cal M})$ be the symmetry that corresponds to the smallest-angle rotation of $IG({\cal M})$.
If $\sigma \in R({\cal M})$, then we have ${\cal M} \simeq {_{-\{ 2,4,\dots,2p-2\}}A_{2,2p-1}}$ for some $p \geq 2$ by Proposition \ref{prop:classification_rank2_rotation}.
Therefore, we have $G({\cal M}) \simeq D_{4p-2}$.
If $\sigma \in G({\cal M}) \setminus R({\cal M})$, the square $\sigma^2$ is a rotational symmetry and thus $|E|=4q-2$ for some $q \geq 2$.
We have $_{-A}{\cal M} \simeq A_{2,4q-2}$ for some $A \subseteq E$ and $_{-A}\chi (\sigma^i(x), \sigma^j (x)) = +$ for all $i,j \in \mathbb{Z}$ with $0 \leq i < j < 4q-2$.
Suppose now that $x, \sigma (x) \in A$. Then, it holds that $- = -\chi(x,\sigma (x)) = \sigma \cdot \chi(x,\sigma (x)) = \chi(\sigma(x),\sigma^2 (x)) = (-1)^{|A \cap \{ \sigma (x), \sigma^2(x) \}|}$
and thus $\sigma^2 (x) \notin A$. Continuing this discussion, we have 
$A=\{ \sigma^{4k}(x) \mid k \in \mathbb{Z}, 0 \leq k < q \} \cup \{ \sigma^{4k+1}(x) \mid k \in \mathbb{Z}, 0 \leq k < q \}$.
By a similar discussion, $A$ must be one of the following types:
\begin{itemize}
\item[(a)] $A=\{ \sigma^{4k}(x) \mid k \in \mathbb{Z}, 0 \leq k < q \} \cup \{ \sigma^{4k+1}(x) \mid k \in \mathbb{Z}, 0 \leq k < q \}$,
\item[(b)] $A=\{ \sigma^{4k+2}(x) \mid k \in \mathbb{Z}, 0 \leq k < q \} \cup \{ \sigma^{4k+3}(x) \mid k \in \mathbb{Z}, 0 \leq k < q-1 \}$,
\item[(c)] $A=\{ \sigma^{4k}(x) \mid k \in \mathbb{Z}, 0 \leq k < q \} \cup \{ \sigma^{4k+3}(x) \mid k \in \mathbb{Z}, 0 \leq k < q-1 \}$,
\item[(d)] $A=\{ \sigma^{4k+1}(x) \mid k \in \mathbb{Z}, 0 \leq k < q \} \cup \{ \sigma^{4k+2}(x) \mid k \in \mathbb{Z}, 0 \leq k < q \}$.
\end{itemize}
If $A$ is of Type (a), we have  $-=\chi(x,\sigma^2 (x)) =  - \sigma \cdot \chi(x,\sigma^2 (x)) = - \chi(\sigma (x),\sigma^3 (x))= +$, which is a contradiction.
Types (b), (c) and (d) are also impossible.
Therefore, the case $\sigma \in G({\cal M}) \setminus R({\cal M})$ never happens.

Next, we consider the case $G({\cal M}) \simeq D_{|E|}$. In this case, we have $|E|=2r$ for some $r \geq 2$ and $R({\cal M}) \simeq \mathbb{Z}_r$.
Therefore, we have $r=2q-1$ for some $q \in \mathbb{N}$ by Proposition \ref{prop:classification_rank2_rotation}.
We can take $\sigma \in R({\cal M}), \tau \in G({\cal M}) \setminus R({\cal M})$ and $x \in E$ such that
$x$,$\tau \sigma^{-1}(x)$,$\sigma(x)$,$\tau \sigma^{-2}(x)$,$\sigma^2(x)$,$\dots$,$\tau (x)$ form a cycle ($= IG({\cal M})$) in this order.
There exists $A \subseteq E$ such that $_{-A}{\cal M} \simeq A_{2,4q-2}$.
Without loss of generality, we assume that $A \cap \{ \sigma^k (x) \mid k \in \mathbb{Z}, \ 0 \leq k < 4q-2 \} = \{ \sigma^{2k} (x) \mid k \in \mathbb{Z}, \ 0 \leq 2k < 4q-2 \}$.
Suppose that $\tau \sigma^{-1}(x) \in A$.
Then, we have $A \cap \{ \tau \sigma^{-k} (x) \mid k \in \mathbb{Z}, \ 0 \leq k < 4q-2 \} = \{ \tau\sigma^{-2k-1} (x) \mid k \in \mathbb{Z}, \ 0 \leq 2k+1 < 4q-2 \}$.
If $q > 1$, this leads to that $- = \chi (x,\sigma (x)) = - \tau \cdot \chi (x,\sigma (x)) = \chi (\tau (x), \tau \sigma^{-(2q-2)}(x)) = +$, which is a contradiction.
If $\tau \sigma^{-1}(x) \notin A$ and $q > 1$, we have  
 $A \cap \{ \tau \sigma^{-k} (x) \mid k \in \mathbb{Z}, \ 0 \leq k < 4q-2 \} = \{ \tau\sigma^{-2k} (x) \mid k \in \mathbb{Z}, \ 0 \leq 2k < 4q-2 \}$.
This also yields that $- = \chi (x,\sigma (x)) = \chi (\tau (x), \tau \sigma^{-(2q-2)}(x)) = +$,  a contradiction.
Hence, the case $G({\cal M}) \simeq D_{|E|}$ never happens.
\end{proof}
\\
We note here that a simple acyclic oriented matroid rank $2$ cannot have a rotational symmetry of order $p>2$.

\section{Symmetry groups of simple oriented matroids of rank $3$}
In this section, we investigate symmetry groups of rank $3$ oriented matroids.
The strategy is, as the rank $2$ case, to investigate structure of orbits under the actions of rotational and full symmetry groups.

Orbit structure in the acyclic case is easy to understand.
Let ${\cal M}$ be a simple acyclic oriented matroid of rank $3$ on a ground set $E$ on which $R({\cal M})$ acts transitively.
If $|R({\cal M})| > 2$, then ${\cal M}$ is a matroid polytope of rank $3$ and thus is isomorphic to the alternating matroid $A_{3,|E|}$ (Proposition \ref{prop:rank_relabeling} in Appendix 1).
Therefore, the rotational symmetry group $R({\cal M})$ is a cyclic group of order $|E|$.
This yields that rotational symmetry groups of simple acyclic oriented matroids of rank $3$ are cyclic groups $\mathbb{Z}_n$ ($n \geq 1$).
Full symmetry groups are dihedral groups $D_{2n}$ ($n \geq 0$), where $D_0 := \{ \id \}$.

In the remaining part of this section, we will classify rotational and full symmetry groups of simple oriented matroids of rank $3$, without assuming acyclicity.
Roughly speaking, we can classify situations into three cases.
\begin{itemize}
\item There exists a {\it non-proper rotational symmetry} of order $p>2$
(Let us call a rotational symmetry of an  oriented matroid of rank $3$ a {\it proper  rotational symmetry} if its order is $2$ or there is a rank $3$ orbit under the cyclic group generated by $\sigma$).
This case will be studied in Section \ref{sec:non_proper}.
\item There is a proper rotational symmetry of order $p>2$ (Sections \ref{sec:proper1} and \ref{sec:proper2}).
\item Every rotational symmetry has order $1$ or $2$ (Section \ref{sec:proper3}).
\end{itemize}

\subsection{Structure of orbits under the actions of rotational symmetry groups that contain non-proper rotational symmetries}
\label{sec:non_proper}
We first study the condition when a simple oriented matroid of rank $3$ has a non-proper rotational symmetry.

\begin{prop}
\label{prop:non_proper}
Let ${\cal M}=(E,\{ \chi, -\chi \})$ be a simple oriented matroid of rank $3$.
Let $\sigma \in R({\cal M})$ be a non-proper rotational symmetry of order $2p-1 \geq 3$ and $X:=\{ x, \sigma (x), \dots, \sigma^{2p-2}(x) \}$, 
$Y:=\{ y, \sigma (y), \dots, \sigma^{2p-2}(y) \}$ be orbits under the action of the cyclic group generated by $\sigma$. 
If $\rank(X \cup Y) = 3$, then $|X|=1$ or $|Y|=1$.
\end{prop}
\begin{proof}
If $|X| \geq 2$, then $|X|=2p-1$ by simplicity (and Corollary \ref{cor:uniqueness}). The same applies to $Y$.
Without loss of generality, we assume that $|X|=2p-1 \geq 3$ and that the elements $x,\sigma (x), \dots, \sigma^{2p-1}(x)$ form the oriented matroid $_{- \{ 2,4,\dots,2p-2\}}A_{2,2p-1}$
in this order. 

Now suppose that $|Y| = 2p-1$.
Let $y \in Y$ be such that $\rank(X \cup \{ y \})=3$
(this leads to that $\rank(X \cup \{ \sigma^k(y) \}) = 3$ for any $k \in \mathbb{Z}$).
Let us write $A:= \{ \sigma (x),\sigma^3 (x), \dots, \sigma^{2p-3}(x), \sigma (y),\sigma^3 (y), \dots, \sigma^{2p-3}(y) \}$ from here on.
Without loss of generality, we assume that
$\chi (\sigma^i (x), \sigma^j (x), y) = (-1)^{|A \cap \{ \sigma^i(x),\sigma^j(x) \} |}$ for all $i,j \in \mathbb{Z}$ with $0 \leq i < j < 2p-1$.
Then, simple computation yields that 
$\chi (\sigma^i (x), \sigma^j (x), \sigma^a (y)) = (-1)^{A \cap \{ \sigma^i(x),\sigma^j(x) \} |}$ for all $a, i,j \in \mathbb{Z}$ with $0 \leq i < j < 2p-1$.
Similarly, if we denote $s := \chi (y,\sigma (y), x)$, it holds that
$\chi (\sigma^i (y), \sigma^j (y), \sigma^a (x)) = s \cdot (-1)^{|A \cap \{ \sigma^i(y),\sigma^j(y) \} |+1}$ for all $a, i,j \in \mathbb{Z}$ with $0 \leq i < j < 2p-1$.
By (B3) of the chirotope axioms, we have
\begin{align*}
 \{ \chi (x,\sigma (x),\sigma^2(x)) \chi (y,\sigma (y),\sigma^2(y)),  \chi (y,\sigma (x), \sigma^2 (x)) \chi (x,\sigma (y),\sigma^2(y)), \\
 \chi (\sigma (y),\sigma (x), \sigma^2 (x)) \chi (y,x,\sigma^2 (y)), \chi (\sigma^2 (y), \sigma (x),\sigma^2 (x)) \chi (y,\sigma (y),x) \} \\
= \{ 0 \} \text{ or } \supseteq \{ +,- \}.
\end{align*}
This yields that $\{ -s, 0 \} = \{ 0 \}$ and that $s=0$, which is a contradiction.
As a conclusion, it must hold that $|Y|=1$.
The same applies to the case $|Y|=2p-1$.
\end{proof}
\\
By Proposition \ref{prop:non_proper}, if a simple oriented matroid ${\cal M}$ of rank $3$ has a non-proper rotational symmetry,
${\cal M}$ is of the form ${\cal N} \cup q$,
where ${\cal N}$ is a simple oriented matroid of rank $2$ with a rotational symmetry $\sigma$ and $q$ is a coloop.
This leads to the following proposition.
\begin{prop}
Let ${\cal M}$ be a simple oriented matroid of rank $3$.
If $R({\cal M})$ contains a non-proper rotational symmetry, then we have
\begin{itemize}
\item $R({\cal M}) \simeq \mathbb{Z}_{2p-1}$ for some $p \geq 2$. An $R({\cal M})$-orbit is isomorphic to $A_{1,1}$ or $_{-\{ 2,4,\dots,2p-2\}}A_{2,2p-1}$.
\item $G({\cal M}) \simeq D_{4p-2}$ for some $p \geq 2$. A $G({\cal M})$-orbit is isomorphic to  $A_{1,1}$ or  $_{-\{ 2,4,\dots,2p-2\}}A_{2,2p-1}$.
\end{itemize}
\end{prop}
For each $p \geq 2$, we have $R({_{-\{ 2,4,\dots,2p-2\}}A_{2,2p-1}} \cup q) \simeq \mathbb{Z}_{2p-1}$ and $G({_{-\{ 2,4,\dots,2p-2\}}A_{2,2p-1}} \cup q) \simeq D_{4p-2}$, where $q$ is a coloop.
Therefore, the cases $R({\cal M}) \simeq \mathbb{Z}_{2p-1}$ and $G({\cal M}) \simeq D_{4p-2}$ indeed happen for each $p \geq 2$.

\begin{figure}[h]
\begin{center}
\includegraphics[bb = 60 475 476 750, scale=0.30,clip]{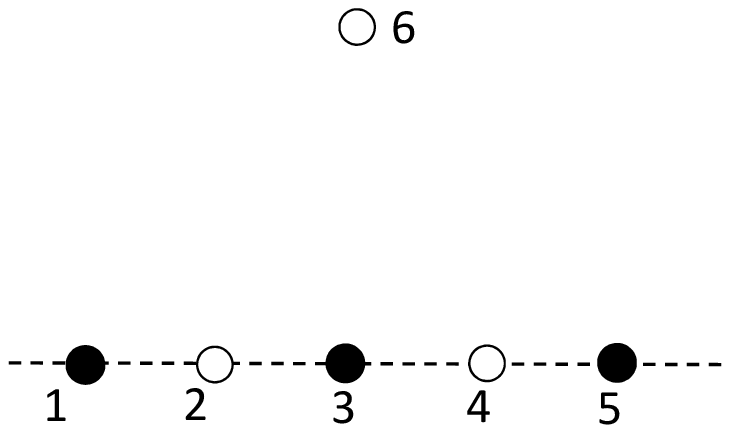}
\end{center}
\caption{A rank $3$ oriented matroid with non-proper rotational symmetries (depicted as a signed point configuration. White points are positive points and black points are negative points.
Point $6$ is a coloop)}
\label{fig:non_proper}
\end{figure}

\subsection{Structure of rank $3$ orbits under the actions of cyclic groups generated by proper rotational symmetries}
\label{sec:proper1}
Cyclic groups are groups of simplest type.
Let us first investigate orbits under the actions of cyclic groups generated by proper rotational symmetries.

\begin{prop}
\label{rank3_transitivity1}
Let ${\cal M}$ be a simple (not necessarily acyclic) oriented matroid of rank $3$ on a ground set $E$ with a non-trivial rotational symmetry $\sigma$.
If the cyclic group $G$ generated by $\sigma$ acts transitively on $E$, then we have ${\cal M} \simeq A_{3,|E|}$.
\end{prop}
\begin{proof}
It suffices to prove that ${\cal M}$ is a matroid polytope by Proposition \ref{prop:rank_relabeling} in Appendix 1.
Since the cyclic group $G$ acts transitively on $E$, the set $E$ can be written as 
$E=\{ x, \sigma (x), \sigma^2 (x), \dots, \sigma^{n-1}(x) \}$, where $n:=|E|$.
Let us simply write $i$ to describe $\sigma^i(x) \in E$ and denote by $[m]$ the set $\{ \sigma^i (x) \mid i=0,\dots,m-1 \}$ when there is no confusion.

Remark that ${\cal M}|_{[4]}$ is acyclic and thus is a matroid polytope by transitivity.
Let $\chi$ be a chirotope of ${\cal M}|_{[4]}$.
Then, we have $\chi (1,2,3) = \sigma \cdot \chi (1,2,3) = \chi (2,3,4)$.
Simple case analysis on the values of $\chi (1,2,4)$ and $\chi (1,3,4)$ shows existence of the positive covector.
We prove that if ${\cal M}|_{[m]}$ is a matroid polytope, then ${\cal M}|_{[m+1]}$ is a matroid polytope as well, for $m=4,\dots,n-1$.

\begin{quote}
{\bf Lemma:}
Let ${\cal N}$ be a simple oriented matroid of  rank $3$ on the ground set $[p]$ ($p \geq 5$) such that
${\cal N}|_{[p] \setminus \{ 1 \}}$ and ${\cal N}|_{[p] \setminus \{ p \}}$ are  matroid polytopes of rank $3$.
Then ${\cal N}$ is acyclic.
\end{quote}
(Proof of the lemma)

Assume that ${\cal N}$ is cyclic, i.e., there exists an integer $k \in [p]$ and a vector $X$ such that 
\begin{align*}
\begin{split}
&X(i_1)=X(i_2)=\dots =X(i_k)=+, \\
&X(a)=0 \text{ for all $a \in [p] \setminus \{ i_1,i_2,\dots,i_k \}$}, 
\end{split}
\end{align*}
where $i_1,\dots,i_k \in [p]$.
If $i_1,\dots,i_k > 1$, then $X|_{[p] \setminus \{ 1 \}}$ is a vector of ${\cal N}|_{[p] \setminus \{ 1 \}}$, a contradiction.
This yields that $1 \in \{ i_1,\dots,i_k\}$.
Since ${\cal N}|_{[p] \setminus \{ p \}}$ is a matroid polytope of rank $3$, i.e., a relabeling of the alternating matroid $A_{3,p-1}$ (Proposition \ref{prop:rank_relabeling} in Appendix 1), 
there exists a vector $Y$ of ${\cal N}$ such that 
\begin{align*}
\begin{split}
&Y(1)=-, Y(s)=-, Y(t)=+, Y(u)=+,\\
&Y(a)=0 \text{ for all $a \in [p] \setminus \{ 1,s,t,u\}$},
\end{split}
\end{align*}
for distinct numbers $s,t,u \in [p] \setminus \{ 1, p \}$ (recall Proposition \ref{prop:alter_circuits}).
By applying vector elimination to $X,Y$ and $1$, we obtain a vector $Z$ of ${\cal N}$
such that 
\begin{align*}
\begin{split}
&Z(1)=0, Z(s) \in \{ +,-,0\},\\
&Z(a) \geq 0 \text{ for all $a \in [p] \setminus \{ 1,s\}$}.
\end{split}
\end{align*}
$Z|_{[p] \setminus \{ 1 \}}$ is a vector of ${\cal N}|_{[p] \setminus \{ 1\}}$.
This contradicts to the assumption that ${\cal N}|_{[p] \setminus \{ 1\}}$ is a matroid polytope.
This completes the proof of the lemma.
\\
\\
Now we prove that ${\cal M}|_{[m+1]}$ is a matroid polytope.
By the above lemma, the oriented matroid ${\cal M}|_{[m+1]}$ is acyclic (note that ${\cal M}|_{[m+1] \setminus \{ 1\}}$ is a matroid polytope by transitivity).
Assume that ${\cal M}|_{[m+1]}$ is not a matroid polytope.
Then one of the following holds.
\\ \\
{\bf (i)} There exists a vector $X$ of ${\cal M}|_{[m+1]}$ such that $X(m+1)=-$ and $X(a) \geq 0$ for all $i \in [m]$.

Note that $|X^+| \geq 2$ since ${\cal M}$ is simple.
We classify the situations into the following two cases.
\\
\\
{\bf (i-a)} $X^+$ is a facet of ${\cal M}|_{[m]}$ (only when $|X^+|=2$).

Let $\widehat{V}$ be the cocircuit of ${\cal M}$ such that $\widehat{V}^0 \supseteq X^+$ and $\widehat{V}^+ \supseteq [m] \setminus X^+$, and
$\widehat {X}$ be the vector of ${\cal M}$ such that $\widehat{X}|_{[m+1]} = X$ and $\widehat{X}^0 \supseteq [n] \setminus [m+1]$.
Since $\widehat{V} \perp \widehat{X}$, we have $\widehat{V}(m+1) = 0$. 
On the other hand, the relation $\sigma (\widehat{X}) \perp \widehat{V}$ leads to $\widehat{V}(m+2) = +$.
Continuing this discussion, we have $\widehat{V}(k) = +$ for all $k \in [n] \setminus [m+1]$.
This yields that the composition $\widehat{V} \circ \sigma (\widehat{V}) \circ \dots \circ \sigma^{n-1}(\widehat{V})$ is the positive covector, 
and thus ${\cal M}$ is acyclic.
Because of transitivity, this yields that ${\cal M}|_{[m+1]} \simeq A_{3,m+1}$. This is a contradiction.
\\
\\
{\bf (i-b)} $X^+$ is not a facet of ${\cal M}|_{[m]}$.

Since ${\cal M}|_{[m]} \simeq A_{3,m}$,
${\cal M}$ has a circuit $C_x$ such that $|C_x|=4$ and $C_x^- \subseteq X^+$, and 
$C_x^+ \supseteq \{ x \}$  for each $x \in [m] \setminus X^+$ (recall Proposition \ref{prop:alter_circuits}).
Considering the composition of $\widehat{X}$ and the circuits $C_x$ appropriately,
we obtain the vector $Y$ of ${\cal M}|_{[m+1]}$ such that $Y(m+1)=-$ and $Y(a)=+$ for all $i \in [m]$.
Therefore, ${\cal M}$ has the vector $\widehat{Y}$ such that 
\begin{align*}
\begin{split}
&\widehat{Y}(m+1)=-, \\
&\widehat{Y}(a)=+ \text{ for all $a \in [m]$,}\\
&\widehat{Y}(b) = 0 \text{ for all $[n] \setminus [m+1]$}.
\end{split}
\end{align*}
Let us consider the vector $\widehat{Y}_1:= \sigma (\widehat{Y})$, which satisfies 
\begin{align*}
\begin{split}
&\widehat{Y}_1(m+2)=-,\\
&\widehat{Y}_1(a)=+ \text{ for all $a \in [m+1] \setminus \{ 1 \},$}\\
&\widehat{Y}_1(b) = 0 \text{ for all $b \in ([n] \setminus [m+2]) \cup \{ 1 \}$}.
\end{split}
\end{align*}
By applying vector elimination to $\widehat{Y},\widehat{Y}_1$ and $m+1$, we obtain a vector $\widehat{Z}_1$ of ${\cal M}$ such that
\begin{align*}
\begin{split}
&\widehat{Z}_1(m+2)=-,\\
&\widehat{Z}_1(a)=+ \text{ for all $a \in [m]$}, \\
&\widehat{Z}_1(b) = 0 \text{ for all $b \in ([n] \setminus [m]) \setminus \{ m+2 \}$}.
\end{split}
\end{align*} 
Let us consider the vector $\widehat{Y}_2:= \sigma^2 (\widehat{Y})$, which satisfies
\begin{align*}
\begin{split}
&\widehat{Y}_2(m+3)=-,\\
&\widehat{Y}_2(a)=+ \text{ for all $a \in [m+2] \setminus \{ 1,2 \},$}\\
&\widehat{Y}_2(b) = 0 \text{ for all $b \in ([n] \setminus [m+3]) \cup \{ 1,2 \}$}.
\end{split}
\end{align*}
Let us take a vector $\widehat{W}$ obtained by applying vector elimination to $\widehat{Y}_2, \widehat{Z}_1$ and $m+2$,
and then a vector $\widehat{Z}_2$ obtained by applying vector elimination to $\widehat{W},\widehat{Y}$ and $m+1$ satisfies
\begin{align*}
\begin{split}
&\widehat{Z}_2(m+3)=-,\\
&\widehat{Z}_2(a)=+ \text{ for all $a \in [m]$}, \\
&\widehat{Z}_2(b) = 0 \text{ for all $b \in ([n] \setminus [m]) \setminus \{ m+3 \}$}.
\end{split}
\end{align*}
By repeating the same argument, we obtain a vector $\widehat{Z}_*$ of ${\cal M}$ such that
\begin{align*}
\begin{split}
&\widehat{Z}_*(n)=-, \\
&\widehat{Z}_*(a)=+ \text{ for all $a \in [m]$},\\
&\widehat{Z}_*(a)=0 \text{ for all $a \in ([n] \setminus [m]) \setminus \{ n \}$}.
\end{split}
\end{align*}
The vector $\widehat{Z}_{**}:= \sigma (\widehat{Z}_*)$ satisfies
\begin{align*}
\begin{split}
&\widehat{Z}_{**}(1)=-,\\
&\widehat{Z}_{**}(a)=+ \text{ for all $a \in [m+1] \setminus \{ 1 \}$},\\
&\widehat{Z}_{**}(b)=0 \text{ for all $b \in [n] \setminus [m+1]$}.
\end{split}
\end{align*}
By applying vector elimination to $\widehat{Y},\widehat{Z}_{**}$ and $m+1$, we obtain a vector $\widetilde{Z}$
such that
\begin{align*}
\begin{split}
&\widetilde{Z}(1) \in \{ -,0,+\}, \\
&\widetilde{Z}(a)=+ \text{ for all $a \in [m] \setminus \{ 1 \}$,}\\
&\widetilde{Z}(b)=0 \text{ for all $b \in [n] \setminus [m]$}.
\end{split}
\end{align*}
This contradicts to the assumption that ${\cal M}|_{[m]}$ is a matroid polytope.
\\
\\
{\bf (ii)} There exists a vector $X$ of ${\cal M}|_{[m+1]}$ such that $X(k)=-$ and $X(a) \geq 0$ for all $[m+1] \setminus \{ k\}$, where $ k \in [m+1] \setminus \{ 1 \}$.

One can apply a similar argument to (i), keeping in mind that ${\cal M}|_{[m+1]}$ is acyclic.
\end{proof}

Therefore, an orbit $X$ under the cyclic subgroup generated by a proper rotational symmetry $\sigma$ of order $p$ is isomorphic to $A_{3,p}$ or $A_{1,1}$
(recall that there is no rank 2 orbits by the result of Section \ref{sec:non_proper}).
With a small modification of the proof, we have that if $\sigma$ is a reflection symmetry, then ${\cal M}|_X \simeq {_{-[p] \cap 2\mathbb{N}}A_{3,p}}$

\subsection{Structure of orbits under the actions of the rotational symmetry groups that contain only proper rotational symmetries}
\label{sec:proper2}
Now we proceed to investigating structure of orbits under the groups generated by two rotational symmetries.
We assume that $R({\cal M})$ contains a proper rotational symmetry of order $p>2$.
By the results of Section \ref{sec:non_proper}, all elements of  $R({\cal M})$ are proper rotational symmetries in this case.

\begin{prop}
\label{prop:rank3_transitivity2}
Let ${\cal M}=(E,\{ \chi, -\chi \})$ be a simple oriented matroid of  rank $3$ and
$G$ the cyclic group generated by a proper rotational symmetry $\sigma \in R({\cal M})$ of order $p >2$.
For rank $3$ $G$-orbits $X:=\{ x,\sigma(x),\dots,\sigma^{p-1}(x)\}$ and $Y:=\{ y,\sigma (y),\dots, \sigma^{p-1}(y)\}$,
it holds that $X = Y$ or that $X \cap Y = \emptyset$.

If $X \cap Y = \emptyset$, then ${\cal M}|_{X \cup Y}$ or $_{-Y}{\cal M}|_{X \cup Y}$ is acyclic.
If in addition $y = \tau (x)$ for some $\tau \in G({\cal M})$, then we have ${\cal M}|_{X \cup Y} \simeq A_{3,2|X|}$ or 
$_{-Y}{\cal M}|_{X \cup Y} \simeq A_{3,2|X|}$ (In this case, we have ${\cal M} \simeq {_{- \{ 2,4, \dots, 2|X| \}}A_{3,2|X|}}$). 
If $\sigma$ is the 1st rotational symmetry of ${\cal M}|_X$, then $\sigma$ is the 2nd rotational symmetry of  ${\cal M}|_{X \cup Y}$ or $_{-Y}{\cal M}|_{X \cup Y}$.
\end{prop}
\begin{proof} 
We prove the proposition assuming $X \neq Y$.
Without loss of generality, we suppose that $x,\sigma (x), \sigma^2(x),\dots,\sigma^{p-1}(x)$ form an alternating matroid in this order, i.e.,
$\chi (\sigma^i(x),\sigma^j(x),\sigma^k(x))=+$ for all $i,j,k \in \mathbb{Z}$ with $0 \leq i < j < k \leq p-1$ (take the negative of $\chi$ if necessary).
Note that $\sigma|_Y$ is a rotational symmetry of ${\cal M}|_Y$.
Also, note that $(\sigma|_Y)^i \neq \id$ for any $i \in [p-1]$. Otherwise, we have $\sigma^i=\id$ by Corollary \ref{cor:uniqueness}, which is a contradiction.
Therefore, there exists $l \in \mathbb{Z}$ with $1 \leq l \leq p-1$ such that $y,\sigma^l (y), \dots, \sigma^{(p-1)l}(y)$ form an alternating matroid in this order
(In this setting, $\{ x, \sigma (x)\}$ (resp.\ $\{ y, \sigma^l (y) \}$) is a facet of ${\cal M}|_X$ (resp.\ ${\cal M}|_Y$)).
In this proof, we write ${\cal N}:= {\cal M}|_{X \cup Y}$ and $x_1:=x, x_2:=\sigma(x),\dots, x_p:=\sigma^{p-1}(x)$.

First, we prove that ${\cal N}$ or $_{-Y}{\cal N}$ is acyclic under the assumption that $l=1$ or $l=p-1$.
Consider the cocircuit $V$ of ${\cal N}$ such that 
\begin{align*} 
\text{$V(x_1)=V(x_2)=0$ and  $V(e)=+$ for all $e \in X \setminus \{ x_1, x_2\}$}
\end{align*}
and write $V_0:=V,V_1:=\sigma (V),\dots, V_{p-1}:=\sigma^{p-1}(V)$.
\\
\begin{figure}[h]
\begin{center}
\includegraphics[scale=0.25]{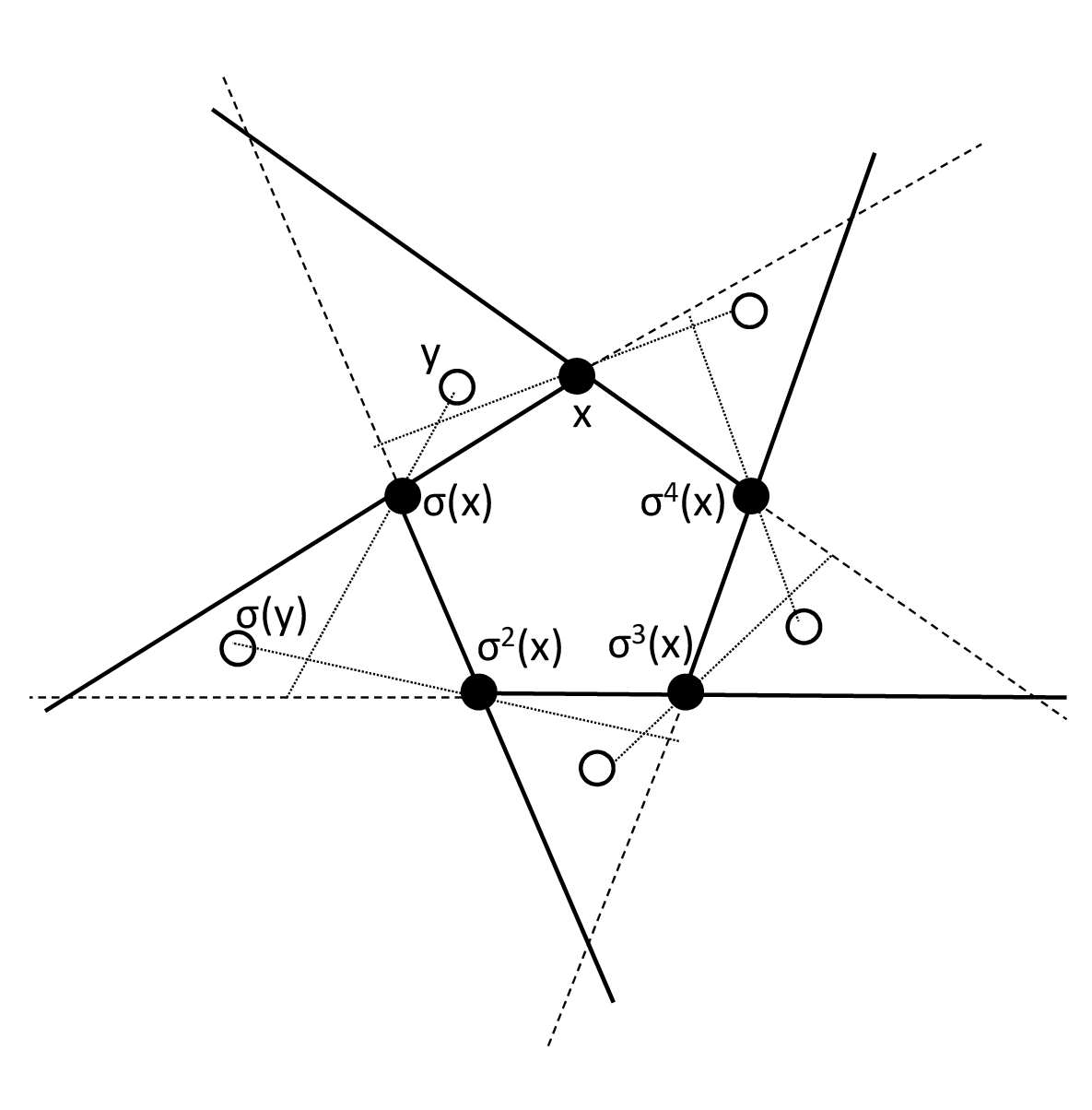}
\includegraphics[scale=0.25]{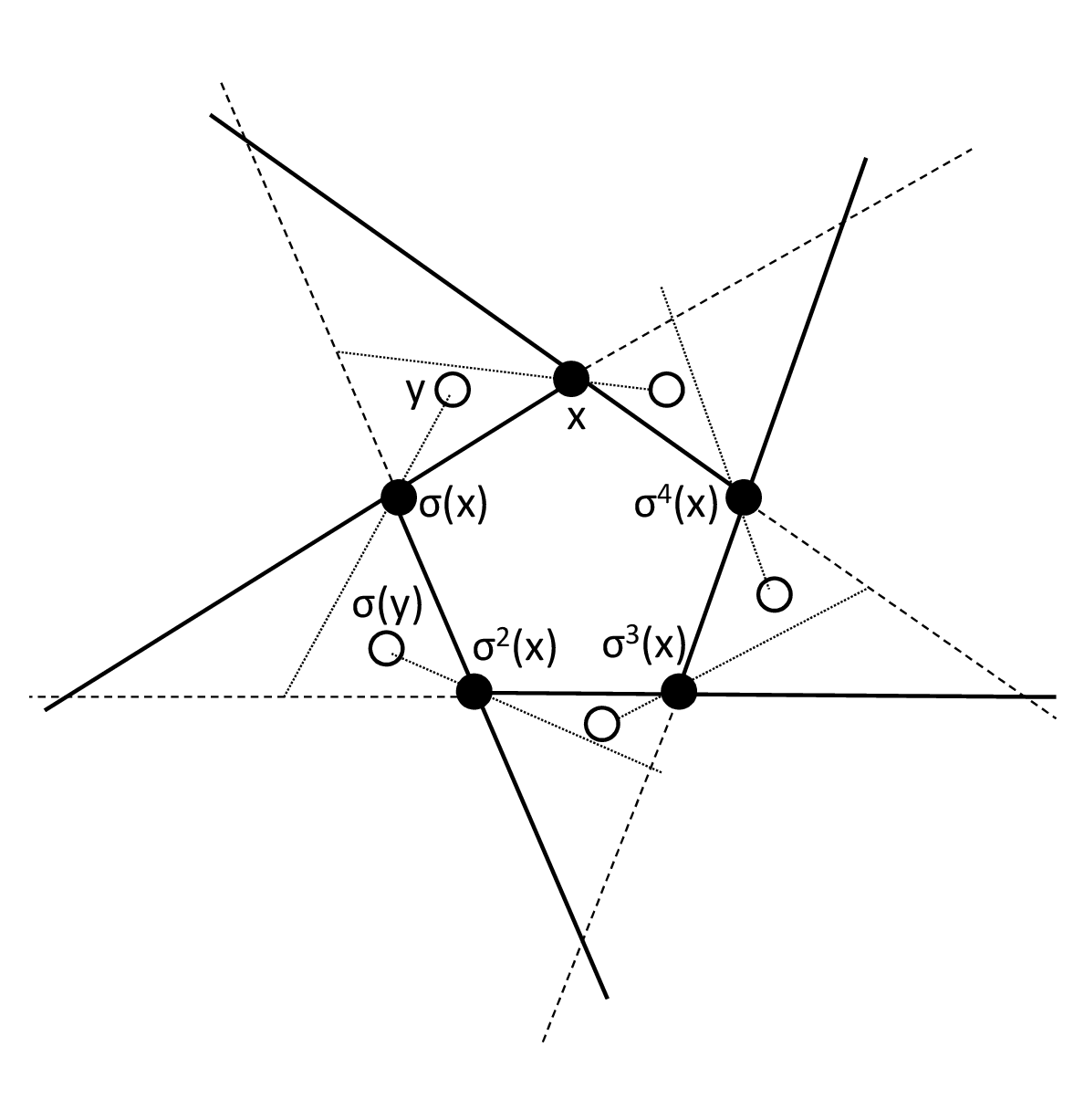}
\includegraphics[scale=0.25]{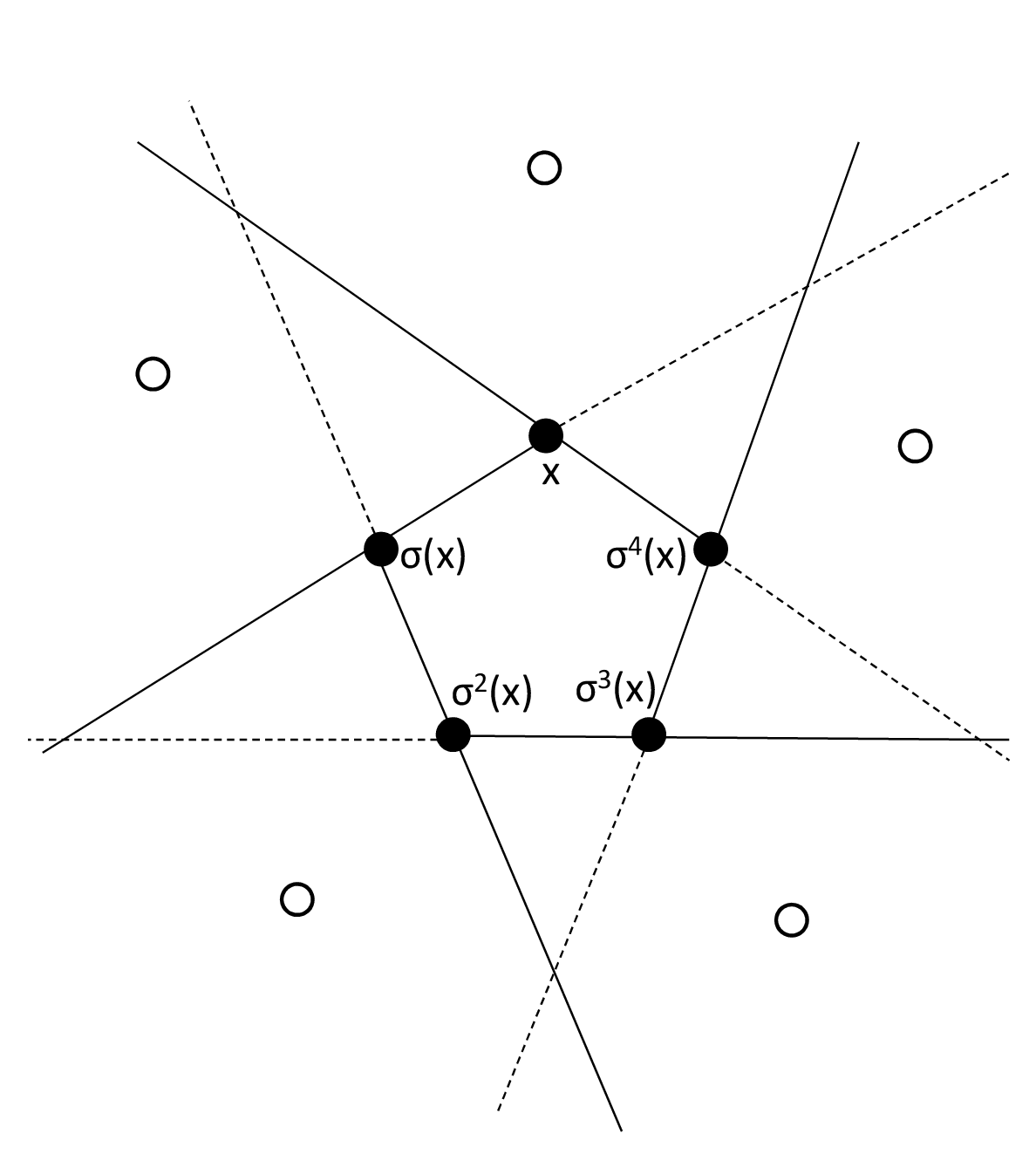}
\caption{${\cal N} = {\cal M}|_{X \cup Y}$ (depicted as signed point configurations. $l_1 = 0, l_2 = p-1$ [left, center],  $l_1 = 1, l_2 = p-1$ [right])}
\end{center}
\end{figure}
\\
By a property of alternating matroids of rank $3$ (Proposition \ref{prop:alter_ordering} in Appendix 1), 
the sign sequence $V_0(e_0)$,$V_1(e_0)$,$\dots$,$V_{p-1}(e_0)$ must be one of the following forms:
\[ +\dots +- \dots -+ \dots +, \ +\dots +0- \dots-+ \dots +, \ +\dots +- \dots-0+ \dots +, \ +\dots +0- \dots-0+ \dots +,\]
\[ -\dots -+ \dots +- \dots -, \ -\dots -0+ \dots+- \dots -, \ -\dots -+ \dots+0- \dots -, \ -\dots -0+ \dots+0- \dots -,\]
where $+\dots +$ and $- \dots -$ may be empty.
If $V_0(e)=+$ for all $e \in Y$ or $V_0(e)=-$ for all $e \in Y$,
it is clear that ${\cal N}$ or $_{-Y}{\cal N}$ is acyclic.
Otherwise, if we assume that $V_i(e_0) \neq 0$ for $i=0,\dots,p-1$ (just for simplicity. All the discussion below similarly applies to the case where $V_i(e_0) = 0$ for some $i$), it holds that
\begin{align*}
 V_0(e_0)=\dots =V_{l_1}(e_0)=+,V_{l_1+1}(e_0)=\dots =V_{l_2}(e_0)=-,V_{l_2+1}(e_0)=\dots =V_{p-1}(e_0)=+
\end{align*}
for some $l_1,l_2 \in \mathbb{Z}$ with $0 \leq l_1 < l_2 \leq p-1$ and $e_0 \in Y$ (Case (A)) or that
\begin{align*}
V_0(e_0)=\dots =V_{l_1}(e_0)=-,V_{l_1+1}(e_0)=\dots =V_{l_2}(e_0)=+,V_{l_2+1}(e_0)=\dots =V_{p-1}(e_0)=-
\end{align*}
for some $l_1,l_2 \in \mathbb{Z}$ with $0 \leq l_1 < l_2 \leq p-1$ and $e_0 \in Y$ (Case (B)).

First, let us consider Case (A).
Without loss of generality (by relabeling $V_0,\dots,V_{p-1}$ appropriately if necessary), we can assume that 
\begin{align*} 
V_0(e_0)=\dots =V_{k}(e_0)=+,V_{k+1}(e_0)=\dots =V_{p-1}(e_0)=-
\end{align*}
for some $k \in \mathbb{Z}$ with $0 \leq k \leq p-1$.
If $k=p-1$, then ${\cal N}$ is clearly acyclic.
On the other hand, if $k=0$, then $_{-Y}{\cal N}$ is acyclic. 
In the following, we consider the case where $1 \leq k \leq p-2$.
Applying $\sigma^i$ to the above relation, we obtain
\begin{align*}
V_0(e_i)=\dots=V_{i-1}(e_i)=-, V_i(e_i)=\dots =V_{i+k}(e_i)=+,V_{i+k+1}(e_i)=\dots =V_{p-1}(e_i)=-
\end{align*}
if $1 \leq i \leq p-1-k$ and
\begin{align*} 
V_0(e_i)=\dots =V_{i+k-p}(e_i)=+, V_{i+k-p+1}(e_i)=\dots = V_{i-1}(e_i)=-,V_i(e_i)=\dots =V_{p-1}(e_i)=+
\end{align*}
if $p-1-k < i < p$.
Therefore, it holds that
\begin{align*} V_0(e_0)=+,V_0(e_1)=\dots=V_0(e_{p-1-k})=-,V_0(e_{p-k})=\dots =V_0(e_{p-1})=+.
\end{align*}
We also see that
\begin{align*}
\begin{split}
&V_{p-1}(e_0)=\dots=V_{p-1}(e_{p-2-k})=-,V_{p-1}(e_{p-1-k})=\dots =V_{p-1}(e_{p-1})=+,\\
&V_{k}(e_0)=\dots = V_{k}(e_{k})=+, V_{k}(e_{k+1})=\dots = V_{k}(e_{p-1})=-,\\
&V_{k+1}(e_0)=-,V_{k+1}(e_1)=\dots=V_{k+1}(e_{k+1})=+,V_{k+1}(e_{k+2})=\dots =V_{k+1}(e_{p-1})=-.
\end{split}
\end{align*}
Applying vector elimination to $V_0,V_{p-1}$ and $e_0$ (resp. $V_{k},V_{k+1}$ and $e_0$),
we obtain a covector $W_1$ (resp. $W_2$) such that
\begin{align*}
\begin{split}
&W_1(e_0)=0, W_1(e_1)=\dots=W_1(e_{p-2-k})=-, W_1(e_{p-k})=\dots =W_1(e_{p-1})=+, \\
&W_2(e_0)=0, W_2(e_1)=\dots=W_2(e_{k})=+, W_2(e_{k+2})=\dots =W_2(e_{p-1})=-.
\end{split}
\end{align*}
If $1 \leq k \leq p-3$, then apply vector elimination to $W_1,W_2$ and $e_1$.
Since $\{ e_0,e_1\}$ is a facet of ${\cal M}|_Y$, we obtain a covector $W_3$ such that
\begin{align*}
&W_3(e_0)=W_3(e_1)=0,\  W_3(e)=+ \text{ for all $e \in X$,} \ W_3(e)=+ \text{ for all $Y \setminus \{ e_0,e_1\}$, or} \\
&W_3(e_0)=W_3(e_1)=0, \ W_3(e)=+ \text{ for all $e \in X$,} \ W_3(e)=- \text{ for all $Y \setminus \{ e_0,e_1\}$.}
\end{align*}
Clearly, it holds that $W_3(e)=+$ for all $e \in X$.
By considering $W_3 \circ \sigma (W_3) \circ \dots \circ \sigma^{p-1}(W_3)$, it is proved that the oriented matroid ${\cal N}$ or $_{-Y}{\cal N}$
is acyclic.
If $k=p-2$ and $W_2(e_{p-1}) \geq 0$, 
then ${\cal N}$ is acyclic because of the positive covector $W_2 \circ \sigma (W_2) \circ \dots \circ \sigma^{p-1}(W_2)$.
If $k=p-2$ and $W_2(e_{p-1})=-$, then
apply vector elimination to $W_1,W_2$ and $e_{p-1}$.
Since $\{ e_0,e_{p-1}\}$ is a facet of ${\cal M}|_Y$, the obtained covector $W_4$ fulfills
\begin{align*}
&W_4(e_0)=W_4(e_{p-1})=0,\  W_4(e)=+ \text{ for all $e \in X$,} \ W_4(e)=+ \text{ for all $Y \setminus \{ e_0,e_{p-1}\}$, or} \\
&W_4(e_0)=W_4(e_{p-1})=0,\  W_4(e)=+ \text{ for all $e \in X$,} \ W_4(e)=- \text{ for all $Y \setminus \{ e_0,e_{p-1}\}$.}
\end{align*}
Clearly, it holds that $W_4(e)=+$ for all $e \in X$.
Therefore, ${\cal N}$ or $_{-Y}{\cal N}$ is acyclic because of the positive covector $W_4 \circ \sigma (W_4) \circ \dots \circ \sigma^{p-1}(W_4)$.

The above discussion also applies to the case where $V_i(e_0)=0$ for some $i$ and Case (B).
Therefore, it is concluded that the oriented matroid ${\cal N}$ or $_{-Y}{\cal N}$ is acyclic.
\\
\\
Next, we see that the case $2 \leq l \leq p-2$ never happens. Assume that this relation holds and
consider the covector $W_2$ again.
Since $e_0,\sigma^l(e_0),\sigma^{2l}(e_0),\dots,\sigma^{(p-1)l}(e_0)$ form an alternating matroid in this order, 
the sequence $W_2(\sigma^l (e_0))$,$\dots$, $W_2(\sigma^{(p-1)l}(e_0))$ must be one of the following forms:
\[ -\dots -+ \dots +, \ 0-\dots -+ \dots +, \ -\dots -0+ \dots +, \ -\dots -+ \dots +0,\]
\[ +\dots +- \dots -, \ 0+\dots +- \dots -, \ +\dots +0- \dots -, \ +\dots +- \dots -0,\]
where $+\dots +$ and $- \dots -$ may be empty (see Proposition \ref{prop:devide} in Appendix 1).
Therefore, the following must hold:
\begin{itemize}
\item $\{ \sigma^l(e_0), \sigma^{2l}(e_0),\dots, \sigma^{kl}(e_0) \} = \{ e_1,\dots,e_k \} \text{ or } = \{ e_{k+1},\dots,e_{p-1} \}$ \\
\text{ (i.e., 
$\{ l, 2l, \dots, kl \} \bmod p = \{ 1,\dots,k\} \text{ or } = \{ k+1,\dots,p-1\}$), and}
\item
$\{ \sigma^{(k+1)l}(e_0), \sigma^{(k+2)l}(e_0),\dots, \sigma^{(p-1)l}(e_0) \} = \{ e_1,\dots,e_k \} \text{ or } = \{ e_{k+1},\dots,e_{p-1} \}$ \\
\text{ (i.e., 
$\{ (k+1)l, (k+2)l, \dots, (p-1)l \} \bmod p = \{ 1,\dots,k\} \text{ or } = \{ k+1,\dots,p-1\}$)}
\end{itemize}
\text{or} 
\begin{itemize}
\item $\{ \sigma^l(e_0), \sigma^{2l}(e_0),\dots, \sigma^{(k+1)l}(e_0) \} = \{ e_1,\dots,e_{k+1} \} \text{ or } = \{ e_{k+2},\dots,e_{p-1} \}$ \\
\text{ (i.e., 
$\{ l, 2l, \dots, (k+1)l \} \bmod p = \{ 1,\dots,k+1\} \text{ or } = \{ k+2,\dots,p-1\}$), and}
\item $\{ \sigma^{(k+2)l}(e_0), \sigma^{(k+3)l}(e_0),\dots, \sigma^{(p-1)l}(e_0) \} = \{ e_1,\dots,e_{k+1} \} \text{ or } = \{ e_{k+2},\dots,e_{p-1} \}$ \\
\text{ (i.e., 
$\{ (k+2)l, (k+3)l, \dots, (p-1)l \} \bmod p = \{ 1,\dots,k+1\} \text{ or } = \{ k+2,\dots,p-1\}$).}
\end{itemize}
It is impossible since we are assuming $2 \leq l \leq p-2$.
Indeed, the case ``$\{ l, 2l, \dots, kl \} \bmod p = \{ 1,\dots,k\}$'' is proved to be impossible as follows.
Suppose that this relation holds.
Let $k_0 \in \mathbb{N}$ be such that $k_0l \equiv 1 \ (\bmod \ p)$ and $1 \leq k_0 \leq k$.
Then, we have $\{ (k_0+1)l, \dots, (k+k_0)l \}  \bmod  p = \{ 2,\dots,k+1\}$.
If $k_0 \geq 2$, there exists $k'\in \mathbb{N}$ such that 
$k+1 \leq k' \leq k+k_0$ and $k'l \bmod p \in \{ 2,\dots,k\}$.
This leads to that there exists  $k''\in \mathbb{N}$ such that $1 \leq k'' \leq k$ and $k'l \equiv k''l \ (\bmod \ p)$, 
and thus that $(k'-k'')l \equiv 0 \ (\bmod \ p)$.
Since $1 \leq k'-k'' \leq k+k_0$, this is a contradiction.
Therefore, we have $l \bmod p \equiv 1$.
However, it is impossible since $2 \leq l \leq p-2$.
This proves that the case ``$\{ l, 2l, \dots, kl \} \bmod p = \{ 1,\dots,k\}$'' never happens.
The other cases are also proved to be impossible in the same way.
\\
\\
Now, let us take a look at structure of ${\cal N}$ if there exists $\tau \in G({\cal M})$ such that $y = \tau (x)$
under the conclusion that ${\cal N}$ or $_{-Y}{\cal N}$ is acyclic.
If ${\cal N}$ is acyclic, there exists an extreme point $e$, i.e., there exists a covector $V$ of ${\cal M}$ such that $V(e)=0$ and $V(f)=+$ for all $f \in X \cup Y \setminus \{ e \}$.
Let us consider the covectors $\sigma^i \tau^j (V)$ for $i=0,\dots,p-1$ and $j=0,1$ if $e \in X$ and
the covectors $\sigma^i \tau^j (V)$ for $i=0,\dots,p-1$ and $j=0,-1$ if $e \in Y$.
Then, the oriented matroid ${\cal N}$ turns out to be a matroid polytope of rank $3$, i.e.,
a relabeling of the alternating matroid $A_{3,2|X|}$ (Proposition \ref{prop:rank_relabeling} in Appendix 1).
If $_{-Y}{\cal N}$ is acyclic,  there exists a covector $V'$ of ${\cal M}$ such that $V'(e)=0$ and $V'(f)=+$ for all $f \in X \setminus \{ e \}$, and $V'(g)=-$ for all $g \in Y$,
or there exists a covector $V''$ of ${\cal M}$ such that $V''(e)=0$ and $V''(f)=+$ for all $f \in X$, and $V'(g)=-$ for all $g \in Y \setminus \{ e\}$.
Applying $\sigma^i \tau^j$ to $V'$ for $i=0,\dots,p-1$ and $j=0,1$ 
or to $V''$ for $i=0,\dots,p-1$ and $j=0,-1$, 
the reorientation $_{-Y}{\cal N}$ turns out to be a matroid polytope of rank $3$, i.e., a relabeling of the alternating matroid $A_{3,2|X|}$.

Finally, we see that $\sigma$ acts on ${\cal N}$ or $_{-Y}{\cal N}$ as the 2nd rotational symmetry.
If ${\cal N} \simeq A_{3,2|X|}$, then we have $y \in V_{k-1}^+ \cap V_k^- \cap V_{k+1}^+$ for some $k \in \mathbb{Z}$ with $0 \leq k \leq p-1$ (, where $V_{-1}$ is interpreted as $V_{p-1}$).
Since $|Y|=p$ and $|V_{k-1}^+ \cap V_k^- \cap V_{k+1}^+| = |V_{l-1}^+ \cap V_l^- \cap V_{l+1}^+|$ for any $k,l \in \mathbb{Z}$ with $0 \leq k,l \leq p-1$,
we have  $|V_{k-1}^+ \cap V_k^- \cap V_{k+1}^+| = 1$ for all $k \in \mathbb{Z}$ with $0 \leq k \leq p-1$.
This implies that $\sigma$ acts on ${\cal N}$ as the 2nd rotational symmetry.
The same argument yields that $\sigma$ acts on $_{-Y}{\cal N}$ as the 2nd rotational symmetry and that ${\cal N} \simeq {_{-\{ 2,4,\dots,2|X|\}}}A_{3,2|X|}$ if $_{-Y}{\cal N} \simeq A_{3,2|X|}$.
\begin{figure}[h]
\begin{center}
\includegraphics[scale=0.32,clip]{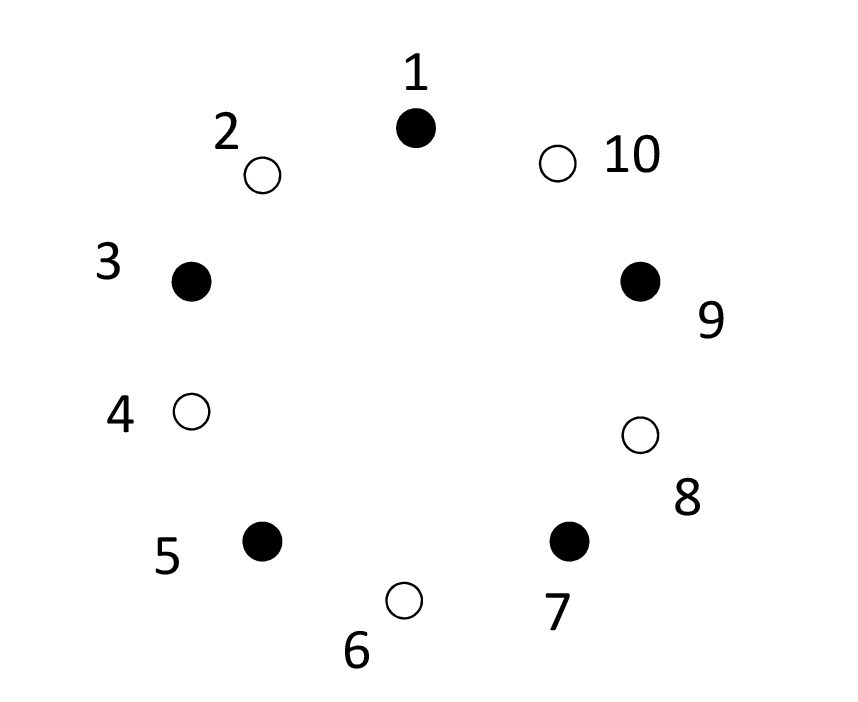}
\includegraphics[scale=0.20,clip]{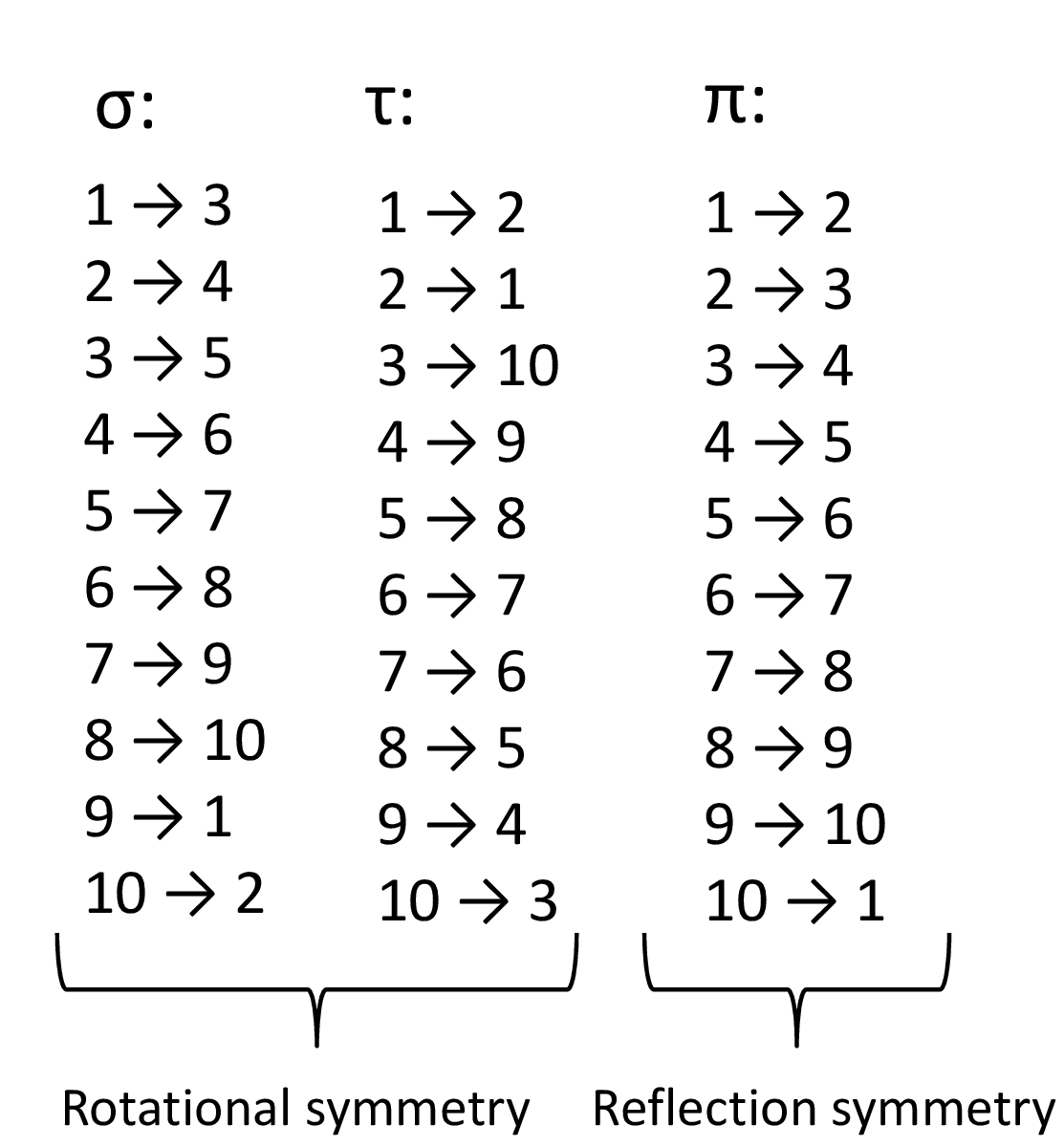}
\end{center}
\caption{The case where $_{-Y}{\cal N}$ is acyclic (represented as a signed point configuration)}
\label{fig:half_acyclic}
\end{figure}
\end{proof}
\\
The proposition is generalized as follows.
\begin{prop}
\label{prop:rank3_transitivity3}
Let ${\cal M}=(E,\{ \chi, -\chi \})$ be a simple oriented matroid of rank $3$ and
$G$ the cyclic group generated by a proper rotational symmetry $\sigma \in R({\cal M})$ of order $p(>2)$ (let us assume that $\sigma$ is the 1st rotational symmetry).
Let  $P_1,\dots,P_a, N_1, \dots, N_b$ be pairwise disjoint rank 3 $G$-orbits (let $P:=\cup_{i=1}^a{P_a}$, $N:=\cup_{i=1}^b{N_b}$ and $c:=a+b$) such that $_{-N}{\cal M}|_{P \cup N} \simeq A_{3,cp}$, 
on which $\sigma$ acts as the $c$-th rotational symmetry.

For a rank 3 $G$-orbit $X$,
if there exists $x \in P \cup N \setminus X$, then we have  $X \cap (P \cup N) = \emptyset$
and $_{-N}{\cal M}|_{P \cup N \cup X}$ or $_{-N \cup X}{\cal M}|_{P \cup N \cup X}$ is acyclic.
If in addition $y = \tau (x)$ for some $x \in P \cup N$, $y \in X$ and $\tau \in G({\cal M})$, then
we have  $_{-N}{\cal M}|_{P \cup N \cup X} \simeq A_{3,(c+1)p}$ or $_{-N \cup X}{\cal M}|_{P \cup N \cup X} \simeq A_{3,(c+1)p}$. 
Here, $\sigma$ is the $(c+1)$-th rotational symmetry of  $_{-N}{\cal M}|_{P \cup N \cup X}$ or $_{-N \cup X}{\cal M}|_{P \cup N \cup X}$.
\end{prop}
\begin{proof}
Let us label the elements of $P \cup N$ by $x_0(:=x),x_1\dots,x_{cp-1}$ so that they form the reorientation $_{-N}A_{3,cp}$ in this order.
Let $V$ be the cocircuit of $_{-N}{\cal M}|_{P \cup N \cup X}$ such that $V(x_0)=V(x_1) = 0$ and $V(e)=+$ for all $e \in P \cup N \setminus \{ x_0, x_1\}$.
\\
\\
{\bf (I)} $V(e)=+$ for all $e \in X$.

Since $V \circ \sigma (V) \circ \dots \circ \sigma^{p-1}(V)$ is the positive covector, the oriented matroid $_{-N}{\cal M}|_{P \cup N \cup Y}$ is acyclic.
\\
\\
{\bf (II)} $V(e_0)=-$ for some $e_0 \in Y$.

Let us label $e_1:= \sigma (e_0),\dots,e_{p-1}:=\sigma^{p-1}(e_0)$.
Note that $x_c=\sigma(x),x_{2c}=\sigma^2(x),\dots,x_{(p-1)c}=\sigma^{p-1}(x)$.
Let $V_k$ be the cocircuit of $_{-N}{\cal M}|_{P \cup N \cup X}$ such that
\[ V_k(x_k)=V_k(x_{k+1})=0, V_k(e)=+ \text{ for all $e \in P \cup N \setminus \{ x_k,x_{k+1} \}$,} \]
for $k=1,\dots,pc-1$, where $x_{pc}:=x_0$.
Then there exist $l_1,l_2 \in \mathbb{Z}$ with $0 \leq l_2 < l_1 \leq pc-1$ such that
\begin{align}
\label{first_part}
\begin{split}
&V_{l_1}(e_0)=V_{l_1+1}(e_0)=V_{pc-1}(e_0)=V_0(e_0)=V_1(e_0)=\dots =V_{l_2-1}(e_0)=+,\\
&V_{l_2}(e_0)=V_{l_2+1}(e_0)=\dots = V_{l_1-1}(e_0)=-
\end{split}
\end{align}
or
\begin{align}
\label{latter_part}
\begin{split}
&V_{l_1}(e_0)=V_{l_1+1}(e_0)=V_{pc-1}(e_0)=V_0(e_0)=V_1(e_0)=\dots =V_{l_2-1}(e_0)=-,\\
&V_{l_2}(e_0)=V_{l_2+1}(e_0)=\dots = V_{l_1-1}(e_0)=+
\end{split}
\end{align}
by Proposition \ref{prop:alter_ordering} in Appendix 1. 
Let us first consider Case (\ref{first_part}).
Without loss of generality, we can assume that $0 \leq l_2 \leq c-1$.
For $k=1,\dots,p-1$, let us apply $\sigma^k$. Then, we obtain
\begin{align*}
\begin{split}
&V_{l_1+kc}(e_k)=V_{l_1+kc+1}(e_k)=V_{pc-1}(e_k)=V_0(e_k)=V_1(e_k)=\dots =V_{l_2+kc-1}(e_k)=+,\\
&V_{l_2+kc}(e_k)=V_{l_2+kc+1}(e_k)=\dots = V_{l_1+kc-1}(e_k)=-
\end{split}
\end{align*}
if $kc+l_1 < pc$.
On the other hand, we have
\begin{align*}
\begin{split}
&V_{l_1+kc-pc}(e_k)=V_{l_1+kc-pc+1}(e_k)=\dots =V_{l_2+kc-1}(e_k)=+,\\
&V_0(e_k)=V_1(e_k)=\dots = V_{l_1+kc-pc-1}(e_k)=-, V_{l_2+kc}(e_k)=V_{l_2+kc+1}(e_k)=\dots =V_{pc-1}(e_k)=-
\end{split}
\end{align*}
if $kc+l_1 \geq pc$.
Therefore, the following holds:
\begin{align*} 
\begin{split}
&V_{l_2-1}(e_0)=\dots = V_{l_2-1}(e_{p-\lceil \frac{l_1}{c} \rceil})=+, \ V_{l_2-1}(e_{p-\lceil \frac{l_1}{c} \rceil+1})=\dots =V_{l_2-1}(e_{p-1})=-, \\
&V_{l_2}(e_0)=-, \ V_{l_2}(e_1)=\dots = V_{l_2}(e_{p-\lceil \frac{l_1}{c} \rceil+1})=+, \ V_{l_2}(e_{p-\lceil \frac{l_1}{c} \rceil+2})=\dots =V_{l_2}(e_{p-1})=-,
\end{split}
\end{align*}
and
\begin{align*}
\begin{split}
&V_{l_3}(e_0)=\dots = V_{l_3}(e_{\lceil \frac{l_1}{c} \rceil-2})=-, \ V_{l_3}(e_{\lceil \frac{l_1}{c} \rceil -1})=\dots =V_{l_3}(e_{p-1})=+, \\
&V_{l_3+1}(e_0)=+, \ V_{l_3+1}(e_1)=\dots = V_{l_3+1}(e_{\lceil \frac{l_1}{c} \rceil -1})=-, \ V_{l_3+1}(e_{\lceil \frac{l_1}{c} \rceil})=\dots =V_{l_3+1}(e_{p-1})=+,
\end{split}
\end{align*}
where $l_3:=l_2-1+(\lceil \frac{l_1}{c} \rceil-1)c$.
\begin{figure}[h]
\begin{center}
\includegraphics[scale=0.28,clip]{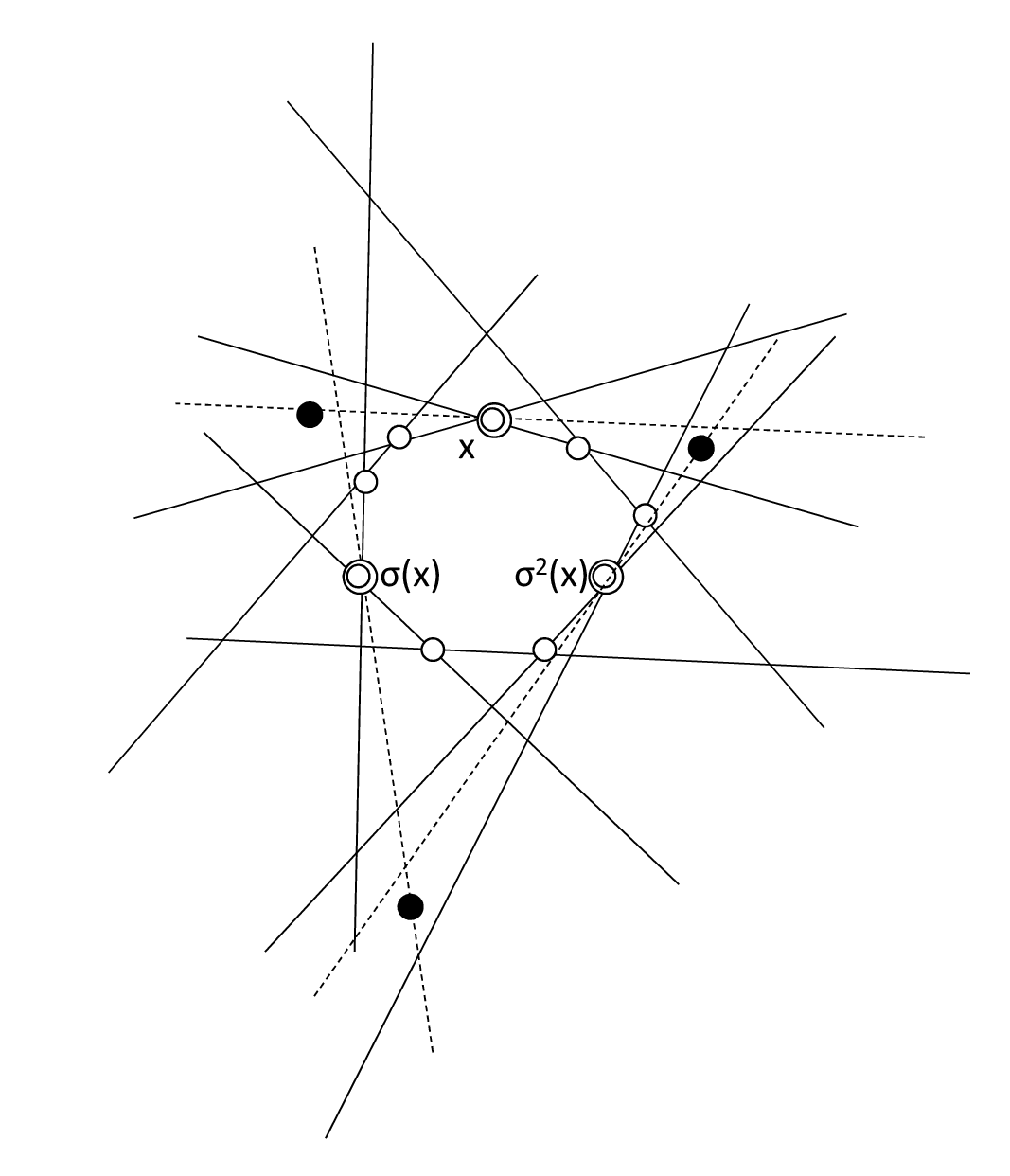}
\includegraphics[scale=0.28,clip]{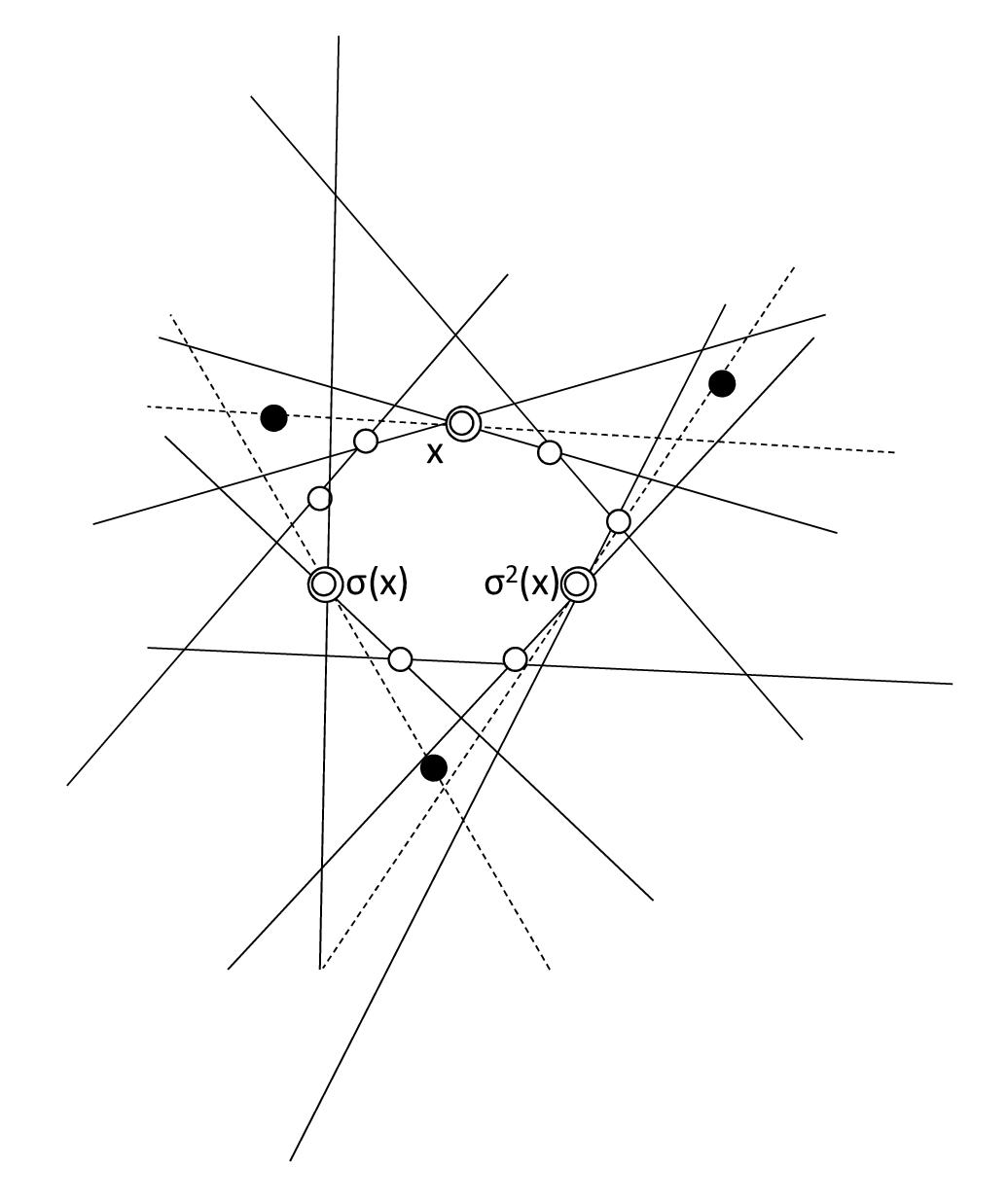}
\includegraphics[scale=0.28,clip]{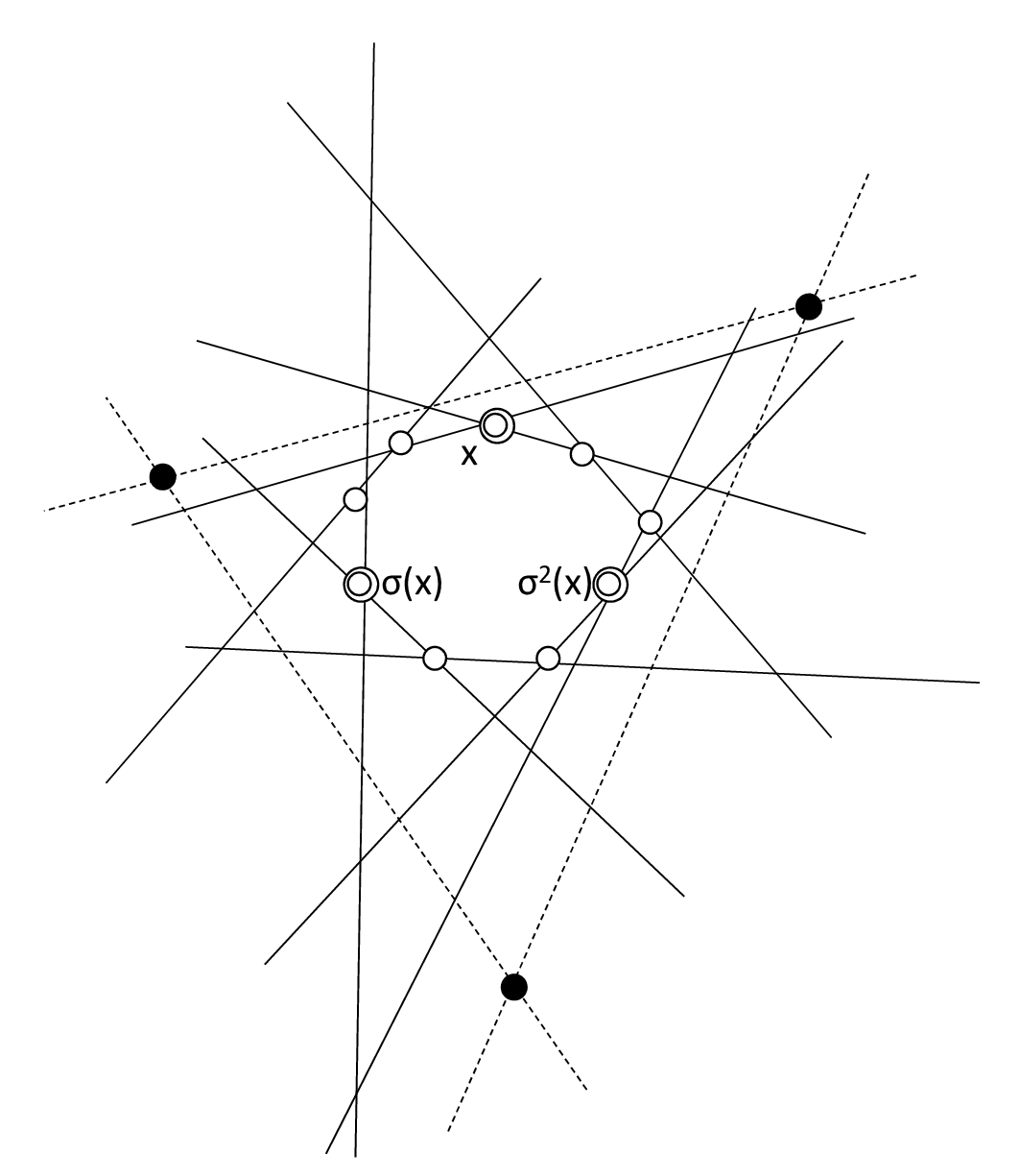}
\caption{$_{-N}{\cal M}|_{P \cup N \cup X}$ (depicted as a signed point configuration. $p=q=3$)}
\end{center}
\end{figure}
\\
(a) $l_1 \leq c$.

The oriented matroid ${\cal M}|_{P \cup N \cup X}$ is acyclic because of the positive covector $V_{l_2} \circ \sigma (V_{l_2}) \circ \dots \circ \sigma^{p-1}(V_{l_2})$.
\\ \\
(b) $l_1 > c$.

Let us apply vector elimination to $V_{l_2-1},V_{l_2}$ and $e_0$ (resp. $V_{l_3},V_{l_3+1}$ and $e_0$), and 
obtain a covector $W_1$ (resp. $W_2$) such that
\begin{align*}
&W_1(e_0)=W_1(x_{l_2})=0, W_1(e)=+ \text{ for all $e \in P \cup N \setminus \{ x_{l_2}\}$}.\\
&W_2(e_0)=W_2(x_{l_3+1})=0, W_2(e)=+ \text{ for all $e \in P \cup N \setminus \{ x_{l_3+1}\}$}.
\end{align*}
By the same argument as the proof of Proposition \ref{prop:rank3_transitivity2},
the oriented matroid $_{-N}{\cal M}|_{P \cup N \cup Y}$ or $_{-N \cup X}{\cal M}|_{P \cup N \cup X}$ is proved to be acyclic.
The same argument applies to Case (\ref{latter_part}).
By the same argument as the proof of Proposition \ref{prop:rank3_transitivity2},
if $y = \tau (x)$ for some $x \in P \cup N$, $y \in X$  and $\tau \in G({\cal M})$, then
we have  $_{-N}{\cal M}|_{P \cup N \cup X} \simeq A_{3,(c+1)p}$ or $_{-N \cup X}{\cal M}|_{P \cup N \cup X} \simeq A_{3,(c+1)p}$. 
Here, $\sigma$ is the $(c+1)$-th rotational symmetry of  $_{-N}{\cal M}|_{P \cup N \cup X}$ or $_{-N \cup X}{\cal M}|_{P \cup N \cup X}$.
\end{proof}

\begin{prop}
\label{prop:insep}
Let ${\cal M}=(E,\{ \chi, -\chi \})$ ($n:=|E|$) be a simple oriented matroid of rank $3$ with a proper rotational symmetry $\sigma$ of order $p > 2$.
Suppose that $R({\cal M})$ acts transitively on $E$.
Then, we have ${\cal M} \simeq A_{3,n}$ or ${\cal M} \simeq {_{-[n] \cap 2 \mathbb{N}}A_{3,n}}$ (only when $n$ is even).
Therefore, we have one of $R({\cal M}) \simeq D_n$ ($n$: even), $R({\cal M}) \simeq \mathbb{Z}_n$ and $R({\cal M}) \simeq A_4$.
\end{prop}
\begin{proof}
Let us first consider the case that $\sigma$ does not have a fixed point for any $\sigma \in R({\cal M}) \setminus \{ \id \}$.
Let  $G$ be the cyclic group generated by $\sigma$ and take a rank $3$ $G$-orbit $X_0:= G \cdot x_0 \subseteq E$.
For $i=1,2,\dots,$ we continue to take $X_i := \cup_{k=0}^i{G \cdot x_k}$ while there exists $x_i \in E \setminus X_i$.
By Proposition \ref{prop:rank3_transitivity3}, ${\cal M}|_{X_i}$ is a reorientation of an alternating matroid of rank $3$ for any $i=1,2,\dots$.
This leads to that ${\cal M}$ is a reorientation of an alternating matroid of rank $3$.
Therefore, $IG({\cal M})$ is an $n$-cycle or the complete graph $K_4$.
If $IG({\cal M})$ is the complete graph $K_4$, an easy case analysis yields that ${\cal M} \simeq {_{-\{ 2,4\}}A_{3,4}}$ or that ${\cal M} \simeq A_{3,4}$.
Since $R({_{-\{ 2,4\}}A_{3,4}}) (\simeq A_4)$ has a fixed point, it must hold that ${\cal M} \simeq A_{3,4}$ and thus that $R({\cal M}) \simeq D_8$.
In the case that $IG({\cal M})$ is an $n$-cycle,
since the automorphism group of an $n$-cycle is the dihedral group $D_{2n}$, $R({\cal M})$ must be a subgroup of $D_{2n}$.
By Cavior's theorem, $R({\cal M})$ is either a cyclic group or a dihedral group.
If $R({\cal M})$ is a dihedral group, we have $R({\cal M})= H \cup \tau H$ for some cyclic group $H \subseteq R({\cal M})$ and $\tau \in R({\cal M})$ of order $2$.
Let $m:=|H|$.
Since $R({\cal M})$ contains an element of order $p>2$, we have $m > 2$.
By Proposition \ref{prop:rank3_transitivity2}, it holds that ${\cal M} \simeq {_{-[2m] \cap 2 \mathbb{N}}A_{3,2m}}$. 
This leads to that $m= \frac{n}{2}$ and that $R({\cal M}) \simeq D_n$.
If $R({\cal M})$ is a cyclic group, i.e., $R({\cal M}) \simeq \mathbb{Z}_n$, then we have ${\cal M} \simeq A_{3,n}$ by Proposition \ref{rank3_transitivity1}.

Suppose now that $\sigma_0 \in R({\cal M})$ has a fixed point $q$ and let $G$ be the cyclic group generated  by $\sigma_0$.
Let us consider $x \in E \setminus \{ q \}$ and take $y \in E \setminus (G \cdot x \cup \{ q \})$. 
Then, we have ${\rm rank}(G \cdot y)=3$ (there are not two fixed points in a simple oriented matroid of rank $3$ by Theorem \ref{prop:uniqueness}).
By Proposition \ref{prop:rank3_transitivity2}, ${\cal M}|_{G \cdot x \cup G \cdot y}$ is a reorientation of an alternating matroid.
Continuing this discussion, we have ${\cal M} = {\cal N} \cup q$, where ${\cal N}$ is a reorientation of an alternating matroid and $q$ is a fixed point of $\sigma$.
Because of transitivity, ${\cal M}|_F$ is a reorientation of $A_{3,n-1}$ for any $F \subseteq E$ with $|F|=n-1$.
Since the reorientation class of ${\cal M}$ is uniquely determined by the list of the reorientation classes of ${\cal M}|_F$ for $F \subseteq E$ with 
$|F|=3+2$ (see \cite[Theorem 3.1.]{R88}),
${\cal M}$ is reorientation equivalent to $A_{3,n}$ if $n \geq 6$ (recall that $(A_{3,n})|_F \simeq A_{3,n-1}$ for any $F \subseteq E$ with $|F|=n-1$).
An easy case analysis leads to that ${\cal M}$ is a reorientation of an alternating matroid regardless of the value of $n$.
Therefore,  the inseparability graph $IG({\cal M})$ is an $n$-cycle if $n > 4$.
We observe that $\sigma_0$ induces a symmetry of $IG({\cal M})$, but an $n$-cycle does not have an automorphism of order $> 2$ with exactly one fixed point.
This is a contradiction.
Therefore, $n=4$ must hold.
An easy case analysis leads to that ${\cal M} \simeq A_{3,4}$ or ${\cal M} \simeq {_{-\{ 2, 4 \}}A_{3,4}}$
and that $R({\cal M}) \simeq \mathbb{Z}_4$ or $R({\cal M}) \simeq A_4$.
\end{proof}

Therefore, the rotational symmetry group of a simple oriented matroid ${\cal M}$ of rank $3$ with a proper rotational symmetry of order $p>2$ is a subgroup of one of $\mathbb{Z}_n$ ($n \geq 3$),  $D_{2n}$ ($n \geq 3$) and $A_4$.
By Cavior's theorem, it is same to say that the rotational symmetry group is 
isomorphic to one of $\mathbb{Z}_n$ ($n \geq 3$),  $D_{2n}$ ($n \geq 3$) and $A_4$.

\begin{prop}
Let ${\cal M}$ be a simple oriented matroid of rank $3$ with a proper rotational symmetry of order $p>2$.
An $R({\cal M})$-orbit is isomorphic to one of $A_{1,1}$, $A_{3,n}$ ($n \geq 3$) and $_{-[2n] \cap 2 \mathbb{N}}A_{3,2n}$ ($n \geq 2$).
\end{prop}

\subsection{Orbit structure when all rotational symmetries have order $1$ or $2$}
\label{sec:proper3}
Finally, let us consider the case where all elements of $R({\cal M})$ have order $1$ or $2$, i.e., $R({\cal M}) \simeq \mathbb{Z}_2^n$ for some $n \geq 0$.
We first prove that  $n \leq 2$ must hold.
\begin{prop}
Let ${\cal M}$ be a simple oriented matroid of rank $3$.
Then, $R({\cal M})$ does not contain a subgroup isomorphic to $\mathbb{Z}_2^3$.
\end{prop}
\begin{proof}
We prove the proposition by contradiction.
Let $\sigma, \tau, \pi \in R({\cal M})$ be generators of a subgroup $G \subseteq R({\cal M})$ that is isomorphic to $\mathbb{Z}_2^3$.
Consider a rank $2$ orbit $X:= \{ x, \sigma (x) \}$, and orbits $Y:= X \cup \tau (X)$ and $Z:= Y \cup \pi (Y)$.
Note that $Y \neq X$. Otherwise, we have $\tau|_X = \id$ or $\tau|_X = \sigma|_X$ and thus $\tau = \id$ or $\tau = \sigma$, which is a contradiction.
Also, note that ${\rm rank}(Y) = 3$ by Propositions \ref{prop:classification_rank2_rotation} and \ref{prop:classification_rank2_full}.
Therefore, it holds that ${\cal M}|_Y \simeq A_{3,4}$ or that ${\cal M}|_Y \simeq {_{-\{ 2, 4 \}}A_{3,4}}$.
A similar discussion to Propositions \ref{prop:rank3_transitivity2} and \ref{prop:rank3_transitivity3}
leads to that ${\cal M}|_Z$ is a reorientation of $A_{3,4}$ or $A_{3,8}$.
Therefore, the inseparability graph $IG({\cal M}|_Z)$ is a $4$-cycle or $8$-cycle.
Their automorphism groups  do not contain a subgroup isomorphic to $\mathbb{Z}_2^3$. 
This is a contradiction.
\end{proof}
\\
The case $n=2$ is realized as follows.
Let $V:=(v_1,\dots,v_5)$ be the vector configuration defined by
\[
\begin{pmatrix}
1 & -1 & -1 & 1 & 0 \\
1 & 1  & -1 & -1 & 0  \\
1 & -1 & 1  & -1 & 1  
\end{pmatrix} 
\]
and ${\cal M}_V$ be the associated oriented matroid.
Its rotational symmetry group is generated by the symmetries $(1 4) (2 3)$ and $(1 2) (3 4)$.
Therefore, we have  $R({\cal M}_V) \simeq \mathbb{Z}_2^2$.
The following proposition describes the orbit structure when $R({\cal M}) \simeq \mathbb{Z}_2^2$.

\begin{prop}
Let ${\cal M}$ be a simple oriented matroid of rank $3$ with $R({\cal M}) \simeq \mathbb{Z}_2^2$.
An $R({\cal M})$-orbit is isomorphic to $A_{1,1}$ or $_{-\{ 2,4\}}A_{3,4}$.
\end{prop}
\begin{proof}
Let $E$ be the ground set of ${\cal M}$ and $\sigma, \tau \in R({\cal M})$ the generators of $R({\cal M})$.
Take an $R({\cal M})$-orbit $X$ with $\rank(X) \geq 2$.
If $|X|=2$, we have $\sigma|_X = \id$ or $\sigma|_X = \tau|_X$ and thus $\sigma = \id$ or $\sigma = \tau$, which is a contradiction.
Since $|X|=3$ is impossible, we have $|X|=4$.
By Propositions \ref{prop:classification_rank2_rotation} and \ref{prop:classification_rank2_full}, it is impossible to have $\rank(X)=2$.
If $\rank(X)=3$, then
it holds that
$\chi (x,\sigma (x),\sigma \tau (x)) = \chi (\tau (x), \sigma \tau (x), \sigma (x)) = \chi (\sigma \tau (x), \tau (x), x) = \chi (\sigma (x), x, \tau (x))$.
It follows that ${\cal M}|_X \simeq {_{-\{ 2,4\}}}A_{3,4}$.
If $\rank(X)=1$, then we have ${\cal M}|_X \simeq A_{1,1}$.
\end{proof}
\\
\\
Since the cases $n=0$ and $n=1$ are also possible,
we can summarize this subsection as the following proposition.
\begin{prop}
Let ${\cal M}$ be a simple oriented matroid of rank $3$.
If $\sigma^2 = \id$ for all $\sigma \in R({\cal M})$, then
$R({\cal M})$ is isomorphic to one of  $\mathbb{Z}_2^2$ ($\simeq D_4$), $\mathbb{Z}_2$  ($\simeq D_2$) and $\{ \id \}$.
An $R({\cal M})$-orbit is isomorphic to one of ${_{-\{ 2,4\}}}A_{3,4}$, $A_{2,2}$ and $A_{1,1}$.
\end{prop}

\subsection{Classification}
Combining the results in this section, we obtain the following theorem.
\begin{thm}
\label{thm:classification_rotation_rank3}
Let ${\cal M}$ be a simple oriented matroid of rank $3$.
Then, $R({\cal M})$ is isomorphic to one of the cyclic group $\mathbb{Z}_n$ ($n \geq 1$), the dihedral group $D_{2n}$ ($n \geq 1$) and the alternating group $A_4$.
An $R({\cal M})$-orbit is isomorphic to one of $A_{1,1}$, $A_{2,2}$, $_{- \{ 2,4,\dots,2p-2\}}A_{2,2p-1}$ ($p \geq 2$), $A_{3,n}$ ($n \geq 3$) and $_{-[2n] \cap 2\mathbb{N}}A_{3,2n}$ ($n \geq 2$).
\end{thm}
Based on Theorem \ref{thm:classification_rotation_rank3}, we now prove the following theorem.
\begin{thm}
\label{thm:classification_full_rank3}
Let ${\cal M}$ be a simple oriented matroid of rank $3$.
Then, $G({\cal M})$ is isomorphic to one of the cyclic group $\mathbb{Z}_n$ ($n \geq 1$), the dihedral group $D_{2n}$ ($n \geq 1$) and the symmetric group $S_4$.
A $G({\cal M})$-orbit is isomorphic to one of $A_{1,1}$, $A_{2,2}$, $_{- \{ 2,4,\dots,2p-2\}}A_{2,2p-1}$ ($p \geq 2$), $A_{3,n}$ ($n \geq 3$) and $_{-[2n] \cap 2\mathbb{N}}A_{3,2n}$ ($n \geq 2$).
\end{thm}
\begin{proof}
Let $E$ be the ground set of ${\cal M}$. 
We proceed by case analysis.
\\
{\bf Case 1.} $G({\cal M}) = R({\cal M})$.

In this case, $G({\cal M})$ is isomorphic to a cyclic group or a dihedral group, or the alternating group $A_4$ by Theorem \ref{thm:classification_rotation_rank3}.
Let us see that the case $G({\cal M}) \simeq A_4$ never happens.
Assume that $R({\cal M}) = G({\cal M}) \simeq A_4$ and consider a rank $3$ orbit $X:= R({\cal M}) \cdot x$.
Take $\sigma \in R({\cal M})$ of order $3$ and let $G$ be the cyclic group generated by $\sigma$.
Then, we have $X = X_1 \cup \dots \cup X_m \cup \{ q \}$ for some rank $3$ $G$-orbits $X_1,\dots,X_m$, where $q$ is the fixed point of $\sigma$
(recall the proof of Proposition \ref{prop:insep}).
By Proposition \ref{prop:rank3_transitivity3}, ${\cal M}|_{X_1 \cup \dots \cup X_m}$ is a reorientation of $A_{3,3m}$.
Similarly to the proof of Proposition \ref{prop:insep},  $m=1$ must hold.
Then, we have ${\cal M}|_X \simeq {_{-\{ 2,4 \}}A_{3,4}}$. This leads to $G({\cal M}|_X) \simeq S_4$ and thus to $E \neq X$.
Take $y \in E \setminus X$ and let $Y := R({\cal M}) \cdot y$ (note that $X \cap Y = \emptyset$).
Similarly to the proof of Proposition \ref{prop:insep}, $\sigma$ has a fixed point in $Y$.
Because ${\cal M}$ is simple, this implies that $\rank (\Fix(\sigma)) \geq 2$, which is a contradiction.
Therefore, $G({\cal M})$ must be a cyclic group or a dihedral group in Case 1.
Since $G({\cal M})=R({\cal M})$, a $G({\cal M})$-orbit is isomorphic to one of
$A_{1,1}$, $A_{2,2}$, $_{- \{ 2,4,\dots,2p-2\}}A_{2,2p-1}$ ($p \geq 2$), $A_{3,n}$ ($n \geq 3$) and $_{-[2n] \cap 2\mathbb{N}}A_{3,2n}$ ($n \geq 2$).
(Here, we do not prove that each group or orbit indeed appears since it does appear in Case 2.)
\\
\\
{\bf Case 2.} There exists $\sigma \in G({\cal M}) \setminus R({\cal M})$.

Take a reflection symmetry $\tau \in  G({\cal M}) \setminus R({\cal M})$ arbitrarily.
Then, we have  $\sigma \tau \in R({\cal M})$, which means that $\tau \in \sigma^{-1}R({\cal M}) = \sigma R({\cal M})$.
It follows that $G({\cal M}) = R({\cal M}) \cup \sigma R({\cal M})$.
\\
{\bf Case 2-(i).} $R({\cal M})$ contains non-proper rotational symmetries.

In this case, we have ${\cal M}={\cal N} \cup q$ for some oriented matroid ${\cal N}$ of rank $2$ and a coloop $q$.
Since  $R({\cal M})$ contains non-proper rotational symmetries, we have $G({\cal M})=G({\cal N})$ and thus $G({\cal M})$ is a dihedral group.
By Proposition \ref{prop:classification_rank2_full}, a $G({\cal M})$-orbit is isomorphic to one of
 $A_{1,1}$, $A_{2,2}$, $_{- \{ 2,4,\dots,2p-2\}}A_{2,2p-1}$ ($p \geq 2$).
\\
{\bf Case 2-(ii).} $R({\cal M})$ has a proper rotational symmetry of  order $p > 2$.

In this case, there exists $x \in E$ such that $\rank(G \cdot x)=3$ for the cyclic group generated by $\sigma$, and 
${\cal M}|_{G({\cal M}) \cdot x}$ is reorientation equivalent to an alternating martroid of rank $3$ by Proposition \ref{prop:rank3_transitivity3}.
Therefore, $IG({\cal M}|_{G({\cal M}) \cdot x})$ is a cycle or the complete graph $K_4$.
If $IG({\cal M}|_{G({\cal M}) \cdot x})$ is the complete graph $K_4$, we have $|G({\cal M}) \cdot x| = 4$ and an easy case analysis yields that 
${\cal M}|_{G({\cal M}) \cdot x} \simeq A_{3,4}$ or that ${\cal M}|_{G({\cal M}) \cdot x} \simeq {_{- \{ 2,4 \}}A_{3,4}}$.
Therefore, we have $G({\cal M}) \simeq D_8$ or $G({\cal M}) \simeq S_4$.
If $IG({\cal M}|_{G({\cal M}) \cdot x})$ is a cycle, $G({\cal M})$ is a subgroup of a cyclic group or a dihedral group.
By Proposition \ref{rank3_transitivity1} (and the remark that follows it), an orbit under a cyclic group is isomorphic to one of
$A_{3,n}$ ($n \geq 3$) and $_{-[2n] \cap 2\mathbb{N}}A_{3,2n}$ ($n \geq 2$) in Case 2-(ii). 
By Proposition \ref{prop:rank3_transitivity3}, an orbit under a dihedral group is isomorphic to one of
$A_{3,n}$ ($n \geq 3$) and $_{-[2n] \cap 2\mathbb{N}}A_{3,2n}$ ($n \geq 2$) in Case 2-(ii). 
\\
{\bf Case 2-(iii).} $R({\cal M}) \simeq \mathbb{Z}_2^2$.

In this case, it holds that $G({\cal M}) \simeq \mathbb{Z}_2^3$ or $G({\cal M}) \simeq \mathbb{Z}_4 \times  \mathbb{Z}_2$
or  $G({\cal M}) \simeq D_8$.
Let us see that the cases $G({\cal M}) \simeq \mathbb{Z}_2^3$ or $G({\cal M}) \simeq \mathbb{Z}_4 \times  \mathbb{Z}_2$ never happens.
Let $X := R({\cal M}) \cdot x$ be an orbit such that ${\cal M}|_X \simeq {_{-\{ 2,4\}}A_{3,4}}$
and consider $Y := G({\cal M}) \cdot x$.
Then, ${\cal M}|_Y$ must be reorientation equivalent to $A_{3,4}$ or $A_{3,8}$ (similarly to Proposition \ref{prop:rank3_transitivity3}), and thus $IG({\cal M}|_Y)$ must be
the complete graph $K_4$ or the $8$-cycle.
However, their automorphism groups do not contain subgroups isomorphic to $\mathbb{Z}_2^3$ or $\mathbb{Z}_4 \times  \mathbb{Z}_2$.
This is a contradiction.
Therefore, only the case $G({\cal M}) \simeq D_8$ is possible.
In this case we have ${\cal M}|_Y \simeq {_{-\{ 2,4\}}A_{3,4}}$ or ${\cal M}|_Y \simeq {_{-\{ 2,4,6,8\}}A_{3,8}}$.
Therefore, a $G({\cal M})$-orbit is isomorphic to one of 
$A_{1,1}$, $_{-\{ 2,4 \}}A_{3,4}$ and ${_{-\{ 2,4,6,8\}}A_{3,8}}$ in Case 2-(iii).
\\
{\bf Case 2-(iv).} $R({\cal M}) \simeq \mathbb{Z}_2$.

In this case, we must have $G({\cal M}) \simeq \mathbb{Z}_2^2$ or $G({\cal M}) \simeq \mathbb{Z}_4$.
In the case $G({\cal M}) \simeq \mathbb{Z}_2^2$, an orbit is isomorphic to  $A_{1,1}$, $A_{2,2}$ (recall the situation of Figure \ref{fig:order2}) or $A_{3,4}$.
In the case $G({\cal M}) \simeq \mathbb{Z}_4$,  an orbit is isomorphic to  $A_{1,1}$ or $_{-\{ 2,4 \}}A_{3,4}$.

Hence, in Case 2, we have one of  $G({\cal M}) \simeq \mathbb{Z}_n$ ($n \geq 2$), $G({\cal M}) \simeq \mathbb{Z}_n$ ($n \geq 2$) and  $G({\cal M}) \simeq S_4$ 
and a $G({\cal M})$-orbit is isomorphic to one of $A_{1,1}$, $A_{2,2}$, $A_{3,n}$ ($n \geq 3$) and $_{-[2n] \cap 2\mathbb{N}}A_{3,2n}$ ($n \geq 2$).
\\
\\
Combining the results in Cases 1 and 2, we have one of $G({\cal M}) \simeq \mathbb{Z}_n$ ($n \geq 1$), $G({\cal M}) \simeq D_{2n}$ ($n \geq 1$) and  $G({\cal M}) \simeq S_4$.
A $G({\cal M})$-orbit is isomorphic to one of $A_{1,1}$, $A_{2,2}$, $_{- \{ 2,4,\dots,2p-2\}}A_{2,2p-1}$ ($p \geq 2$), $A_{3,n}$ ($n \geq 3$) and $_{-[2n] \cap 2\mathbb{N}}A_{3,2n}$ ($n \geq 2$).
\end{proof}

For each $n \geq 3$, the cyclic group $\mathbb{Z}_n$ indeed appears as a (full) symmetry group.
For $k=0,1,\dots,n-1$, let $a_k^{(n)}:=  ( {\rm cos}(\frac{2k\pi}{n}), {\rm sin}(\frac{2k\pi}{n}))$
and $b_k:= \frac{1}{3}a^{(n)}_k+\frac{2}{3}a^{(n)}_{k+1}$ (where $a_{n}^{(n)}$ is interpreted as $a_0^{(n)}$).
Then, consider the point configuration $Q_n := (a_0^{(n)}, \dots, a_{n-1}^{(n)},  b_0^{(n)}, \dots, b_{n-1}^{(n)})$
and the associated oriented matroid ${\cal M}_{Q_n}$. Then, we have $G({\cal M}_{Q_n}) \simeq \mathbb{Z}_n$.
The dihedral group $D_{2n}$ also appears as $G(A_{3,n}) \simeq D_{2n}$.
Since $G(_{- \{ 2,4 \}}A_{3,4}) \simeq S_4$, the symmetric group $S_4$ is also the symmetry group of a simple oriented matroid of rank $3$.
This completes the classification of symmetry groups of simple oriented matroids of rank $3$.

\section{FPA rotational symmetry groups of simple acyclic oriented matroids of rank $4$}
\label{section:rank4}
In this section, we classify FPA rotational symmetry groups of simple acyclic oriented matroids of rank $4$.
The classification will be done, taking an idea from the classical approach to the classification of finite subgroups of the special orthogonal group $SO(3)$,
which we briefly review here. For more details, see \cite{S94}.

Let $G \subseteq SO(3)$ be a rotational symmetry group of a $3$-dimensional point configuration.
Each element $g \in G \setminus \{ \id \}$ fixes exactly two points, denoted by $c_{\sigma}$ and $-c_{\sigma}$, on the unit sphere $S^2$.
Let $S:=\{ c_{\sigma} \mid \sigma \in G\} \cup \{ -c_{\sigma} \mid \sigma \in G\}$.
Then $G$ acts on $S$.
If we denote by $r$ the number of  $G$-orbits in $S$,
we have
\begin{align*} 
r = \frac{1}{|G|}\sum_{\sigma  \in G}{f(\sigma)}
\end{align*}
by the Cauchy-Frobenius lemma, 
where $f(\sigma )$ is the number of fixed points of $\sigma$ in $S$.
If $\sigma = \id$, then $f(\sigma ) = |S|$ and otherwise $f(\sigma ) = 2$.
Therefore, we have
\begin{align*} 
r = \frac{1}{|G|}\{ |S| + 2(|G|-1) \}. 
\end{align*}
Let $S_1,\dots,S_r$ be the partition of $S$ induced by the action of $G$.
For $x_i \in S_i$, let $G_i$ be the stabilizer subgroup of $x_i$.
Then $|S_i|=\frac{|G|}{|G_i|}$.
Note that $|G_i| \geq 2$ and thus
\begin{align*} 
r = 2+\frac{1}{|G|}(|S_1|+ \dots + |S_r|-2) = 2 - \frac{2}{|G|} + (\frac{1}{|G_1|} + \dots + \frac{1}{|G_r|}) \leq 2 -\frac{2}{|G|} + \frac{r}{2}. 
\end{align*}
This leads to $2 \leq r \leq 3$.
If $r=2$, then $G$ is a cyclic group.
If $r=3$,  possible types of $(|G_1|,|G_2|,|G_3|,|G|)$ ($|G_1| \geq |G_2| \geq |G_3|$) are by further analysis classified into
\begin{align*} 
(n,2,2,2n) \ (n \geq 2), (3,3,2,12), (4,3,2,24), (5,3,2,60).
\end{align*}
This finally leads to that $G$ is isomorphic to one of  $\mathbb{Z}_n$ $(n \geq 1)$, $D_{2n}$ $(n \geq 1)$, $S_4$, $A_4$ and $A_5$.
\\
\\
Following the above discussion, a natural approach to classify rotational symmetry groups of simple acyclic oriented matroids of rank $4$
would be to construct, for each oriented matroid ${\cal M}$, a suitable set $S$ such that
$R({\cal M})$ acts on $S$ and each $\sigma \in R({\cal M}) \setminus \{ \id \}$ fixes exactly two elements of $S$.
For a FPA rotational symmetry group $R_f({\cal M})$, we will construct a desired set $S$ using the orbits under maximal cyclic subgroups of $R_f({\cal M})$.
In Section \ref{subsec:cond_cyclic}, we study the condition when $R_f({\cal M})$ is a cyclic group.
Then, in Section \ref{subsec:suitable}, we construct a set $S$ with the desired property, based on the orbits under cyclic subgroups of $R_f({\cal M})$.
Finally, the classification of FPA rotational symmetry groups is performed in Section \ref{subsec:classification}.
 
\subsection{Condition for a FPA subgroup of $R({\cal M})$ to be a cyclic group}
\label{subsec:cond_cyclic}
First of all, let us investigate,  as the simplest case, a condition that a FPA subgroup of $R({\cal M})$ is a cyclic group.
\begin{prop}
\label{prop:cyclic}
Let ${\cal M}=(E,\{ \chi, -\chi \})$ be a simple acyclic oriented matroid of rank $4$.
A FPA subgroup $G \subseteq R({\cal M})$ with $|G| > 2$ is a cyclic group if and only if
$\rank({\cal M}|_{G \cdot x}) = 3$  or  $\rank({\cal M}|_{G \cdot x}) = 1$ for all $x \in E$. 
\end{prop}
\begin{proof}
($\Leftarrow$)
Since $|G| > 2$, there exists $x \in E$ such that $\rank({\cal M}|_{G \cdot x}) = 3$ and
the group $\widetilde{G} := \{ \tau|_{G \cdot x} \mid \tau \in G\}$ is a subgroup of the symmetry group of the alternating matroid ${\cal M}|_{G \cdot x}$ of rank $3$.
Therefore, $\widetilde{G}$ (and thus $G$) is a dihedral group or a cyclic group.

Take $y \in E$ such that $\rank (G \cdot x \cup \{ y \} ) = 4$.
If $\rank (G \cdot y)=1$, let us consider ${\cal M}|_{G \cdot x \cup \{ y\}}/y$.
Since $y$ is a coloop of ${\cal M}|_{G \cdot x \cup \{ y\}}$, it holds that ${\cal M}|_{G \cdot x \cup \{ y\}}/y \simeq {\cal M}|_{G \cdot x}$
and that $\sigma|_{G \cdot x}$ is a rotational symmetry of ${\cal M}|_{G \cdot x}$ for all $\sigma \in G$. 
Then, $G$ is a cyclic group by Corollary \ref{cor:uniqueness}.
Now consider the case $\rank (G \cdot y)=3$.
Suppose that $G$ is a dihedral group and let $C_G$ be the maximal cyclic subgroup of it. 
Then, there exists $\tau \in G$ such that $\tau (y) = y$ ($\tau|_{G \cdot y}$ is the reflection symmetry of ${\cal M}|_{G \cdot y}$ fixing $y$).
Since $\tau|_{G \cdot x \cup \{ y \}} \in R({\cal M}|_{G \cdot x \cup \{ y \}})$
and since $y$ is a coloop of ${\cal M}|_{G \cdot x \cup \{ y \}}$, it holds that ${\cal M}|_{G \cdot x \cup \{ y \}}/y \simeq {\cal M}|_{G \cdot x}$ and that 
$\tau|_{G \cdot x}$ is a rotational symmetry of ${\cal M}|_{G \cdot x} (\simeq A_{3,|G \cdot x|})$.  
This leads to that $\tau|_{G \cdot x} \in C_G$, which is a contradiction. Therefore, $G$ is a cyclic group.
\\
($\Rightarrow$)
Let $\sigma$ be a generator of $G$ and
consider the oriented matroid ${\cal N}:={\cal M}|_{G \cdot x}$.
Then all elements of ${\cal N}$ are extreme points by transitivity.
We prove a contradiction by assuming $\rank({\cal N})=4$.

Let us consider an acyclic single element extension ${\cal N} \cup q$ of ${\cal N}$ with a fixed point $q$.
Then contract ${\cal N} \cup q$ by $q$ and obtain ${\widetilde {\cal N}}:= ({\cal N} \cup q)/q$.
Note that $\sigma$ is a proper rotational symmetry of ${\widetilde {\cal N}}$.
Also, note that the oriented matroid ${\widetilde {\cal N}}$ is not acyclic (and loopless) by Proposition \ref{prop:FP_inside}.
If ${\widetilde {\cal N}}$ is simple, then ${\widetilde {\cal N}}$ must be isomorphic to an alternating matroid 
by Proposition \ref{rank3_transitivity1}, which is a contradiction. Therefore, ${\widetilde {\cal N}}$ is not simple.
We remark that no two distinct elements are parallel  in ${\widetilde {\cal N}}$ 
by convexity of ${\cal N}$ (recall that $q$ is inside ${\cal N}$ by Proposition \ref{prop:FP_inside}).
Therefore, each element of ${\widetilde {\cal N}}$ has exactly one  antiparallel element and no parallel element (other than itself) as in Figure \ref{fig:non_simple}.
\begin{figure}[ht]
\begin{center}
\includegraphics[scale=0.30,clip]{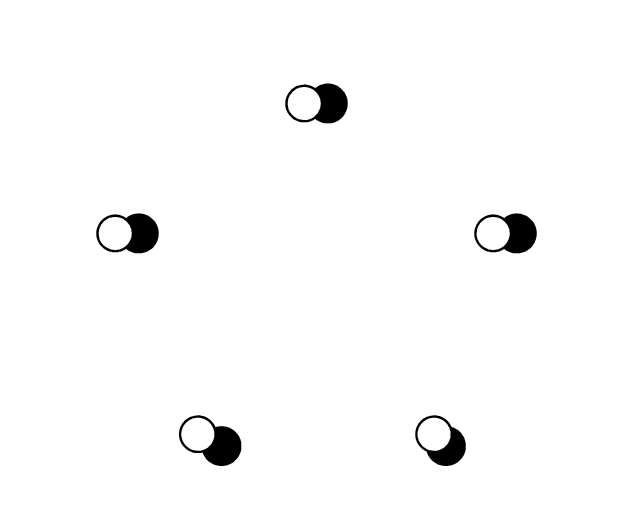}
\end{center}
\caption{non-simple case (without loops)}
\label{fig:non_simple}
\end{figure}

Let $a: E \rightarrow E$ be the map that takes each element of ${\widetilde {\cal N}}$
to its antiparallel element.
Let us take $x,y,z \in E$ such that $\chi(x,y,z) \neq 0$
Note that $\sigma (a(x)) = a(\sigma (x))$ (since $\sigma (x)$ and $\sigma (a(x))$ must be antiparallel).
Let $k > 0$ be the smallest number such that $\sigma^{k} (x) = a(x)$.
Because of transitivity, this number is common for all elements in $E$. Thus, we have
\[ \sigma^{k} \cdot \chi (x,y,z) = \chi (a(x),a(y),a(z))= - \chi (x,y,z). \]
It contradicts to the fact that $\sigma^k$ is a rotational symmetry of ${\widetilde {\cal N}}$.
This implies that the case $\rank({\cal N})=4$ never occurs.
As a conclusion, it holds that $\rank({\cal M}|_{G \cdot x}) \leq 3$ for all $x \in E$.
\end{proof}

\begin{cor}
\label{cor:rank3_restriction}
Let ${\cal M}$ be a simple acyclic oriented matroid of rank $4$ on a ground set $E$ and $G \subseteq R({\cal M})$ be a FPA rotational symmetry group with $|G| > 3$.
Suppose that $\rank (G \cdot x) =3$ or $\rank (G \cdot x) =1$ for any $x \in E$.
Then, for any $\sigma \in G$, the restriction $\sigma|_{G \cdot x}$ is a rotational symmetry of ${\cal M}|_{G \cdot x}$.
\end{cor}
\begin{proof}
Note that $G$ is a cyclic group by Proposition \ref{prop:cyclic}. Let $\tau$ be a generator of $G$.
Then, the order of $\tau|_{G \cdot x}$ is greater than $2$.
Since ${\cal M}|_{G \cdot x}$ is an alternating matroid of rank $3$, $\tau|_{G \cdot x}$ is a rotational symmetry of ${\cal M}|_{G \cdot x}$.
We remark that $\sigma = \tau^i$ for some $i \in \mathbb{Z}$.
Therefore, $\sigma|_{G \cdot x} = (\tau|_{G \cdot x})^i$ is also a rotational symmetry of  ${\cal M}|_{G \cdot x}$.
\end{proof}
\\
We conclude this subsection by proving the following lemma.
\begin{lem}
\label{lem:cyclic_subgroup}
Let ${\cal M}$ be a simple oriented matroid of rank $r$ on a ground set $E$, $H \subseteq R({\cal M})$ a subgroup of $R({\cal M})$,
$G_1,G_2 \subseteq H$ cyclic subgroups of $H$, 
and $X$ a subset of $E$ that is invariant under the actions of $G_1$ and $G_2$, and which satisfies $\rank(X) \geq r-1$.
If there is a cyclic group $G^X \subseteq R({\cal M}|_X)$ such that $G_1|_X (:= \{ \sigma|_X \mid \sigma \in G_1 \}) \subseteq G^X$ and
$G_2|_X (:= \{ \tau|_X \mid \tau \in G_2 \}) \subseteq G^X$,
then there exists a cyclic group $G \subseteq H$ such that $G_1 \subseteq G$ and $G_2 \subseteq G$.
\end{lem}
\begin{proof}
Let $\sigma$ and $\tau$ be generators of $G_1$ and $G_2$ respectively.
Since $\sigma|_X, \tau|_X \in G^X$, there exists $\pi \in R({\cal M}|_X)$ and 
$s,u \in \mathbb{N}$ such that $\sigma|_X = \pi^s$ and $\tau|_X = \pi^t$.
Let $l := {\rm gcd}(s,t)$. Then, there exist $a, b \in \mathbb{Z}$ such that $as+bt = l$.
Therefore, we have $\pi^l = (\sigma^a \tau^b)|_X \in R({\cal M}|_X)$.
Let $\omega := \sigma^a \tau^b \in H$ and then it holds that $\omega^{\frac{s}{l}} = \sigma$ and $\omega^{\frac{t}{l}} = \tau$ by Corollary \ref{cor:uniqueness}.
The cyclic group $G$ generated by $\omega$ is a subgroup of $H$ and contains $\sigma$ and $\tau$.
\end{proof}

\subsection{Construction of a set $S$ with the desired property}
\label{subsec:suitable}
For a simple acyclic oriented matroid ${\cal M}$ of rank $4$, let $R_f({\cal M})$ be a FPA subgroup of $R({\cal M})$.
We will construct a desired set using the orbits under maximal cyclic subgroups of $R_f({\cal M})$.
First, we shall show that each element of $R_f({\cal M}) \setminus \{ \id \}$ belongs to exactly one maximal cyclic subgroup of $R_f({\cal M})$.
\begin{prop}
\label{prop:maximal_cyclic}
Let ${\cal M}$ be a simple acyclic oriented matroid of rank $4$ on a ground set $E$, $R_f({\cal M})$ a FPA subgroup of $R({\cal M})$,
and $G_1,G_2$ distinct maximal cyclic subgroups of $R_f({\cal M})$.
Then we have $G_1 \cap G_2 = \{ \id \}$.
\end{prop}
\begin{proof}
Let $\sigma$ (resp. $\tau$) be a generator of $G_1$ (resp. $G_2$),
$H_1:= G_1 \cap G_2$ and $H_2$ the group generated by $G_1 \cup G_2$.
Since ${\cal M}$ is acyclic, both of $\sigma$ and $\tau$ are proper rotational symmetries.

If $|H_1| \geq 3$, there exists $x \in E$ such that $\rank(H_1 \cdot x)=3$.
Thus it holds that ${\rm span}_{\cal M}(G_1 \cdot x) = {\rm span}_{\cal M}(G_2 \cdot x)$.
Therefore, the restriction ${\cal M}|_{H_2 \cdot x}$ is a matroid polytope of rank $3$.
It follows from Lemma \ref{lem:cyclic_subgroup} that there exists a cyclic group $\widehat{G} \supseteq G_1,G_2$ with $\widehat{G} \subseteq R_f({\cal M})$, 
which is a contradiction.

If $|H_1| =2$, there exists $y \in E$ such that $\rank(H_1 \cdot y)=2$.
Let $Y:= H_1 \cdot y$ and $\widetilde{Y}_1 := G_1 \cdot y$, and  $\widetilde{Y}_2 := G_2 \cdot y$.
Since we have $|G_1|,|G_2| > |H_1| = 2$, it holds that $\rank(\widetilde{Y}_1) = \rank(\widetilde{Y}_2) = 3$.
Let $2a$ and $2b$ be the orders of $\sigma$ and $\tau$ respectively.
Then, we have $\sigma^a = \tau^b$ and $Y= \{ y, \sigma^a(y) \}$.
Note that $\sigma^a \tau = \tau \sigma^a$.
Also, note that ${\cal M}|_{\widetilde{Y}_2} \simeq A_{3,|{\widetilde{Y}}_2|}$ and that $\tau^b|_{\widetilde{Y}_2}$ is the rotational symmetry of order $2$, and 
thus that 
there exists a cocircuit $V$ of ${\cal M}|_{\widetilde{Y}_2}$ such that $V^0 \supseteq Y$ and $V(\tau (y)) = +$, and $V(\tau^{b+1}(y)) = -$.
This leads to that there exists a cocircuit $W$ of ${\cal M}$ such that $W^0 \supseteq {\widetilde Y}_2$ and  $W(\tau (y)) = +$, and $W(\tau^{b+1}(y)) = -$.
Therefore, $\tau|_{E \setminus {\widetilde Y}_2}$ is a reflection symmetry of ${\cal M}/{\widetilde Y}_2$ (see Section \ref{sec:useful}).
Since $\sigma^a$ is a rotational symmetry of ${\cal M}$ and since $\sigma^a|_Y$ is a rotational symmetry of ${\cal M}|_Y$ (Corollary \ref{cor:rank3_restriction}), 
it is a contradiction (see Section \ref{sec:useful}).

As a conclusion, $|H_1|=1$ must hold, which implies $G_1 \cap G_2 = \{ \id \}$.
\end{proof}
\\
\\
Therefore,
if we can construct a set on which a certain group action of $R_f({\cal M})$ is defined and
the stabilizer subgroup of each element in  $R_f({\cal M})$ is a maximal cyclic subgroup of $R_f({\cal M})$,
it has a desired property.
In the following, let us write $G_{\sigma}$ to denote the maximal cyclic subgroup of $R_f({\cal M})$ that contains $\sigma \in R_f({\cal M})$.
The maximal cyclic group $G_{\sigma}$ can be classified into two types, i.e.,
$G_{\sigma} \in G^I \cup G^{II}$, where $G^I$ and $G^{II}$ are defined as follows.
\begin{align*}
G^{I} &:= \{ G_{\sigma} \mid \sigma \in R_f({\cal M}), \  \rank(G_{\sigma} \cdot x)=1 \text{ or } 3 \text{ for all $x \in E$} \},\\
G^{II} &:= \{ G_{\sigma} \mid \sigma \in R_f({\cal M}), \ \rank(G_{\sigma} \cdot x)=1 \text{ or } 2 \text{ for all $x \in E$} \}.
\end{align*}
Note that $G_{\sigma} \in G^{II}$ if and only if $|G_{\sigma}|=2$.

\subsubsection*{(I) Designing a suitable set for $G^{I}$}
The first idea may be to consider the set $S(\sigma ):= \{ G_{\sigma} \cdot x \mid x \in E\}$ for each $\sigma \in R_f({\cal M})$.
However, there is a possibility that $S(\sigma )$ is fixed by an element $\tau \in R_f({\cal M}) \setminus G_{\sigma }$.
Indeed, let $Pr_3$ be the point configuration defined by
\begin{align*}
\begin{pmatrix}
{\rm cos}\frac{2\pi}{3} & {\rm cos}\frac{4\pi}{3} & {\rm cos}0 & {\rm cos}\frac{2\pi}{3} & {\rm cos}\frac{4\pi}{3} & {\rm cos}0 \\
{\rm sin}\frac{2\pi}{3} & {\rm sin}\frac{4\pi}{3} & {\rm sin}0 & {\rm sin}\frac{2\pi}{3} & {\rm sin}\frac{4\pi}{3} & {\rm sin}0 \\
1                               &    1                               &  1             &   -1                           &  -1                            &  -1
\end{pmatrix}
\end{align*}
and ${\cal M}_{Pr_3}$ the associated oriented matroid of $Pr_3$.
Then, ${\cal M}_{Pr_3}$ has a rotational symmetry 
$\sigma = (1 2 3) (4 5 6)$ and $\tau = ((1 4) (2 5) (3 6)) ((2 3) (5 6)) = (1 4) (2 6) (3 5)$.
The set of $G_{\sigma}$-orbits is $\{ \{ 1,2,3\}, \{ 4,5,6\} \}$.
This is fixed by $\tau \notin G_{\sigma}$.
\\
\begin{figure}[h]
\begin{center}
\includegraphics[bb =119 468 455 773, scale=0.25]{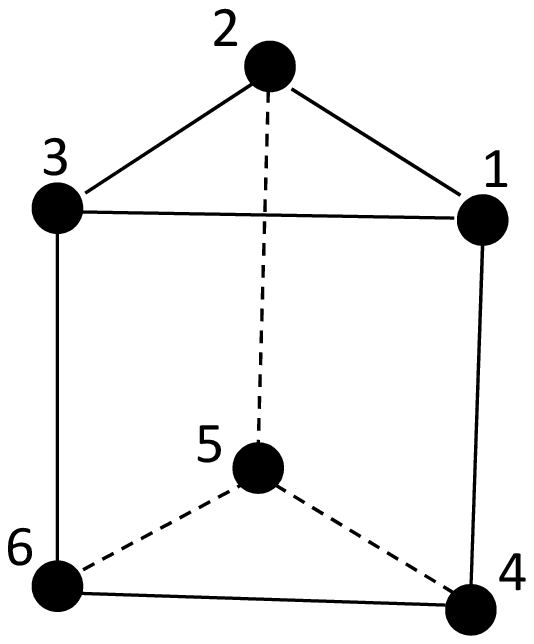}
\caption{Point configuration $Pr_3$}
\label{fig:prism}
\end{center}
\end{figure}
\\
The following proposition concerns a condition that such an event occurs.
\begin{prop}
\label{prop:orbit_action}
Let ${\cal M}$ be a simple acyclic oriented matroid of rank $4$, $R_f({\cal M})$ a FPA rotational symmetry group, 
$G$ a cyclic subgroup of $R_f({\cal M})$ generated by $\sigma (\neq \id )$,
and $X_1,X_2$ rank $3$ $G$-orbits.
If $\tau \cdot X_1 = X_2$ (i.e., $\tau^{-1}G\tau = G$) and $\sigma \tau \neq \tau \sigma$ for $\tau \in R_f({\cal M}) \setminus G$, then 
$\tau^2 =  \id$ and $(\sigma \tau)^2 = \id$.
If $\tau \cdot X_1 = X_2$ (i.e., $\tau^{-1}G\tau = G$) and $\sigma \tau = \tau \sigma$, then 
there is the cyclic subgroup $G_{\sigma, \tau}$ of $R_f({\cal M})$ that contains $\sigma$ and $\tau$.
\end{prop}
\begin{proof}
Let $H$ be the group generated by $\sigma$ and $\tau$, and $X := H \cdot x$ for $x \in X_1$.
Note that $X \supseteq X_1, X_2$.
\\
\\
{\bf (i)} $\sigma \tau \neq \tau \sigma$.
\begin{quote}
{\bf (i-a)} $\rank(X)=3$.

Note that $\sigma|_X$ is a rotational symmetry of ${\cal M}|_X$ (Corollary \ref{cor:rank3_restriction}).
Since ${\cal M}$ is acyclic, the restriction ${\cal M}|_{X}$ is also acyclic.
By transitivity, the oriented matroid ${\cal M}|_X$ is a matroid polytope of rank $3$, i.e., a relabeling of the alternating matroid $A_{3,|X|}$ (Proposition \ref{prop:rank_relabeling} in Appendix 1).
If $(\tau \sigma) |_X = (\tau \sigma) |_X$, then we have $\sigma \tau = \tau \sigma$ by Corollary \ref{cor:uniqueness}, which is a contradiction.
This implies that $\sigma|_X \tau|_X \neq \tau|_X \sigma|_X$ and that $\tau|_X$ is a reflection symmetry of ${\cal M}|_{X}$ ($\simeq A_{3,|X|}$).
Therefore, we have
$\tau^2|_X = \id$ and $(\sigma \tau)^2|_X = \id$, and thus
$\tau^2 = \id$ and $(\sigma \tau)^2 = \id$ by Corollary \ref{cor:uniqueness}.
\\
\\
{\bf (i-b)}  $\rank(X)=4$.

First of all, note that $(\sigma \tau )|_X \neq (\tau \sigma )|_X$ because
$(\sigma \tau)|_X = (\tau \sigma)|_X$ implies $\sigma \tau = \tau \sigma$.
By transitivity, ${\cal M}|_X$ is a matroid polytope.
Let us consider ${\cal N}:=({\cal M}|_X \cup q)/q$, where $q$ is a fixed point of $R_f({\cal M})$ such that ${\cal M}|_X \cup q$ is acyclic.
Every rotational symmetry $\sigma$ induces a rotational symmetry of ${\cal N}$.

If ${\cal N}$ is simple, then ${\cal N} \simeq A_{3,|X|}$ or ${\cal N} \simeq {_{-\{ 2,4,\dots,|X|\}}}A_{3,|X|}$, which leads to
$(\tau \sigma)^2 = \id$ and $\tau^2 = \id$.
Thus we assume that ${\cal N}$ is non-simple.
By a similar argument to the proof of Proposition \ref{prop:cyclic}, each element of ${\cal N}$ has exactly one antiparallel element 
and no parallel element (other than itself) as in Figure \ref{fig:non_simple}.
Let $a:X\rightarrow X$ be the map that takes each element of $X$ to its antiparallel element.
The permutation $a$ is a reflection symmetry of ${\cal N}$ (and ${\cal M}$).
Note that $a(\sigma (x)) = \sigma (a(x))$ and $a(\tau (x)) = \tau (a(x))$ for all $x \in X$ and thus that $\tau (a(X_1)) = a(X_2)$.

If there exists $x_0 \in X_1$ such that $\tau (x_0) \notin a(X_1)$, then we have $a(X_1) \cap X_2 = \emptyset$.
Thus the oriented matroid ${\cal N}|_{X_1 \cup X_2}$ (of rank $3$) is simple, which is a contradiction.

If $X_2 = a(X_1)$, then $(a^{-1}\tau)|_{X_1}$ is a symmetry of ${\cal N}|_{X_1}$.
If  $(a^{-1}\tau)|_{X_1}$ is a rotational symmetry of ${\cal N}|_{X_1}$, there exists $k \in \mathbb{N}$ such that
$(a^{-1}\tau)|_{X_1} = \sigma^k|_{X_1}$ and thus $a^{-1}\tau = \sigma^k$.
This implies that $\tau$ is a reflection symmetry of ${\cal M}$, which is a contradiction.
Therefore, $(a^{-1}\tau)|_{X_1}$ is a reflection symmetry of ${\cal N}|_{X_1}$ and thus we have
$(a^{-1}\tau)^2|_{X_1}=\id$ and $(a^{-1}\tau \sigma)^2|_{X_1}=\id$.
This yields $\tau^2 = \id$ and $(\tau \sigma)^2 = \id$ by Corollary \ref{cor:uniqueness}.
\end{quote}
{\bf (ii)} $\sigma \tau = \tau \sigma$.
\begin{quote}
{\bf (ii-a)} $\rank(X)=3$.

Similarly to Case (i-a), we have ${\cal M}|_X \simeq A_{3,|X|}$ and thus
there is a cyclic subgroup $G_{\sigma,\tau}^X$ of $R({\cal M}|_X)$ that contains $\tau|_X$ and $\sigma|_X$.
This implies that there exists a cyclic subgroup $G_{\sigma, \tau} \subseteq R_f({\cal M})$ with $\sigma \in G_{\sigma, \tau}$ and 
$\tau \in G_{\sigma, \tau}$ by Lemma \ref{lem:cyclic_subgroup}.
\\ \\
{\bf (ii-b)} $\rank(X)=4$.

By the same argument as Case (i-b),  there exits a cyclic subgroup $G_{\sigma, \tau} \subseteq R_f({\cal M})$ with
$\sigma \in G_{\sigma, \tau}$ and $\tau \in G_{\sigma, \tau}$.
\end{quote}
\end{proof}
\\
\\
We can remedy the situation by considering orderings of rank $3$ orbits instead of the set of them.
Let us pick up rank $3$ $G_{\sigma}$-orbits and consider their flats $F_1,\dots,F_m$.
We assume that there is no duplication in this list by removing some of the flats if necessary.
Then, consider flat orderings
$(F_{p(1)},\dots,F_{p(m)})$
that satisfy the condition in the following lemma.
\begin{lem}
\label{lem:ordering}
Let ${\cal M}$ be a simple acyclic oriented matroid of rank $4$ on a ground set $E$, $R_f({\cal M})$ a FPA subgroup of $R({\cal M})$,
and $G \subseteq R_f({\cal M})$ a cyclic subgroup of $R_f({\cal M})$ of order $q>2$.
Consider rank $3$ $G$-orbits $O_1,\dots,O_m$ and their flats $F_1:={\rm span}_{\cal M}(O_1),\dots,F_m:={\rm span}_{\cal M}(O_m)$.
We assume that there is no duplication in this list by removing some of them if necessary.
Suppose $m \geq 2$.
Then, there exist exactly two permutations $p$ on $[m]$ that satisfy the following condition:
for each $i=1,\dots,m$, there exists a covector $X_i$ of ${\cal M}$ such that
\begin{align*}
\begin{split}
&\text{$X_i(e)=0$ for all $e \in F_{p(i)}$,} \\
&\text{$X_i(f)=-$ for all $f \in F_{p(1)} \cup \dots \cup F_{p(i-1)}$,} \\
&\text{$X_i(g)=+$ for all $g \in F_{p(i+1)} \cup \dots \cup F_{p(m)}$.} \\
\end{split}
\end{align*}
\end{lem}
The lemma abstracts the geometric fact that all the flats of rank $3$ orbits are parallel
and can be ordered along the direction of the normal vector (as depicted in Figure \ref{fig:ordering}).
A proof is given in Appendix 2.
\\
\begin{figure}[h]
\begin{center}
\includegraphics[scale=0.25]{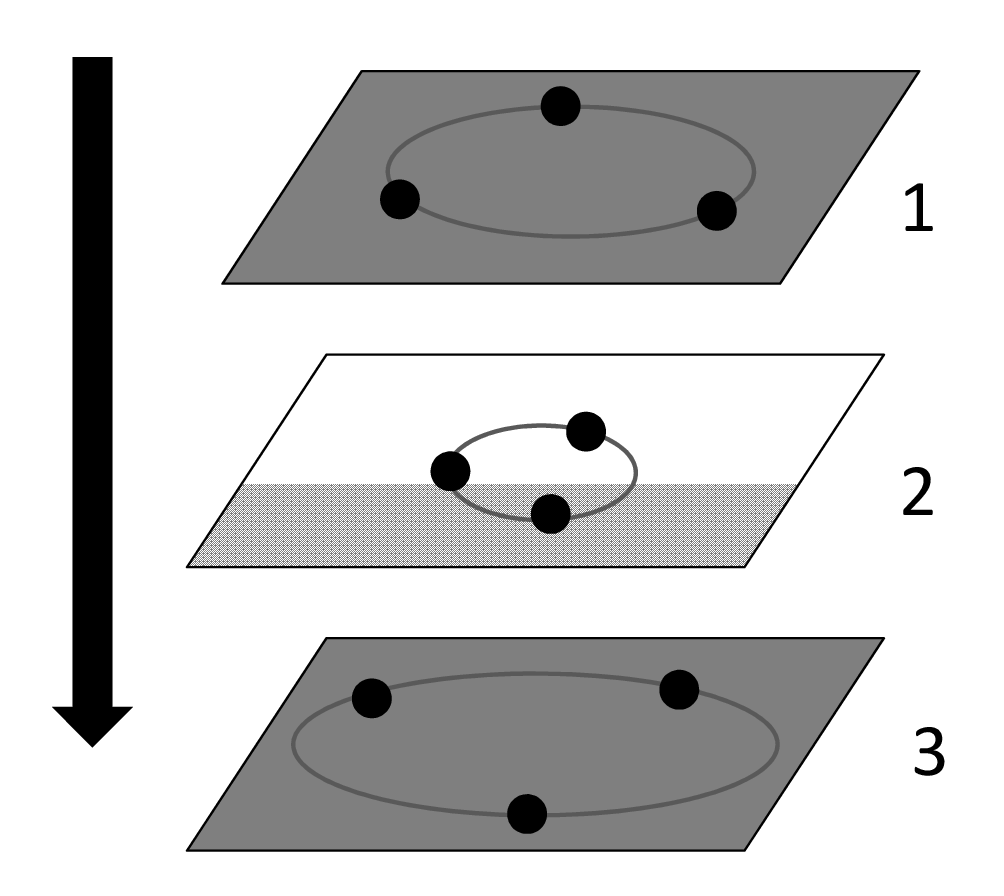}
\includegraphics[scale=0.25]{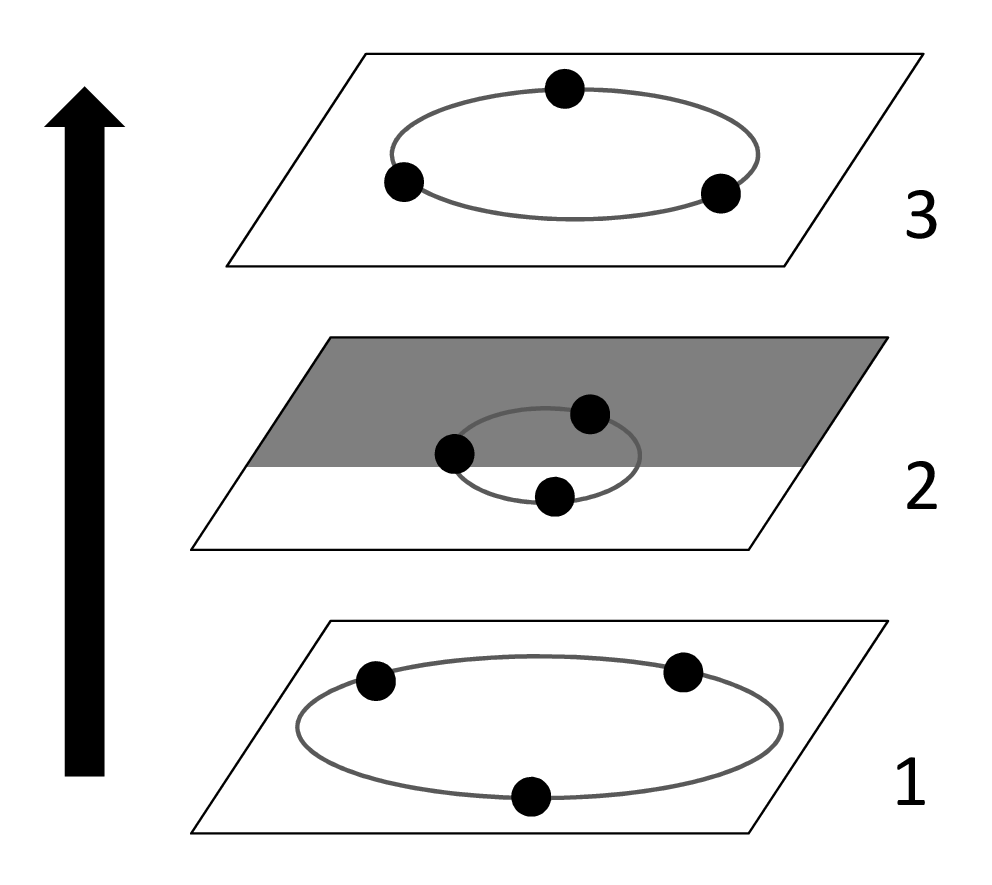}
\caption{two orderings of the flats of rank 3 orbits along the normal vectors}
\label{fig:ordering}
\end{center}
\end{figure}
\\
Let us write these two orderings as $T_1(\sigma )$ and $T_2(\sigma )$ (it is arbitrarily decided which one goes to $T_1(\sigma )$ and  to $T_2(\sigma )$).
If there is only one flat of rank $3$ orbit, we denote it by $F$.
Let $X_F$ be one of the two opposite cocircuits with $(X_F)^0=F$. Then, 
we set $T_1(\sigma ):= (( X_F)^-, (X_F)^0, (X_F)^+) $ and $T_2(\sigma ):=((X_F)^+, (X_F)^0, (X_F)^-)$.
Then, consider the collection of $T_1 (\sigma )$ and $T_2 (\sigma )$: 
\begin{align*}
S^{I} &:= \{ T_1 (\sigma ) \mid \text{ $\sigma \in R_f({\cal M})$, $G_{\sigma} \in G^{I}$}\} \cup \{ T_2 (\sigma ) \mid \text{  $\sigma \in R_f({\cal M})$, $G_{\sigma} \in G^{I}$}\}. 
\end{align*}
Let us consider the group action of $R_f({\cal M})$ on $S^{I}$ defined as follows:
For $\sigma \in R_f({\cal M})$ with $G_{\sigma} \in G^{I}$,  $\tau \in R_f({\cal M})$ and $x \in E$, let
\begin{align*} 
\tau \cdot \{ x, \sigma (x), \sigma^2 (x), \dots \} 
&:= \{ \tau (x), \tau (\sigma (x)), \tau (\sigma^2 (x)), \dots \}  \\
&= \{ \tau (x), \tau \sigma \tau^{-1} (\tau (x)), (\tau \sigma \tau^{-1})^2 (\tau (x)), \dots \} 
\end{align*}
and extend it component-wisely to $T_i(\sigma )$.
Clearly, the group action is well-defined, i.e., $T_i(\sigma_1) = T_i(\sigma_2) \Rightarrow \tau \cdot T_i(\sigma_1) = \tau \cdot T_i(\sigma_2)$
for all $\sigma_1,\sigma_2 \in R_f({\cal M})$, for $i=1,2$.
It holds that $\tau \cdot T_1 (\sigma ) = T_1 (\tau \sigma \tau^{-1})$ and $\tau \cdot T_2 (\sigma ) = T_2 (\tau \sigma \tau^{-1})$,
or that
$\tau \cdot T_1 (\sigma ) = T_2 (\tau \sigma \tau^{-1})$ and $\tau \cdot T_2 (\sigma ) = T_1 (\tau \sigma \tau^{-1})$.
\\
\\
The following lemma assures that there is a one-to-one correspondence between $S^{I}$ 
and $G^{I}$. 
\begin{prop}
Let $\sigma, \tau$ be generators of $G_{\sigma},G_{\tau} \in G^{I}$ respectively.
Then $T_i(\sigma ) = T_i(\tau ) \Leftrightarrow G_{\sigma} = G_{\tau}$, for $i=1,2$.
\end{prop}
\begin{proof}
($\Rightarrow$)
Let $X:= G_{\sigma} \cdot x$ be a rank $3$ orbit.
Since $X \in T_i (\tau)$, we have $X = G_{\tau} \cdot x$. 
Therefore, $\sigma|_X$ and $\tau|_X$ are rotational symmetries of ${\cal M}|_X$ (Corollary \ref{cor:rank3_restriction}).
Since ${\cal M}|_X \simeq A_{3,|X|}$, there exists a cyclic subgroup $G_{\sigma, \tau} \subseteq R_f({\cal M})$
that contains $\sigma$ and $\tau$ by Lemma \ref{lem:cyclic_subgroup}.
This leads to that $G_{\sigma, \tau} = G_{\sigma} = G_{\tau}$.
The ($\Leftarrow$) part is trivial.
\end{proof}
\\
By the following proposition, the stabilizer subgroup of $T_i(\sigma )$ in $R_f({\cal M})$ is $G_{\sigma }$ for $i=1,2$.
\begin{prop}
$\tau \cdot T_i(\sigma) = T_i(\sigma) \Leftrightarrow \tau \in G_{\sigma}$, for $i=1,2$.
\end{prop}
\begin{proof}
Without loss of generality, we assume that $\sigma$ is a generator of $G_{\sigma}$.
It is clear that $\tau \cdot T_i(\sigma)=T_i(\sigma)$ if $\tau \in G_{\sigma}$.
Now let us assume that $\tau \in R_f({\cal M}) \setminus G_{\sigma}$ and prove that $\tau \cdot T_i(\sigma) \neq T_i(\sigma)$.
Let $X$ be a rank $3$ $G_{\sigma}$-orbit.
If $\tau \cdot T_i(\sigma) = T_i(\sigma)$, we have $\tau \cdot X = X$.
All rotational symmetries of ${\cal M}|_X$ are generated by $\sigma|_X$ 
and thus $\tau|_X$ is not a rotational symmetry of ${\cal M}|_X$ but a reflection symmetry.
Since $\tau|_X$ is a reflection symmetry of ${\cal M}|_X$ and $\tau$ is a rotational symmetry of ${\cal M}$,
$\tau|_{E \setminus X}$ must be a reflection symmetry of ${\cal M}/X$ (recall Section \ref{sec:useful}).
Therefore, we have $V(x) = - V(\tau (x)) (\neq 0)$, where $V$ is one of the two opposite cocircuits with $V^0 = {\rm span}_{\cal M}(X)$.
This leads to that $\tau \cdot F \neq F$ for $F \in T_i(\sigma) \setminus \{ {\rm span}_{\cal M}(X)\}$.
Therefore, we have $\tau \cdot T_i(\sigma) \neq T_i(\sigma)$. 
\end{proof}

\subsubsection*{(II) Designing a suitable set for $G^{II}$}

As a first step, let us consider the set $S_{pre}(\sigma ) := \{ O \mid \text{$O \subseteq E$ is a $G_{\sigma}$-orbit}\}$ for each
$\sigma \in R_f({\cal M})$ with $G_{\sigma} \in G^{II}$ and
collect $S_{pre}(\sigma )$ for $\sigma \in R_f({\cal M})$ with $G_{\sigma} \in G^{II}$:
\begin{align*}
S_{pre}^{II} :=\{ S_{pre}(\sigma ) \mid \text{ $\sigma \in R_f({\cal M}), G_{\sigma} \in G^{II}$}\}. 
\end{align*}
Let us consider the group action of $R_f({\cal M})$ on $S_{pre}^{II}$ defined as follows.
For  $\sigma \in R_f({\cal M})$ with $G_{\sigma} \in G^{II}$, $\tau \in R_f({\cal M})$ and $x \in E$, let
\[ \tau \cdot \{ x, \sigma (x) \} := \{ \tau (x), \tau \sigma \tau^{-1}(\tau (x)) \} \]
and extend it element-wisely to $S_{pre}(\sigma )$.
Clearly, the group action is well-defined, i.e., $S_{pre}(\sigma_1) = S_{pre}(\sigma_2) \Rightarrow \tau \cdot S_{pre}(\sigma_1) = \tau \cdot S_{pre}(\sigma_2)$
for all $\sigma_1,\sigma_2 \in R_f({\cal M})$.
The group action is closed in $S_{pre}^{II}$.

The following proposition shows that
for a rotational symmetry $\sigma \in R_f({\cal M})$ with $G_{\sigma} \in G^{II}$,
the stabilizer subgroup of $S_{pre}(\sigma )$ in $R_f({\cal M})$ equals to $G_{\sigma}$ except a few cases.
\begin{prop}
Let ${\cal M}$ be a simple acyclic oriented matroid of rank $4$ on a ground set $E$, $R_f({\cal M})$ a FPA rotational symmetry group of ${\cal M}$.
Let $\sigma \in R_f({\cal M})$ be such that $G_{\sigma} \in G^{II}$
and $X,Y$ rank $2$ $G_{\sigma}$-orbits.
If $\tau \cdot S_{pre}(\sigma ) = S_{pre}(\sigma )$, then the group $H$ generated by $\sigma$ and
$\tau$ is isomorphic to $\mathbb{Z}_2$ or $\mathbb{Z}_2^2$.
\end{prop}
\begin{proof}
Let $X:=\{ x, \sigma (x) \}$, $Y:= \{ y, \sigma (y) \}$
and $\tau \cdot X = Y$.
If $\tau (x) = y$, then we have $\tau \sigma (x) = \sigma (y) = \sigma \tau (x)$.
If $\tau \sigma (x) = y$, then it holds that $\sigma \tau \sigma (x) = \tau (x)$.
Thus we have $(\tau \sigma)|_{X \cup Y}= (\sigma \tau)|_{X \cup Y}$ in any cases.
We remark that $X \cup Y$ is an $H$-orbit.

If $\rank (X \cup Y) \geq 3$, we have  $\sigma \tau = \tau \sigma$ by Corollary \ref{cor:uniqueness}.
Let $p$ be the order of $\tau$.
If $\rank(X \cup Y)=3$ and $p > 2$, we have $H \simeq \mathbb{Z}_{2p}$, which contradicts to the maximality assumption of $G_{\sigma}$.
If $p=2$, then $H$ is isomorphic to one of $\mathbb{Z}_4$, $D_4 (\simeq \mathbb{Z}_2^2)$ and $\mathbb{Z}_2$.
If $\rank(X \cup Y)=4$, let us consider the contraction by a fixed point of $R_f({\cal M})$.
If $p > 2$, we have $H \simeq \mathbb{Z}_{2p}$ by the same argument as Proposition \ref{prop:orbit_action}, which is a contradiction.
If $p=2$, then we have $H \simeq \mathbb{Z}_4$ or $H \simeq \mathbb{Z}_2^2$, or $H \simeq \mathbb{Z}_2$.
Since $H \simeq \mathbb{Z}_4$ contradicts to the maximality assumption, we have $H \simeq \mathbb{Z}_2$ or $H \simeq \mathbb{Z}_2^2$.

Now let us consider the case $\rank(X \cup Y) = 2$.
Then, we have $X=Y$ and $\sigma |_X = \tau |_X$.
Here, note that
$\chi (e,f,x,\sigma(x)) = \sigma \cdot \chi(e,f,x,\sigma (x)) = -\chi (e,f,x,\sigma (x))$ and thus that $\chi (e,f,x,\sigma (x)) = 0$ for $e,f \in \Fix(\sigma)$, 
where $\chi$ is a chirotope of ${\cal M}$.
Therefore, it holds that
\begin{align*}
 \rank(\Fix(\sigma)) \leq 2, \ \rank(\Fix(\sigma ) \cup X) \leq 3.
 \end{align*}
Thus, there must exist a rank $2$ $G_{\sigma}$-orbit $X'$ with $\rank(X \cup X') \geq 3$.
Let $Y':=\tau \cdot X'$.
If $X'=Y'$, then $\tau|_{X \cup X'} = \sigma |_{X \cup X'}$.
Since $\rank(X \cup X') \geq 3$, we have $\tau = \sigma$ by Corollary \ref{cor:uniqueness}, which is a contradiction.
This concludes that $X' \neq Y'$ and that $\rank(X' \cup Y') \geq 3$.
Therefore, we obtain $H \simeq \mathbb{Z}_2$ or $H \simeq \mathbb{Z}_2^2$ also in this case by replacing $X,Y$ by $X',Y'$.
\end{proof}
\\
The set $S_{pre}(\sigma)$ may be fixed by $\tau \notin G_{\sigma}$ if 
the group generated by $\sigma$ and $\tau$ is isomorphic to $\mathbb{Z}_2^2$.
Indeed, let $BP_4$ be the point configuration defined by
\begin{align*}
\begin{pmatrix}
1 & 0 & -1 & 0  & 0 & 0 \\
0 & 1 & 0  & -1 & 0 & 0 \\
0 & 0 & 0  & 0  &  1 & -1
\end{pmatrix}
\end{align*}
and  ${\cal M}_{BP_4}$ be the associated oriented matroid of $Bi_4$.
It has rotational symmetries $\sigma := (2 4) (5 6)$ and $ \tau := (1 3) (2 4)$.
Then, we have $S_{pre}(\sigma) = \{ \{ 2, 4\}, \{ 5, 6\}, \{ 1 \}, \{ 3 \} \}$.
This is fixed by $\tau$.
\\
\begin{figure}[h]
\begin{center}
\includegraphics[bb = 80 189 504 761, scale=0.25]{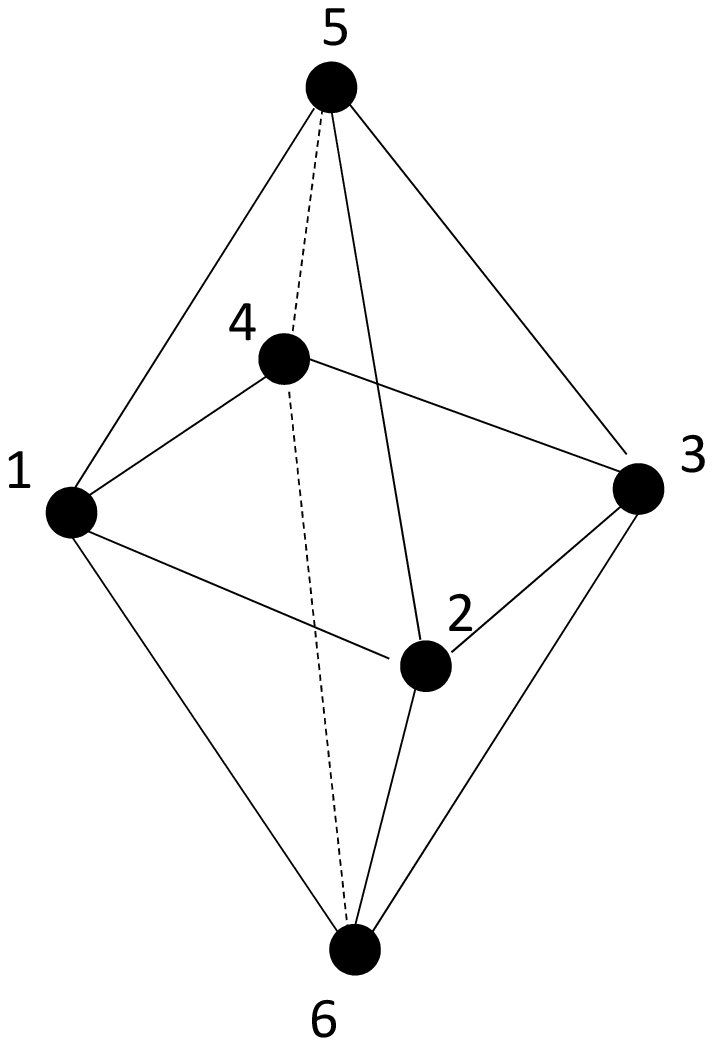}
\caption{Point configuration $BP_4$}
\label{fig:bypiramid}
\end{center}
\end{figure}
\\
The following two propositions show that there are not two such rotational symmetries $\tau$.

\begin{prop}
Let ${\cal M}$ be a simple acyclic oriented matroid of rank $4$ with a FPA rotational symmetry group $R_f({\cal M})$.
If there exist $\sigma,\tau,\pi \in R_f({\cal M})$ of order $2$ such that
the group generated by $\sigma$ and $\tau$, and that generated by $\sigma$ and $\pi$ are isomorphic to $\mathbb{Z}_2^2$,
then there is a cyclic group $G \subseteq R_f({\cal M})$ such that $\sigma \in G$ and $\pi \tau \in G$, or
the group generated by $\sigma$, $\tau$ and $\pi$ is isomorphic to $\mathbb{Z}_2^3$.
\end{prop}
\begin{proof}
Let $p$ be the order of $\pi \tau$.
Note that $\sigma (\pi \tau) = (\pi \tau )\sigma$.
If $p > 2$, then there is a cyclic group $G \subseteq R_f({\cal M})$ such that $\sigma \in G$ and $\pi \tau \in G$ by Proposition \ref{prop:orbit_action}.
If $p=2$, then the group generated by $\tau$ and $\pi$ is isomorphic to $\mathbb{Z}_2^2$.
\end{proof}
\\
Note that only the case $R_f({\cal M}) \simeq \mathbb{Z}_2^3$ is possible under the assumption that $G_{\sigma} \in G^{II}$.
However, this is also proved to be impossible by the following proposition.
\begin{prop}
Let ${\cal M}$ be a simple acyclic oriented matroid of rank $4$  with a FPA rotational symmetry group $R_f({\cal M})$.
Then $R_f({\cal M})$ does not contain a subgroup isomorphic to $\mathbb{Z}_2^3$.
\end{prop}
\begin{proof}
A proof can be done similarly to Proposition \ref{prop:rank3_transitivity3}.
Suppose that there exists $G \subseteq R_f({\cal M})$ such that $G \simeq \mathbb{Z}_2^3$.
Let $\sigma,\tau,\pi$ be generators of $G$ and $X$ be a $G$-orbit.
Then, we have $\rank(X)=3$ or $\rank(X)=4$.
If $\rank(X)=3$, it is a contradiction since $G({\cal M}|_X)$ does not contain a subgroup isomorphic to $\mathbb{Z}_2^3$.
If $\rank(X)=4$, consider a fixed point $q$ of $G$ such that ${\cal M} \cup q$ is acyclic, and let ${\cal N} := ({\cal M}|_X \cup q) /q$.
Then, ${\cal N}$ is not simple
since the symmetry group of a simple oriented matroid of rank $3$ cannot have a subgroup isomorphic to $\mathbb{Z}_2^3$.
Note that each element of ${\cal N}$ has exactly one antiparallel element and no parallel element (other than itself).
For $x \in X$, there exists $a \in R({\cal N})$ such that $a(x)$ is the antiparallel element of $x$.
Then $a$ must be the map that takes each element of ${\cal N}$ to its antiparallel element.
This leads to that $a$ is a reflection symmetry of ${\cal N}$, which is a contradiction.
\end{proof}
\\
\\
Based on the above observation, let us design a set $S^{II}$.
For a rank $2$ $G_{\sigma}$-orbit $O_0:= \{ x,\sigma(x)\}$ with $O_0 \neq \tau (O_0)$, 
consider the pairs
$(O_0,\tau(O_0))$ and $(\tau (O_0),O_0)$, where $\tau$ is a rotational symmetry such that the group generated by $\sigma$ and $\tau$
is isomorphic to $\mathbb{Z}_2^2$:
\begin{align*}
S_1(\sigma ) &:= S_{pre}(\sigma) \cup \{ (O_0,\tau(O_0)) \}, \\
S_2(\sigma ) &:= S_{pre}(\sigma) \cup \{ (\tau (O_0),O_0) \}.
\end{align*}
If there does not exist $\tau \in R_f({\cal M})$ such that the group generated by $\sigma$ and $\tau$
is isomorphic to $\mathbb{Z}_2^2$, we set $S_1(\sigma)=(1,S_{pre}(\sigma))$ and $S_2(\sigma)=(2,S_{pre}(\sigma))$.
Then, the stabilizer subgroup of $S_1(\sigma)$ and that of $S_2(\sigma)$ are equal to $G_{\sigma}$.
Note that $\tau \cdot S_1 (\sigma) = S_2(\sigma)$ and that $\tau \cdot S_2(\sigma) = S_1(\sigma)$
if $\tau \in R_f({\cal M}) \setminus G_{\sigma}$ satisfies $\tau \cdot S_{pre}(\sigma) = S_{pre}(\sigma)$.
Let us define $S^{II}$ as the collection of $S_1(\sigma)$ and $S_2(\sigma)$:
\[ S^{II}:= \{ S_1(\sigma) \mid \sigma \in R_f({\cal M}), G_{\sigma} \in G^{II}\} \cup \{ S_2(\sigma) \mid \sigma \in R_f({\cal M}), G_{\sigma} \in G^{II}\}. \]
The following lemma assures that there is a one-to-one correspondence between $S^{II}$ and $G^{II}$. 
\begin{lem}
Let $\sigma, \tau \in R_f({\cal M})$ be a rotational symmetry  with $G_{\sigma}, G_{\tau} \in G^{II}$.
Then $S_i(\sigma ) = S_i(\tau ) \Rightarrow \sigma = \tau$, for $i=1,2$.
\end{lem}
\begin{proof}
Let $\{ x, \sigma (x) \}$ be a rank $2$ $G_{\sigma}$-orbit.
Then $S_i(\sigma ) = S_i(\tau )$ leads to $\{ x, \sigma (x) \} = \{ x, \tau (x) \}$ and thus to $\sigma (x) = \tau (x)$.
For a rank $1$ orbit $\{ y \}$ (when $\sigma (y) =y$), we have $\tau (y) = y$.
Therefore, we have  $\sigma (y) = \tau (y)$ if $S_i(\sigma ) = S_i(\tau )$.
The same applies to all the orbits and thus $\sigma = \tau$.
\end{proof}

\subsection{Classification}
\label{subsec:classification}
Now let us classify FPA rotational symmetry groups of simple acyclic oriented matroids of rank $4$.
From here on, we simply write $G$ instead of $R_f({\cal M})$.

Consider the group action of $G$ on the set $S:= S^{I} \cup S^{II}$.
Let $r$ be the number of $G$-orbits in $S$.
First, recall the well-known formula (Cauchy-Frobenius lemma):
\[ r = \frac{1}{|G|}\sum_{\sigma  \in G}{f(\sigma)},\] 
where $f(\sigma )$ is the number of fixed points of $\sigma$ in $S$.
If $\sigma = \id$, then $f(\sigma ) = |S|$ and otherwise $f(\sigma ) = 2$.
\begin{lem}
$f(\sigma ) = 2$ for all $\sigma \in G \setminus \{ \id \}$.
\end{lem}
\begin{proof}
Clearly, it holds that $\sigma \cdot S(\sigma' )= S(\sigma' )$, where $\sigma'$ is a generator of a maximal cyclic subgroup of $G$ to which
$\sigma$ belongs if the order of $\sigma'$ is $2$.
The same applies to $T_1(\sigma' )$ and $T_2(\sigma')$ if the order of $\sigma'$ is greater than $2$.
Therefore, we have $f(\sigma ) \geq 2$.
It is trivial that $f(\sigma ) \leq 2$.
\end{proof}
\\
Therefore, we have
\begin{align*} 
r = \frac{1}{|G|}\{ |S| + 2(|G|-1) \}. 
\end{align*}
For $x_i \in S_i$, let $G_i$ be the stabilizer subgroup of $x_i$.
Then $|S_i|=\frac{|G|}{|G_i|}$.
Note that $|G_i| \geq 2$ and thus
\begin{align*} 
r = 2+\frac{1}{|G|}(|S_1|+ \dots + |S_r|-2) = 2 - \frac{2}{|G|} + (\frac{1}{|G_1|} + \dots + \frac{1}{|G_r|}) \leq 2 -\frac{2}{|G|} + \frac{r}{2}. 
\end{align*}
It follows that $r < 4$.
Also, note that $r \geq 2$.
Therefore, we have $r=2$ or $r=3$.
If $r=2$, then $|S_1|+|S_2|=2$ and thus $|S_1|=|S_2|=1$.
Let $s \in S_1$.
Then $G \cdot s = s$ and thus $G$ is a cyclic group.
If $r=3$, we have 
\begin{align*}
1 < 1+\frac{2}{|G|} = \frac{1}{|G_1|} + \frac{1}{|G_2|} + \frac{1}{|G_3|}.
\end{align*}
Therefore, possible types of $(|G_1|,|G_2|,|G_3|,|G|)$ ($|G_1| \geq |G_2| \geq |G_3|$) are classified into
\begin{align*} 
(n,2,2,2n) \ (n \geq 2), (3,3,2,12), (4,3,2,24), (5,3,2,60).
\end{align*}
\noindent
(i) $(|G_1|,|G_2|,|G_3|,|G|) = (n,2,2,2n)$ ($n \geq 2$).

In this case, we have
$|S_1| = \frac{2n}{n} = 2$, $|S_2| = |S_3| =\frac{2n}{2}=n$ and $|S|=2n+2$.
Let $g,h_1,h_2$ be generators of $G_1,G_2,G_3$ respectively.
Let $\{ s,t \} := S_1$. 

If $n = 2$, then we have $G \simeq \mathbb{Z}_2^2 \simeq D_4$.
If $n \geq 3$, then we have $s,t \in S^{I}$ and thus the flat ordering $t$ must be the reverse ordering of $s$.
We have $h_1 \cdot s = t$ and $h_2 \cdot s = t$.
Therefore, $h_1h_2^{-1} \cdot t = t$ and thus $h_1h_2^{-1} \in G_1$.
This implies that $G$ is generated by $g$ and $h_1$.
Since $h_1 \cdot X_1 = X_2$ for some rank $3$ $G_1$-orbits $X_1,X_2$,
we have $h_1^2 = \id$ and $(h_1g)^2=\id$ by Proposition \ref{prop:orbit_action}. 
Therefore, the elements of $G$ are covered by $\id,g,g^2,\dots,g^{n-1},h_1g,h_1g^2,\dots,h_1g^{n-1}$.
Since $|G|=2n$, all of them must be distinct. 
Therefore, $G$ is isomorphic to the dihedral group $D_{2n}$.
\\
\\
(ii) $(|G_1|,|G_2|,|G_3|,|G|) = (3,3,2,12)$.

A group of order $12$ is isomorphic to one of $\mathbb{Z}_{12},D_{12},A_4,\mathbb{Z}_2 \times \mathbb{Z}_6$ and $Q_{12}$
(One can check it by using the GAP Small Groups Library~\cite{BEO}, for example).
Here, $Q_{4n}$ is the dicyclic group defined by $\langle x,y \mid x^{2n} = \id, x^2 = y^n, yxy^{-1}=x^{-1} \rangle$.
Since $G$ has two non-conjugate cyclic subgroups of order $3$,
$G$ is isomorphic to none of $\mathbb{Z}_{12}$, $D_{12}$, $\mathbb{Z}_2 \times \mathbb{Z}_6$ and $Q_{12}$.
The remaining possibility is $G \simeq A_4$.
\\
\\
(iii)  $(|G_1|,|G_2|,|G_3|,|G|) = (4,3,2,24)$.

A group of order $24$ is isomorphic to one of $\mathbb{Z}_{24}$, $\mathbb{Z}_2 \times \mathbb{Z}_{12}$, $\mathbb{Z}_2^2 \times \mathbb{Z}_6$, $D_{24}$, $Q_{24}$, $\mathbb{Z}_2 \times D_{12}$,
$\mathbb{Z}_2 \times Q_{12}$, $\mathbb{Z}_2 \times A_4$, $\mathbb{Z}_3 \times D_8$, $\mathbb{Z}_3 \times Q_8$, $\mathbb{Z}_4 \times D_6$, $SL(2,3):= \langle a,b,c \mid a^3=b^3=c^2=abc \rangle$, $S_4$,
$P:=\langle a,b \mid a^3=b^8=\id, bab^{-1}=a^{-1}\rangle$ and $Q:=\langle a,b \mid a^3=b^4=c^2=\id,bab^{-1}=a^{-1}, cbc^{-1}=b^{-1},ac=ca \rangle$
(One can check it by using the GAP Small Groups Library~\cite{BEO}, for example).

By Proposition \ref{prop:orbit_action}, $G$ is isomorphic to none of
$\mathbb{Z}_2 \times \mathbb{Z}_{12}$, $\mathbb{Z}_2^2 \times \mathbb{Z}_6$, $Q_{24}$, $\mathbb{Z}_2 \times D_{12}$, $\mathbb{Z}_2 \times Q_{12}$,
$\mathbb{Z}_2 \times A_4$, $\mathbb{Z}_3 \times D_8$, $\mathbb{Z}_3 \times Q_8$, $\mathbb{Z}_4 \times D_6$, $P$ and $Q$.
Since the order of each maximal cyclic subgroup of $G$ must be one of $4,3$ and $2$,
we have $G \not\simeq \mathbb{Z}_{24}$ and $G \not\simeq D_{24}$.
Since $SL(2,3)$ contains (more than) two non-conjugate cyclic subgroups of order $3$, we have $G \not\simeq SL(2,3)$.
The remaining possibility is $G \simeq S_4$.
\\
\\
(iv)  $(|G_1|,|G_2|,|G_3|,|G|) = (5,3,2,60)$.

A group of order $60$ is isomorphic to one of $\mathbb{Z}_{60}$, $\mathbb{Z}_2 \times \mathbb{Z}_{30}$, $D_{60}$, $Q_{60}$, $\mathbb{Z}_3 \times D_{20}$, $\mathbb{Z}_3 \times Q_{20}$,
$\mathbb{Z}_5 \times D_{12}$, $\mathbb{Z}_5 \times Q_{12}$, $\mathbb{Z}_5 \times A_4$, $D_6 \times D_{10}$, $A_5$,
$R:=\langle a,b \mid a^{15}=b^4=\id,bab^{-1}=a^2 \rangle$ and $S:=\langle a,b \mid a^{15}=b^4=\id,bab^{-1}=a^7\rangle$
(One can check it by using the GAP Small Groups Library~\cite{BEO}, for example).

By Proposition \ref{prop:orbit_action}, $G$ is isomorphic to none of $\mathbb{Z}_2 \times \mathbb{Z}_{30}$, $\mathbb{Z}_3 \times D_{20}$, $\mathbb{Z}_3 \times Q_{20}$,
$\mathbb{Z}_5 \times D_{12}$, $\mathbb{Z}_5 \times Q_{12}$, $\mathbb{Z}_5 \times A_4$, $D_6 \times D_{10}$, $R$ and $S$.
Since the order of each maximal cyclic subgroup of $G$ is one of $5,3$ and $2$,
we have $G \not\simeq \mathbb{Z}_{60}$ and $G \not\simeq D_{60}$, and $G \not\simeq Q_{60}$.
The remaining possibility is $G \simeq A_5$.
\\
\\
Therefore, a FPA rotational symmetry group of a simple acyclic oriented matroid of rank $4$ must be isomorphic to one of
$\mathbb{Z}_n$ ($n \geq 1$), $D_{2n}$ ($n \geq 1$), $A_4$, $S_4$ and $A_5$.

Next, we prove that all of the above groups are indeed FPA rotational symmetry groups of some simple acyclic oriented matroids of rank $4$.
\\
\\
(i) $G \simeq \mathbb{Z}_n$.

For $n=1, 2,3,\dots$, let us define $3$-dimensional point configurations $P_n$ as follows.
\begin{align*}
P_1&:=\{ (0,0,0), (0,1,0), (1,0,0), (1/2,1/2,0), (0,0,1)\}, \\
P_2&:=\{ (-1,1,0),(0,1,0),(1,1,0), (-1,-1,0),(0,-1,0),(1,-1,0), (0,0,1)\}, \\
P_n&:=\{ ({\rm cos}\frac{2k}{n}\pi,{\rm sin}\frac{2k}{n}\pi,-1) \mid k \in [n] \} \cup 
\{ ({\rm cos}\frac{2k-1}{n}\pi,{\rm sin}\frac{2k-1}{n}\pi,1) \mid k \in [n] \} \ (n \geq 3),
\end{align*}
and ${\cal M}_{P_n}$ be the associated oriented matroid of $P_n$.
Then, we have $G=R({\cal M}_{P_n}) \simeq \mathbb{Z}_n$.
\\
\\
(ii) $G \simeq  D_{2n}$.

Since $D_2 \simeq \mathbb{Z}_2$, we consider the case $n \geq 2$.
For $n = 2,3,\dots$, let us define $3$-dimensional point configurations $Q_n$ as follows.
\begin{align*}
Q_2 &:= \{ (2,0,0), (0,2,0), (-2,0,0), (0,-2,0), (1,0,0), (-1,0,0), (0,0,2), (0,0,-2) \}, \\
Q_n&:=\{ ({\rm cos}(\frac{2k\pi}{n}),{\rm sin}(\frac{2k\pi}{n}),0) \mid k \in [n] \} \cup \{ (0,0,-1), (0,0,1) \} \ (n \geq 3)
\end{align*}
and
${\cal M}_{Q_n}$ be the associated oriented matroid of $Q_n$.
Note that $R({\cal M}_{Q_n}) \supseteq R(Q_n) \simeq D_{2n}$, where $R(Q_n)$ is the geometric rotational symmetry group of $Q_n$.
By the above discussion, $R({\cal M}_{Q_n})$ must be isomorphic to one of $\mathbb{Z}_k$ ($k \geq 1$), $D_{2k}$ ($k \geq 1$), $A_4$, $S_4$ and $A_5$.
Simple discussion yields that $G=R({\cal M}_{Q_n}) \simeq D_{2n}$.
\\
\\
(iii) $G \simeq  A_4$ or $G \simeq  S_4$, or $G \simeq A_5$.

The associated oriented matroid of the $3$-simplex (resp. the $3$-cube, Icosahedron) has the rotational symmetry group isomorphic to 
$A_4$ (resp. $S_4$, $A_5$).
\begin{figure}[h]
\begin{center}
\includegraphics[scale=0.30,clip]{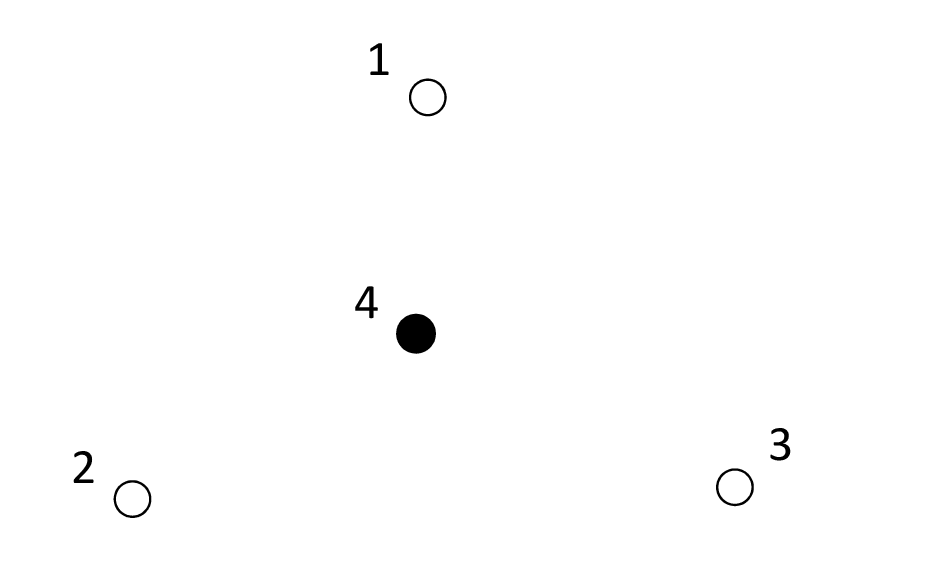}
\includegraphics[scale=0.30,clip]{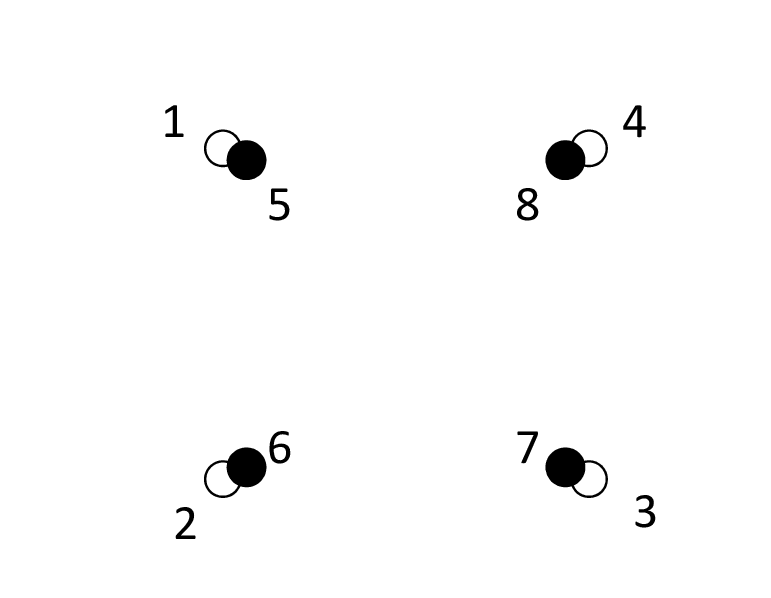}
\includegraphics[scale=0.30,clip]{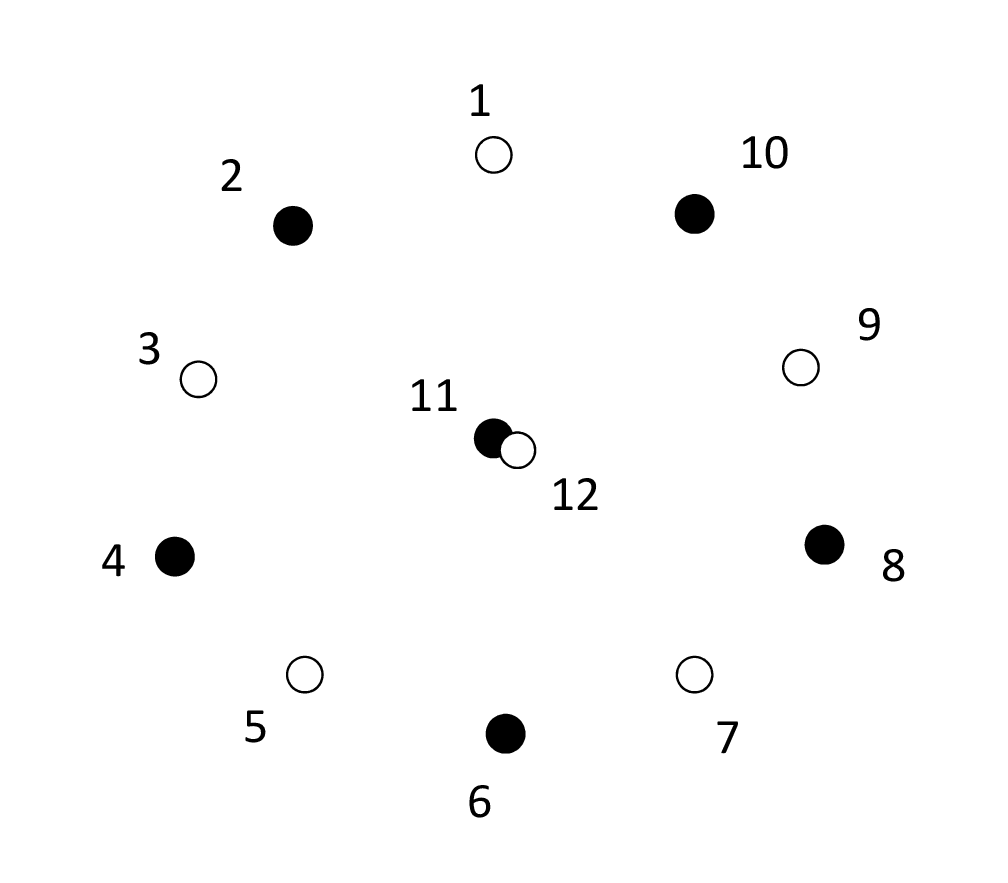}
\end{center}
\caption{The $3$-simplex [left], the $3$-cube [center] and Icosahedron [right] contracted by each fixed point (depicted as signed point configurations)}
\label{fig:regular}
\end{figure}

\begin{thm}
\label{main_thm}
Let ${\cal M}$ be a simple acyclic oriented matroid of rank $4$ and $R_f({\cal M})$ a FPA subgroup of $R({\cal M})$.
Then, $R_f({\cal M})$ is isomorphic to one of 
$\mathbb{Z}_n$ ($n \geq 1$), $D_{2n}$($n \geq 1$), $A_4$, $S_4$ and $A_5$.
\end{thm}

\section{Concluding remarks}
In this paper, we have investigated properties of symmetries of oriented matroids.
It was shown that some fixed point properties of geometric rotational symmetries can naturally be extended to the setting of oriented matroids.
We classified full and rotational symmetry groups of simple oriented matroids of rank $3$.
Also, we made classification of FPA rotational symmetry groups of simple acyclic oriented matroids of rank $4$.
This shows that the classical approach to the classification of finite subgroups of $SO(3)$ can be followed only by using the axioms of oriented matroids
and the FPA property.
As future work, it may be interesting to investigate the following topics.
\begin{itemize}
\item Does the  symmetry group of every simple acyclic oriented matroid has FPA property?
\item Classification of FPA full symmetry groups of simple acyclic oriented matroids of rank $4$.
\item Classification of rotational and full symmetry groups of simple oriented matroids of rank $4$.
\item Classification of rotational and symmetry groups of simple acyclic oriented matroids of higher rank.
\item Is there any simple acyclic oriented matroid of rank $r$ whose symmetry group is not isomorphic to any finite subgroups of the orthogonal group $O(r-1)$?
\end{itemize}

\section*{Acknowlegement}
The author would like to thank an anonymous referee for many helpful comments,
especially for pointing out an error of Theorem 5.5 in the previous version.
This research is partially supported by JSPS Grant-in-Aid for Young Scientists (B) 26730002.

\section*{Appendix 1: properties concerning alternating matroids of rank $3$}
Here, we give proofs of the propositions concerning alternating matroids of rank $3$ that were not proved above.
These should be folklore, but the author could not find appropriate references.
\begin{prop}
\label{prop:alter_ordering}
Let ${\cal M}=(E,\{ \chi, -\chi \})$ be a simple oriented matroid of rank $3$, $X \subseteq E$ and $p:=|X|$.
Suppose that ${\cal M}|_X \simeq A_{3,p}$ and that
$x_0,\dots,x_{p-1} \in X$ are such that
$\chi(x_i,x_j,x_k)=+$ for all $i,j,k \in \mathbb{Z}$ with $0 \leq i < j < k \leq p-1$ (take the negative of $\chi$ if necessary).
For $i=0,\dots,p-1$, let $V_i$ be the cocircuit of ${\cal M}$ such that
\begin{align*} 
V_i(x_i)=V_i(x_{i+1})=0, \text{$V_i(x)=+$ for all $x \in X \setminus \{ x_i,x_{i+1}\}$,}
\end{align*}
where $x_p:=x_0$.
Let us consider $e \in E \setminus X$.
Then, the sign sequence $V_0(e),V_1(e),\dots,V_{p-1}(e)$ must be one of the following forms:
\[ +\dots +- \dots -+ \dots +, \ +\dots +0- \dots-+ \dots +, \ +\dots +- \dots-0+ \dots +, \ +\dots +0- \dots-0+ \dots +,\]
\[ -\dots -+ \dots +- \dots -, \ -\dots -0+ \dots+- \dots -, \ -\dots -+ \dots+0- \dots -, \ -\dots -0+ \dots+0- \dots -,\]
where $+\dots +$ and $- \dots -$ may be empty.
\end{prop}
\begin{proof}
In this proof, we only consider the case where $V_i(e) \neq 0$ for all $i=0,\dots,p-1$.
The following discussion can be applied to the other case by a slight modification.

We prove by contradiction.
Let us assume that there exist $l_1, l_2, l_3 \in \mathbb{Z}$ with $0 \leq l_1 < l_2 < l_3 \leq p-2$ such that
\begin{align*} 
V_{l_1}(e)=+,V_{l_1+1}(e)=-, V_{l_2}(e)=-,V_{l_2+1}(e)=+,V_{l_3}(e)=+,V_{l_3+1}(e)=-.
\end{align*}
Let $W_1,W_2,W_3$ be the cocircuits obtained by applying vector elimination to $V_{l_1},V_{l_1+1}, e$
and $V_{l_2},V_{l_2+1}, e$ and $V_{l_3},V_{l_3+1}, e$, respectively.
Then, the cocircuit $W_i$ satisfies
\begin{align*} 
W_i(e)=W_i(x_{l_i+1})=0, \text{ $W_i(x)=+$ for all $x \in X \setminus \{ x_{l_i+1}\}$,}
\end{align*}
for $i=1,2,3$.
Apply vector elimination to $W_1,-W_3$ and $x_{l_2+1}$.
Then, we obtain a cocircuit $W_4$ such that
\begin{align*} 
W_4(e)=W_4(x_{l_2+1})=0, W_4(x_{l_1+1})=-,W_4(x_{l_3+1})=+,
\end{align*}
which is a contradiction (compare with $W_2$).
A contradiction also occurs if there exist $l'_1, l'_2, l'_3 \in \mathbb{Z}$ with
$0 \leq l'_1 < l'_2 < l'_3 \leq p-1$ such that
\begin{align*} 
V_{l'_1}(e)=-,V_{l'_1+1}(e)=+, V_{l'_2}(e)=+,V_{l'_2+1}(e)=-,V_{l'_3}(e)=-,V_{l'_3+1}(e)=+.
\end{align*}
This proves the proposition.
\end{proof}

\begin{prop}
\label{prop:rank_relabeling}
Every matroid polytope ${\cal M}=(E,\{ \chi, -\chi\})$ of rank $3$ is isomorphic to an alternating matroid.
\end{prop}
\begin{proof}
We proceed by induction on $|E| $ ($=: n$).
By the induction hypothesis, we have ${\cal M} \setminus \{ e \} \simeq A_{3,n-1}$ for $e \in E$.
Let us relabel the elements of $E$ by $1,\dots,n$ so that $e$ is relabeled by $n$ and $\chi (i,j,k) = +$ for all $i,j,k \in \mathbb{Z}$ with
$1 \leq i < j < k \leq n-1$ (take the negative of $\chi$ if necessary).
For $i=1,\dots,n-1$, let $V_i$ be the cocircuit of ${\cal M}$ such that
\begin{align*} 
V_i(i)=V_i(i+1)=0, \text{$V_i(x)=+$ for all $x \in [n-1] \setminus \{ i,i+1\}$,}
\end{align*}
where $(n-1) + 1$ is interpreted as $1$.
By Proposition \ref{prop:alter_ordering}, the sequence $V_1(n),V_2(n),\dots,V_{n-1}(n)$ is one of the following forms:
\[  +\dots +-\dots -+\dots +, \ -\dots -+\dots +-\dots -\]
(since ${\cal M}$ is a matroid polytope of rank $3$, the sequence does not contain $0$).
Since ${\cal M}$ is a matroid polytope, the sequence is neither of the form $+\dots +$ nor $-\dots -$ 
(If the sequence is $+\dots +$, then $e$ is inside ${\cal M} \setminus \{ e \}$. If it is $-\dots -$, then the all-positive vector is a vector of ${\cal M}$.
See also Proposition \ref{prop:FP_inside}).
Without loss of generality, we assume that the sequence is of the form $+\dots +-\dots -$, where $+\dots +$ and $-\dots -$ are not empty.
Suppose that $l_- :=|\{ i \in [n-1] \mid V_i(n)=-\}| \geq 2$ and $l_+ :=|\{ i \in [n-1] \mid V_i(n)=+\}| \geq 2$.
Let $m:= {\rm argmax}\{ i \in [n-1] \mid V_i(n)=+\}$.
Consider the cocircuits $W_1,W_2$ and $W_3$ obtained respectively by vector elimination with respect to $V_{m}$, $V_{m+1}$ 
and $n$, with respect to $V_{n-1}$, $V_1$ and $n$, and with respect to $-V_m,V_{m+1}$ and $n-1$.
Then, $W_1, W_2$ and $W_3$ satisfy
\begin{align*}
& W_1(m+1) = W_1(n) = 0, W_1(x)=+ \text{ for all $x \in [n-1] \setminus \{ m+1,n\}$}, \\
& W_2(n-1) = W_2(n) = 0, W_2(x)=+ \text{ for all $x \in [n-1] \setminus \{ n-1,n\}$.} \\
& W_3(m+1) = W_3(n-1) = 0, W_3(m)=+, W_3(n)=+.
\end{align*}
(Remark that the elements $m,m+1,n-1$ and $n$ are all different by the assumption.)
This leads to that ${\cal M}|_{\{ m, m+1, n-1, n\}}$ is not a matroid polytope, which is a contradiction.
Therefore, we have $l_+=1$ or $l_- = 1$.

Let us first consider the case $l_-= 1$, i.e., the case $m=n-2$.
For each $i,j \in [n-1]$ $(i \neq j)$, let $V_{i,j}$ be the cocircuit of ${\cal M}$
such that 
\[ V_{i,j} (i) = V_{i,j}(j) = 0 \text{ and } V_{i,j}(n) = +.\]
For $k=2,\dots,n-3$, the cocircuit $V_{n-1,k}$ is obtained by vector elimination of $V_{n-2}$ and $-V_{n-1}$ with respect to $k$.
Since $V_{n-2}(1)=+$,
we have 
\[ V_{n-1,k}(1)=+ \text{ for all $k=2,\dots,n-2$.}\]
This means that 
\[ \chi (k,n-1,n)= \chi (k,n-1,1)=+ \text{ for all $k=2,\dots,n-2$}.\]
For $k=1,\dots,n-4$, the cocircuit $V_{n-2,k}$ is obtained by vector elimination of $V_{n-1,k}$ and $V_{k}$ with respect to $n-2$.
Recall that $V_{n-3}(1)=V_{n-2}(1)=+$. Thus, we have 
\[ V_{n-2,k}(1)=+ \text{ for all $k=1,\dots,n-2$.}\]
This means that 
\[ \chi (k,n-2,n)= \chi (k,n-2,1)=+ \text{ for all $k=1,\dots,n-3$}. \]
Continuing this discussion,  we have, for any $l=2,\dots,n-2$,
\[ \chi (k,l,n)= \chi (k,l,1)=+ \text{ for all $k=1,\dots,l-1$}.\]
Therefore, it holds that
\[ \chi (i,j,k) = + \text{ for all $i,j,k \in \mathbb{Z}$ with $1 \leq i < j < k \leq n$.}\]
In the case $l_+=1$, we have ${\cal M} \simeq {_{-\{ n \}}A_{3,n}}$ by the same discussion as above.
This implies that ${\cal M}$ is not a matroid polytope (since $A_{3,n}$ has the circuit $C$ with $C^+=\{ 1,n-1\}$ and $C^-=\{ 2,n\}$, 
$_{-\{ n \}}A_{3,n}$ has the circuit ${_{- \{ n \}}C}$ with ${_{- \{ n \}}C}^+=\{ 1 \}$ and  ${_{- \{ n \}}C}^-=\{ 2,n-1,n\}$), which is a contradiction.
As a conclusion, only the case $l_-=1$ is possible and thus we always have ${\cal M} \simeq A_{3,n}$.
\end{proof}

\begin{prop}
\label{prop:devide}
Let $x_0,\dots,x_{p-1}$ be the elements of $A_{3,p}$ that form an alternating matroid in this order.
For a covector $V$ of $A_{3,p}$ with $V(x_0)=0$,
 the sign sequence $V(x_1),V(x_2),\dots,V(x_{p-1})$ must be one of the following forms:
\[ -\dots -+ \dots +, \ 0-\dots -+ \dots +, \ -\dots -0+ \dots +, \ -\dots -+ \dots +0,\]
\[ +\dots +- \dots -, \ 0+\dots +- \dots -, \ +\dots +0- \dots -, \ +\dots +- \dots -0,\]
where $+\dots +$ and $- \dots -$ may be empty.
\end{prop}
\begin{proof}
Let us assume that there exist $l_1,l_2,l_3 \in \mathbb{Z}$ with $0 \leq l_1 < l_2 < l_3 \leq p-1$ such that
$V(x_{l_1})=+,V(x_{l_2})=-,V(x_{l_3})=+$.
Let $W$ be the cocircuit of $A_{3,p}$ such that
\begin{align*}
W(x_0)=W(x_{l_2})=0, W(x_i)=+ \text{ for $i=1,\dots,l_2-1$}, W(x_j)=- \text{ for $j=l_2+1,\dots,p-1$}
\end{align*}
and consider vector elimination of $V,W$ and $x_{l_3}$.
Then we obtain a cocircuit $Z$ such that
\begin{align*} 
Z(x_0)=Z(x_{l_3})=0, Z(x_{l_1})=+,Z(x_{l_2})=-,
\end{align*}
a contradiction.
Similarly, we obtain a contradiction if 
there exist $l'_1, l'_2, l'_3 \in \mathbb{Z}$ with
$0 \leq l'_1 < l'_2 < l'_3 \leq p-1$ such that
$V(x_{l'_1})=-,V(x_{l'_2})=+,V(x_{l'_3})=-$.
\end{proof}

\section*{Appendix 2: Proof of Lemma \ref{lem:ordering}}
\textbf{Lemma \ref{lem:ordering}}
Let ${\cal M}$ be a simple acyclic oriented matroid of rank $4$ on a ground set $E$, $R_f({\cal M})$ a FPA subgroup of $R({\cal M})$,
and $G \subseteq R({\cal M})$ a cyclic subgroup of $R({\cal M})$ of order $q>2$.
Consider rank $3$ $G$-orbits $O_1,\dots,O_m$ and set $F_1:={\rm span}_{\cal M}(O_1),\dots,F_m:={\rm span}_{\cal M}(O_m)$.
We assume that there is no duplication in this list by removing some of them if necessary.
Suppose $m \geq 2$.
Then, there exist exactly two permutations $p$ on $[m]$ that satisfy the following condition:
for each $i=1,\dots,m$, there exists a covector $X_i$ of ${\cal M}$ such that
\begin{align*}
\begin{split}
&\text{$X_i(e)=0$ for all $e \in F_{p(i)}$,} \\
&\text{$X_i(f)=-$ for all $f \in F_{p(1)} \cup \dots \cup F_{p(i-1)}$,} \\
&\text{$X_i(g)=+$ for all $g \in F_{p(i+1)} \cup \dots \cup F_{p(m)}$.} \\
\end{split}
\end{align*}
\begin{proof}
\begin{quote}
{\bf Claim 1.}
Let $S, T \subseteq E$ be rank 3 $G$-orbits with ${\rm span}_{\cal M}(S) \neq {\rm span}_{\cal M}(T)$.
For any cocircuit $X$ of ${\cal M}$ such that 
$X(s)=0$ for all $s \in {\rm span}_{\cal M}(S)$, it holds that
\begin{align*}
 \text{$X(t)=+$ for all $t \in T$ or $X(t)=-$ for all $t \in T$.}
\end{align*}
\end{quote}
{\scshape Proof of the claim:}
First, assume that there exist $t_1,t_2 \in T$ with $X(t_1)=+,X(t_2)=-$.
Let $\sigma \in G$ be a rotational symmetry such that $\sigma (t_1)= t_2$.
By applying vector elimination to $X,\sigma (X)$ and $t_2$, we obtain a covector $Y$ such that
$Y(t)=0$ for all $t \in S \cup \{ t_2 \}$. Note that if $Y=0$, $\sigma (X) = -X$ must hold.
If this holds, the symmetry $\sigma|_{E \setminus S}$ is a reflection symmetry of ${\cal M}/S$ (see Section \ref{sec:useful}).
Since $\sigma$ and $\sigma|_{S}$ are rotational symmetries of ${\cal M}$ and ${\cal M}|_{S}$ respectively (recall Corollary \ref{cor:rank3_restriction}), it is impossible. 
Hence, we have $Y \neq 0$, which contradicts to the fact $\rank(S \cup \{ t_2 \}) > \rank(S)=3$.
Therefore, it holds that $X(t)=+$ for all $t \in T$ or that $X(t)=-$ for all $t \in T$.

Next, let us assume that there exists $t^* \in T$ such that $X(t^*)=0$.
If $X(t)=0$ for all $t \in T$, it contradicts to the assumption that ${\rm span}_{\cal M}(S) \neq {\rm span}_{\cal M}(T)$.
A contradiction also occurs if there exists $t^{**} \in T$ such that $X(t^{**}) \neq 0$.
Indeed, take $\sigma \in G$ such that $\sigma (t^*) = t^{**}$ and consider the covector $Y:=-\sigma(X) \circ X$.
Then, there exist $t_1,t_2 \in T$ with  $Y(t_1) = +, Y(t_2)=-$.
This is a contradiction (recall  the first part of the proof).
 \quad $\Box$\medskip
\\

\begin{quote}
{\bf Claim 2.}
Let $S_1,\dots,S_m, T \subseteq E$ be $G$-orbits.
Suppose that $\rank(T)=3$ and that
there are covectors $X$ and $Y$ of ${\cal M}$ such that
\begin{align*}
\begin{split}
 &\text{$X(s_i)=\sigma_i$ for all $s_i \in S_i$, for $i=1,\dots,m$,}\\
 &\text{$X(t)=-$ for all $t \in T$}
 \end{split} 
\end{align*}
and
\begin{align*} 
\begin{split}
&\text{$Y(s_i)=\sigma'_i$ for all $s_i \in S_i$ for $i=1,\dots,m$,} \\
&\text{$Y(t)=+$ for all $t \in T$,}
\end{split}
\end{align*}
where $\sigma_i, \sigma'_i \in \{ +,-,0\}$ are such that
$\sigma_i \cdot \sigma'_i \geq 0$, for $i=1,\dots,m$.
For any $\sigma,\sigma' \in \{ +,-,0\}$,
let $\sigma \circ \sigma' := \sigma$ if $\sigma \neq 0$, $\sigma'$ otherwise.
Then, there is a covector $W^*$ of ${\cal M}$ such that
\begin{align*}  
\begin{split}
&\text{$W^*(s_i)=\sigma_i \circ \sigma'_i$ for all $s_i \in S_i$, for $i=1,\dots,m$,} \\
&\text{$W^*(t)=0$ for all $t \in T$.}
\end{split}
\end{align*}
\end{quote}
{\scshape Proof of the claim:}
First of all, note that for any $W \in {\cal V}^*$ we have $W|_T=0$ if $|(W|_T)^0| \geq 3$
(recall that ${\cal M}|_T \simeq A_{3,|T|}$).

Apply conformal elimination to $X,Y$ and $T$, and obtain a covector $W_1$ such that
\begin{align*}  
\begin{split}
&\text{$W_1(s_i)=\sigma_i \circ \sigma'_i$ for all $s_i \in S_i$, for $i=1,\dots,m$,} \\
&\text{$W_1(t) \leq 0$  for all $t \in T$ and $W_1(t_1)=0$ for some $t_1 \in T$}
\end{split}
\end{align*}
Similarly, we obtain a covector $W_2$ such that
\begin{align*} 
\begin{split}
&\text{$W_2(s_i)=\sigma_i \circ \sigma'_i$ for all $s_i \in S_i$, for $i=1,\dots,m$,} \\
&\text{$W_2(t)  \geq 0$ for all $t \in T$ and $W_2(t_2)=0$ for some $t_2 \in T$.}
\end{split}
\end{align*}
If $|(W_1|_T)^0| \geq 3$ (resp.\  $|(W_2|_T)^0| \geq 3$), then $W^*:=W_1$ (resp.\ $W_2$) is a required covector.
In the following, we assume $|(W_1|_T)^0| < 3$ and  $|(W_2|_T)^0| < 3$.

If  $|(W_1|_T)^0| = 1$, take $\tau \in G$ such that $\tau (t_2) = t_1$, and 
apply conformal elimination to $W_1,\tau (W_2)$ and $W_1^- \cap (\tau (W_2))^+$. Then we obtain a new covector $W_3$ such that
\begin{align*}
\begin{split}
&\text{$W_3(s_i)=\sigma_i \circ \sigma'_i$ for all $s_i \in S_i$, for $i=1,\dots,m$,} \\
&\text{$W_3(t_1)=0$, $W_3(t_3)=0$ for some $t_3 \in T \setminus \{ t_1\}$ and $W_3(t)  \leq 0$ for all $t \in T \setminus \{ t_1,t_3\}$.}
\end{split}
\end{align*}
If $|(W_3|_T)^0| \geq 3$, we have $W_3|_T=0$ and thus $W^*:=W_3$ is a required covector.
Otherwise, it holds that $(W_3|_T)^0=\{ t_1,t_3\}$ and that $(W_3|_T)^-=T \setminus \{ t_1,t_3\}$.

If $|(W_1|_T)^0| = 2$, let $W_3:=W_1$ and then the covector $W_3$ fulfills the same condition (under the convention that $\{ t_1, t_3 \} := (W_3)^0$).

Similarly, without loss of generality, we can assume the existence of a covector $W_4$ such that
$(W_4|_T)^0=\{ t_1, t_4\}$ and $(W_4|_T)^+=T \setminus \{ t_1,t_4\}$ for some $t_4 \in T$.
Since ${\cal M}|_T \simeq A_{3,|T|}$ and since $\{ t_1, t_3\}$ and $\{ t_1, t_4\}$ are its facets,
there exists $\pi \in G$ such that $\pi(t_1) = t_4$ and $\pi (t_3) =t_1$.
Take $t_5 \in T \setminus \{ t_1,t_4\}$.
Applying vector elimination to $\pi(W_3),W_4$ and $t_5$, 
we obtain a covector, which satisfies the required condition of $W^*$.
\quad $\Box$\medskip

\begin{quote}
{\bf Claim 3.}
Let $S,T \subseteq E$ be rank $3$ $G$-orbits with $\rank(S \cup T)=4$.
Suppose that there exists a covector $X$ such that 
$X(s)=0$ for all $s \in S$ and
$X(t)= \sigma$ for all $t \in T$, where $\sigma \in \{ +,- \}$.
Then, it holds that
$X(t') = \sigma$ for all $t' \in {\rm span}_{\cal M}(T)$.
 
\end{quote}
{\scshape Proof of the claim:}
Suppose that there exists $t_0 \in {\rm span}_{\cal M}(T)$ such that $X(t_0)= -\sigma$.
Let $T'$ be the $G$-orbit of $t_0$.
Note that the covector $X$ satisfies
\begin{align*} 
\begin{split}
&\text{$X(s)= 0$ for all $s \in S$,} \\
&\text{$X(t)= \sigma$ for all $t \in T$,} \\
&\text{$X(t')= -\sigma$ for all $t' \in T'$}.
\end{split}
\end{align*}
Also, note that there is a covector $Y$ such that
\begin{align*} 
\begin{split}
&\text{$Y(s)= \sigma$ for all $s \in S$,} \\
&\text{$Y(t)= 0$ for all $t \in T$,} \\
&\text{$Y(t')= 0$ for all $t' \in T'$}.
\end{split}
\end{align*}
Since ${\cal M}$ is acyclic, there is the covector $Z$ such that $Z(e)=-\sigma$ for all $e \in E$.
By applying Claim 2 to $Y$ and $Z$, we obtain a covector $W$ such that
\begin{align*} 
\begin{split}
&\text{$W(s)= 0$ for all $s \in S$,} \\
&\text{$W(t)= -\sigma$ for all $t \in T$,} \\
&\text{$W(t')= -\sigma$ for all $t' \in T'$}.
\end{split}
\end{align*}
Applying Claim 2 to $X$ and $W$, we obtain a non-zero covector $W^*$ such that
$W^*(e) = 0$ for all $e \in S \cup T$.
This contradicts to the assumption $\rank(S \cup T)=4$.
\quad $\Box$\medskip
\\ 
\\
Using Claim 3, we obtain modified versions of Claims 1 and 2 as follows. 
\begin{quote}
{\bf Claim 1'.}
Let $S, T \subseteq E$ be rank 3 $G$-orbits with ${\rm span}_{\cal M}(S) \neq {\rm span}_{\cal M}(T)$.
For any cocircuit $X$ of ${\cal M}$ such that 
$X(s)=0$ for all $s \in {\rm span}_{\cal M}(S)$, it holds that
\begin{align*}
 \text{$X(t)=+$ for all $t \in {\rm span}_{\cal M}(T)$ or $X(t)=-$ for all $t \in {\rm span}_{\cal M}(T)$.}
\end{align*}
\end{quote}
\begin{quote}
{\bf Claim 2'.}
Let $F_1,\dots,F_m, T \subseteq E$ be the flats spanned by $G$-orbits.
Suppose that $\rank(T)=3$ and
there are covectors $X$ and $Y$ of ${\cal M}$ such that
\begin{align*}
\begin{split}
 &\text{$X(s_i)=\sigma_i$ for all $s_i \in F_i$, for $i=1,\dots,m$,}\\
 &\text{$X(t)=-$ for all $t \in T$}
 \end{split} 
\end{align*}
and
\begin{align*} 
\begin{split}
&\text{$Y(s_i)=\sigma'_i$ for all $s_i \in F_i$ for $i=1,\dots,m$,} \\
&\text{$Y(t)=+$ for all $t \in T$,}
\end{split}
\end{align*}
where $\sigma_i, \sigma'_i \in \{ +,-,0\}$ are such that
$\sigma_i \cdot \sigma'_i  \geq 0$, for $i=1,\dots,m$.
For any $\sigma,\sigma' \in \{ +,-,0\}$,
let $\sigma \circ \sigma' := \sigma$ if $\sigma \neq 0$, $\sigma'$ otherwise.
Then, there is a covector $W^*$ of ${\cal M}$ such that
\begin{align*}  
\begin{split}
&\text{$W^*(s_i)=\sigma_i \circ \sigma'_i$ for all $s_i \in F_i$, for $i=1,\dots,m$,} \\
&\text{$W^*(t)=0$ for all $t \in T$.}
\end{split}
\end{align*}
\end{quote}

Now we prove Lemma \ref{lem:ordering}.
Since the lemma trivially holds for $m=2$, we assume $m \geq 3$.
Let us first prove that there are at most two permutations that satifies the condition.
If there are three permutations satisfying the condition,
there exists $i,j,k \in [m]$ and covectors $X$ and $Y$
such that
\begin{align*} 
\begin{split}
&\text{$X(s)=Y(s)= 0$ for all $s \in F_i$,} \\
&\text{$X(t)=Y(t) \neq 0$ for all $t \in F_j$,} \\
&\text{$X(u)= -Y(u) \neq 0$ for all $u \in F_k$}.
\end{split}
\end{align*}
Applying Claim 2', we obtain a non-zero covector $Z$ with $\rank(Z^0) =4$, which is a contradiction.

Next, we prove that there are at least two permutations that satifies the condition.
Take the flat $F$ of a rank 3 $G$-orbit.
Applying Claim 2' repeatedly to $F$ and the positive covector (resp.\ negative covector) and taking the negative of the resulting covector if necessary,
we obtain covectors $X_0$ and $Y_0$ such that
for some $i,j$ $(i \neq j) \in [m]$, 
\begin{align*} 
\begin{split}
&\text{$X_0(s)= 0$ for all $s \in F_i$,} \\
&\text{$X_0(t)=+$ for all $t \in \bigcup_{k \in [m] \setminus \{ i\}}{F_k}$,} 
\end{split}
\end{align*}
\begin{align*} 
\begin{split}
&\text{$Y_0(s)= 0$ for all $s \in F_j$,} \\
&\text{$Y_0(t)=+$ for all $t \in \bigcup_{k \in [m] \setminus \{ j\}}{F_k}$,} 
\end{split}
\end{align*}
Apply Claim 2' to $X_0$ and the negative covector with respect to $F_1$ and then obtain a covector $X_1$ 
such that
\begin{align*}
\begin{split}
&\text{$X_0^{(1)}(e)=0$ for all $e \in F_{i_0}$,} \\
&\text{$X_0^{(1)}(f)=-$ for all $f \in \bigcup_{k \in N_0^{(1)}}{F_k}$,} \\
&\text{$X_0^{(1)}(g)=+$ for all $g \in \bigcup_{k \in P_0^{(1)}}{F_k}$,} \\
\end{split}
\end{align*}
where $i_0 \in [m]$ and $P_0^{(1)},N_0^{(1)} \subseteq [m]$ are such that $P_0^{(1)} \cup N_0^{(1)} = [m] \setminus \{ i_0 \}$ 
and $P_0^{(1)} \cap N_0^{(1)} = \emptyset$. 
We continue applying Claim 2' with respect to obtained covectors and $X_0$ until we obtain a covector $X_1$ such that
\begin{align*}
\begin{split}
&\text{$X_1(e)=0$ for all $e \in F_{i_1}$,} \\
&\text{$X_1(f)=-$ for all $f \in \bigcup_{k \in N_1}{F_k}$,} \\
&\text{$X_1(g)=+$ for all $g \in \bigcup_{k \in P_1}{F_k}$,} \\
\end{split}
\end{align*}
where $i_1 \in [m]$ and $P_1,N_1 \subseteq [m]$ are such that $|N_1|=1$, $P_1 \cup N_1 = [m] \setminus \{ i_1 \}$ and $P_1 \cap N_1 = \emptyset$.
We continue this procedure and obtain a permutation $p_1$ on $[m]$: $p_1(1)=i_0,p_1(2)=i_1,\dots$.
Similarly, a permutation $p_2(\neq p_1)$ is obtained by applying the above procedure to $Y_0$.
The permutations $p_1$ and $p_2$ satisfy the required condition.
\end{proof}

\end{document}